\documentclass[11pt]{amsart}
\usepackage[a4paper, margin=3cm]{geometry} % page settings
\usepackage{amsmath,amssymb,amscd,amsthm,amsxtra}
\usepackage{dsfont}
\usepackage{graphicx, mathrsfs}
\usepackage{amssymb,amsmath,amsthm}
\usepackage{mathrsfs,dsfont}
\usepackage{fancyhdr}
\usepackage{enumerate}
\usepackage{hyperref}
\usepackage[toc,page]{appendix}
\usepackage{color}
\usepackage{caption}
\usepackage{subcaption}

\usepackage{mathtools}
\usepackage{float}
\usepackage{tikz}
\usetikzlibrary{patterns}
\usepackage{quiver}
\newtheorem{thm}{Theorem}[section]
\newtheorem{prop}[thm]{Proposition}
\newtheorem{lem}[thm]{Lemma}
\newtheorem{cor}[thm]{Corollary}
 
\theoremstyle{definition}
\newtheorem{definition}[thm]{Definition}

\newtheorem{assumption}[thm]{Assumption}
\theoremstyle{remark}

\newcommand{\bfj}{{\bf j}}

\newcommand{\bfi}{{\bf i}}

\newtheorem{remark}[thm]{Remark}
\numberwithin{equation}{section}

\begin{document}
	
	%%
	%% The title of the paper goes here.  Edit to your title.
	%%
	
	\title{Analysis of a numerical scheme for 3-wave kinetic equations}
	
	%%
	%% Now edit the following to give your name and address:
	%% 
	
	\author[M.-B. Tran]{Minh-Binh Tran}
	\address{Department of Mathematics, Texas A\&M University, College Station, TX 77843, USA}
	\email{minhbinh@tamu.edu} 
	\thanks{B. W and M.-B. T are  funded in part by  the  NSF Grants DMS-2204795, DMS-2305523,      NSF CAREER  DMS-2303146, DMS-2306379.}
	
	\author[B. Wang]{Bangjie Wang}
	\address{Department of Mathematics, Texas A\&M University, College Station, TX 77843, USA}
	\email{bangjiewang@tamu.edu}

	\begin{abstract}
		
		Several numerical schemes for  3-wave kinetic equations have been proposed in recent work and shown to be accurate and computationally efficient \cite{das2024numerical,walton2023numerical,walton2024numerical,walton2022deep}. However, their rigorous numerical analysis remains open. This paper aims to close this gap. We establish a comprehensive well-posedness and qualitative theory for the
		discrete equation arising from those schemes. We prove
		global existence, uniqueness, and Lipschitz stability of nonnegative classical solutions
		in $\ell^1(\mathbb{N})$, together with uniform bounds and decay of moments. We further
		show exponential energy decay and a sharp creation and propagation of positivity
		characterized by the arithmetic structure of the initial support. Finally, we obtain
		the propagation and instantaneous creation of polynomial, Mittag-Leffler, and
		exponential moments, providing quantitative control of high energy tails. We validate the theoretical findings by numerical results.
	\end{abstract}
	
	\maketitle
	\tableofcontents
	\label{contents}
	\section{Introduction}
	Over the past six decades, wave turbulence theory has emerged as a fundamental
	framework for describing the statistical behavior of weakly nonlinear wave
	systems across a wide range of physical settings. Prominent examples include
	inertial waves in rotating fluids, Alfv\'en wave turbulence in the solar wind,
	and wave dynamics in magnetized plasmas relevant to fusion devices, among many
	others. The foundations of the theory trace back to the seminal work of Peierls
	\cite{Peierls:1993:BRK}, with major subsequent developments by Benney and
	Saffman \cite{benney1966nonlinear}, Zakharov and Falkovich
	\cite{zakharov1967weak}, Benney and Newell \cite{benney1969random}, and, in
	particular, Hasselmann
	\cite{hasselmann1962non,hasselmann1974spectral}. These contributions ultimately
	led to the formulation of the classical 3-wave and 4-wave kinetic
	equations, which govern the redistribution of energy among weakly interacting
	waves. Rigorous derivations of wave kinetic equations have been achieved in a series of recent breakthrough works by Deng and Hani
	\cite{deng2019derivation,deng2021propagation,deng2023long,deng2023,deng2021full}.
	For a comprehensive
	physical discussion of wave turbulence theory and its applications, we refer to
	\cite{Nazarenko:2011:WT,PomeauBinh,zakharov2012kolmogorov}.
	
	In the 3-wave case, the wave kinetic equation takes the general form
	\begin{equation}\label{WT1}
		\begin{aligned}
			\partial_t f(t,p)
			= \iint_{\mathbb{R}^{2d}}
			\Big[
			R_{p,p_1,p_2}[f]
			- R_{p_1,p,p_2}[f]
			- R_{p_2,p,p_1}[f]
			\Big]
			\,\mathrm{d}p_1\,\mathrm{d}p_2,
			\qquad
			f(0,p)=f_0(p),
		\end{aligned}
	\end{equation}
	where \(f(t,p)\) denotes the wave density at wavenumber
	\(p\in\mathbb{R}^d\), with spatial dimension \(d\ge 2\), and \(f_0\) is the
	prescribed initial distribution. The interaction kernel is given by
	\begin{equation}\label{WT2}
		R_{p,p_1,p_2}[f]
		:=
		|V_{p,p_1,p_2}|^2
		\delta(p-p_1-p_2)
		\delta(\Omega-\Omega_1-\Omega_2)
		\big(f_1 f_2 - f f_1 - f f_2\big),
	\end{equation}
	where we use the shorthand notations
	\[
	f=f(t,p), \qquad
	f_j=f(t,p_j), \qquad
	\Omega=\Omega(p), \qquad
	\Omega_j=\Omega(p_j),
	\]
	for \(j\in\{1,2\}\).
	Here, \(\omega(p)\) is the dispersion relation of the underlying wave system,
	and \(V_{p,p_1,p_2}\) denotes the interaction coefficient. The theoretical analysis of $3$--wave kinetic equations has been carried out in a wide variety of physical settings, including phonon interactions in anharmonic crystal lattices \cite{CraciunBinh,EscobedoBinh,GambaSmithBinh,tran2020reaction}, capillary wave dynamics \cite{nguyen2017quantum}, beam wave propagation \cite{rumpf2021wave}, stratified oceanic flows \cite{GambaSmithBinh,kim2025wave}, and Bose--Einstein condensates \cite{AlonsoGambaBinh,cortes2020system,EPV,escobedo2023linearized1,escobedo2023linearized,ToanBinh,nguyen2017quantum,soffer2018dynamics,staffilani2025evolution,staffilani2025finite,staffilani2025formation}.
	
	In a recent series of works, we introduced  several numerical schemes specifically designed for the isotropic 3-wave kinetic equations, with a focus on both accuracy and computational efficiency
	\cite{das2024numerical,walton2023numerical,walton2024numerical,walton2022deep}. Nevertheless, a rigorous theoretical analysis of these schemes remains absent. This paper aims to address this deficiency. 
	In particular, the scheme developed in \cite{das2024numerical} is based on the nonlinear coagulation--fragmentation approximation of the 3-wave kinetic equation \eqref{WT1}--\eqref{WT2} proposed by Connaughton and Newell in \cite{connaughton2010dynamical}.
	The coagulation--fragmentation model considered in \cite{das2024numerical}, which serves as an effective approximation of the 3-wave kinetic dynamics, takes the following form:
	
	\begin{equation}\label{coagufrag}
		\frac{\partial f_\omega}{\partial t}
		=\mathcal{O}[f_\omega](t),
		\qquad
		f_\omega(0)=f_0(\omega),
		\qquad
		\omega\in\mathbb{R}_+,
	\end{equation}
	where the operator \(\mathcal{O}\) is defined by
	\[
	\mathcal{O}[f_\omega](t)
	= S_1[f_\omega]-S_2[f_\omega]-S_3[f_\omega]-V[f_\omega],
	\]
	with
	\[
	\begin{aligned}
		S_1[f_\omega]
		= {} & \int_0^\omega K_1(\omega-\mu,\mu)\,f_{\omega-\mu}f_\mu\,\mathrm d\mu
		- \int_0^\infty K_1(\mu,\omega)\,f_\omega f_\mu\,\mathrm d\mu
		- \int_0^\infty K_1(\omega,\mu)\,f_\omega f_\mu\,\mathrm d\mu, \\[0.4em]
		S_2[f_\omega]
		= {} & -\int_0^\omega K_2(\mu,\omega-\mu)\,f_\omega f_\mu\,\mathrm d\mu
		+ \int_\omega^\infty K_2(\omega,\mu-\omega)\,f_{\mu-\omega}f_\mu\,\mathrm d\mu
		+ \int_0^\infty K_2(\omega,\mu)\,f_\omega f_{\omega+\mu}\,\mathrm d\mu, \\[0.4em]
		S_3[f_\omega]
		= {} & -\int_0^\omega K_3(\mu,\omega-\mu)\,f_\omega f_{\omega-\mu}\,\mathrm d\mu
		+ \int_\omega^\infty K_3(\omega,\mu-\omega)\,f_\omega f_\mu\,\mathrm d\mu
		+ \int_0^\infty K_3(\omega,\mu)\,f_\mu f_{\omega+\mu}\,\mathrm d\mu, \\[0.4em]
		V[f_\omega]
		= {} & \gamma(\omega)\,f_\omega .
	\end{aligned}
	\]

	In particular, the operators \(S_2\) and \(S_3\) share the same structural form, differing only through their respective kernels. The term $V[f_\omega]$ represents the viscous dissipation associated with the underlying wave system \cite{connaughton2010dynamical}. We also note that a recent work~\cite{banks2025new} proposes a novel approach to the direct discretization of 3-wave kinetic equations, with applications to a 2-dimensional nonlinear Schr\"odinger system.

	In this work, we develop a comprehensive well-posedness and qualitative theory for the
	discrete equation arising from the finite-volume scheme associated with
	\eqref{coagufrag}. 	 We assume that the kernels \(K_1,K_2,K_3\) are symmetric (see  \cite{connaughton2010dynamical,das2024numerical}) 
	\[
	K_n(\omega,\mu)=K_n(\mu,\omega)
	=\omega^{\alpha_n}\mu^{\alpha_n}(\omega+\mu)^{\beta_n},
	\qquad n\in\{1,2,3\},
	\]
	and
	\[
	\gamma(\omega)=\omega^\delta,
	\]
	where \(\alpha_n,\beta_n,\delta\ge 0\) are fixed real parameters.
	
	Under this symmetry assumption, the operators \(S_1,S_2,S_3\) can be rewritten as
	\[
	\begin{aligned}
		S_1[f_\omega]
		= {} & \int_0^\omega K_1(\omega-\mu,\mu)\,f_{\omega-\mu}f_\mu\,\mathrm d\mu
		- 2\int_0^\infty K_1(\omega,\mu)\,f_\omega f_\mu\,\mathrm d\mu, \\[0.4em]
		S_2[f_\omega]
		= {} & -\int_0^\omega K_2(\omega-\mu,\mu)\,f_\omega f_\mu\,\mathrm d\mu
		+ \int_0^\infty K_2(\omega,\mu)\,f_\omega f_{\omega+\mu}\,\mathrm d\mu
		+ \int_0^\infty K_2(\omega,\mu)\,f_\mu f_{\omega+\mu}\,\mathrm d\mu, \\[0.4em]
		S_3[f_\omega]
		= {} & -\int_0^\omega K_3(\omega-\mu,\mu)\,f_\omega f_\mu\,\mathrm d\mu
		+ \int_0^\infty K_3(\omega,\mu)\,f_\omega f_{\omega+\mu}\,\mathrm d\mu
		+ \int_0^\infty K_3(\omega,\mu)\,f_\mu f_{\omega+\mu}\,\mathrm d\mu .
	\end{aligned}
	\] We establish global existence and uniqueness of nonnegative
	classical solutions in $\ell^1(\mathbb{N})$, together with uniform-in-time bounds on
	the first and higher-order moments. In addition, we prove Lipschitz continuous
	dependence of solutions on the initial data in the $\ell^1$ norm, ensuring stability
	of the discrete dynamics. The main new idea of the analysis is to identify a safe, forward-invariant region of the
	state space and to show that the dynamics never drives solutions outside this set.  
	This admissible region is characterized by three natural constraints: nonnegativity,
	bounded total mass, and bounded higher-order energy. To verify invariance, we first
	introduce suitable truncations so that the interaction operator is well behaved and
	analyze small forward time steps. Exploiting the structural properties of the equation,
	we show that  higher-order energy is controlled through a barrier mechanism: when
	the energy becomes large, the dynamics forces it to decrease. Finally, continuity
	properties of the operator allow us to remove the truncation, yielding a tangency
	condition of the vector field to the admissible set. This geometric property provides
	the key mechanism underlying global existence, uniqueness, and stability.

	We further show that the discrete energy decays exponentially in time at a rate determined by the discretization parameter, which implies convergence of solutions to zero in the long-time limit. In addition, we establish a sharp creation-and-propagation result for positivity: for any positive time, the solution is strictly positive exactly on the arithmetic lattice generated by the greatest common divisor of the support of the initial data and vanishes identically outside this set.
	The proof is based on a new, crucial observation: if the solution is positive at some indices, the gain terms immediately generate positivity at new indices obtained by combining them, while the loss terms cannot suppress this effect. Iterating this mechanism shows that positivity propagates to all indices reachable from the initially positive ones. A number-theoretic argument then identifies this reachable set as the sublattice determined by the greatest common divisor of the initial support. Finally, uniqueness of solutions implies that positivity cannot appear outside this sublattice, so the dynamics both creates and confines positivity in a precise manner. 
	
	We also investigate the evolution of moments. Polynomial moments of arbitrary order
	are propagated whenever they are initially finite, while higher-order polynomial
	moments are instantaneously created for any positive time. Beyond algebraic moments,
	we introduce Mittag-Leffler moments and prove their propagation under suitable
	structural conditions on the collision kernel. As a consequence, we obtain the creation
	of exponential-type moments, providing quantitative control of the high-energy tails
	of the solution. The key idea of the proof is to truncate the initial data so that only finitely many
	sizes are present. Positivity and nontriviality ensure that the total mass remains strictly positive, while mass conservation provides uniform in time bounds that are independent of the truncation. We then derive an inequality for a prescribed higher order moment, consisting of a bounded production term controlled by the mass and a strong damping term that dominates when the moment becomes large. We observe that the use of forward invariant regions in the proof of well-posedness naturally extends to the propagation of moments. Indeed, the propagation of a given moment is equivalent to the existence of a forward-invariant region within which the solution evolves. Classically, proving the propagation of moments relies on the analysis of differential inequalities \cite{GambaSmithBinh,kim2025wave}. Such approaches are typically more involved, as they require continuity of the moment in time in order to extend the interval of existence and iterate the argument. In contrast, our proof is based on showing that the operator admits a forward invariant region with a finite associated moment. This argument depends only on the structural properties of the operator, which significantly simplifies the analysis and yields a sharper result. Finally, we remove the truncation by exploiting stability with respect to the initial data together with a lower semicontinuity argument, thereby transferring the uniform bounds to the original solution.
	
	Finally, we validate our theoretical results through numerical
	experiments.

	\section{The Setting and main results}
	
	\subsection{The setting}

	The work \cite{das2024numerical} proposed a finite volume scheme for \eqref{coagufrag}. 
	To make the presentation self-contained, we detail in this section the derivation of the numerical scheme.
	We adopt the finite volume framework of \cite{das2024numerical}, adapting it to the unbounded domain $\mathbb{R}_{\ge 0}$ and restricting ourselves to uniform meshes.
	
	We discretize $\mathbb{R}_{\ge 0}$ using a uniform mesh consisting of intervals
	\[
	[(i-1)h,ih], \qquad i=1,2,\dots,
	\]
	where $h>0$ denotes the mesh size.
	Throughout the paper, we use $\bfi$ to denote the physical location $ih$ and $\bfj$ to denote $jh$.
	
	We define the cell averages
	\[
	f_i(t)=\int_{(i-1)h}^{ih} f_\omega(t)\,\mathrm d\omega,
	\qquad
	f_i(0)=\int_{(i-1)h}^{ih} f_0(\omega)\,\mathrm d\omega .
	\]
	
	Integrating the continuous equation over the cell $[(i-1)h,ih]$ yields
	\begin{align*}
		\frac{\partial f_i}{\partial t}
		&= \int_{(i-1)h}^{ih}\frac{\partial f_\omega}{\partial t}\,\mathrm d\omega
		= \int_{(i-1)h}^{ih}\mathcal O[f_\omega]\,\mathrm d\omega \\
		&= \int_{(i-1)h}^{ih}\int_0^\omega K_1(\omega-\mu,\mu) f_{\omega-\mu} f_\mu
		\,\mathrm d\mu\,\mathrm d\omega
		-2\int_{(i-1)h}^{ih}\int_0^\infty K_1(\omega,\mu) f_\omega f_\mu
		\,\mathrm d\mu\,\mathrm d\omega \\
		&\quad -\int_{(i-1)h}^{ih}\int_0^\omega
		\bigl[K_2(\omega-\mu,\mu)+K_3(\omega-\mu,\mu)\bigr]
		f_\omega f_\mu\,\mathrm d\mu\,\mathrm d\omega \\
		&\quad +\int_{(i-1)h}^{ih}\int_0^\infty
		\bigl[K_2(\omega,\mu)+K_3(\omega,\mu)\bigr]
		f_\omega f_{\omega+\mu}\,\mathrm d\mu\,\mathrm d\omega \\
		&\quad +\int_{(i-1)h}^{ih}\int_0^\infty
		\bigl[K_2(\omega,\mu)+K_3(\omega,\mu)\bigr]
		f_\mu f_{\omega+\mu}\,\mathrm d\mu\,\mathrm d\omega
		-\int_{(i-1)h}^{ih}\gamma(\omega)f_\omega\,\mathrm d\omega .
	\end{align*}
	
	Assuming that $f$ is piecewise constant on each mesh cell, we approximate, for example,
	\begin{align*}
		&\int_{(i-1)h}^{ih}\int_0^\omega
		K_1(\omega-\mu,\mu) f_{\omega-\mu} f_\mu
		\,\mathrm d\mu\,\mathrm d\omega \\
		&\approx
		\sum_{j=1}^{i-1}\int_{(i-1)h}^{ih}\int_{(j-1)h}^{jh}
		K_1(\omega-\mu,\mu) f_{\omega-\mu} f_\mu
		\,\mathrm d\mu\,\mathrm d\omega \\
		&\approx
		\sum_{j=1}^{i-1}
		K_1((i-j)h,jh)
		\int_{(i-1)h}^{ih}\int_{(j-1)h}^{jh}
		f_{\omega-\mu} f_\mu
		\,\mathrm d\mu\,\mathrm d\omega \\
		&\approx
		\sum_{j=1}^{i-1} K_1((i-j)h,jh)\, f_{i-j} f_j .
	\end{align*}
	Analogous approximations apply to the remaining terms.
	
	We therefore obtain the finite volume scheme
	\begin{equation}\label{eqn:discrete_pde}
		\frac{\partial f}{\partial t}
		= S^{(1)}[f]+S^{(2)}[f]+S^{(3)}[f]-V[f],
	\end{equation}
	where
	\[
	\begin{aligned}
		S_i^{(1)}[f]
		= {} & \sum_{j=1}^{i-1} K_1(\bfj,\bfi-\bfj) f_j f_{i-j}
		-2\sum_{j=1}^{\infty} K_1(\bfi,\bfj) f_i f_j, \\
		S_i^{(2)}[f]
		= {} & -\sum_{j=1}^{i-1} K_2(\bfi-\bfj,\bfj) f_i f_j
		+\sum_{j=i+1}^{\infty} K_2(\bfj,\bfi) f_i f_j
		+\sum_{j=i+1}^{\infty} K_2(\bfi,\bfj-\bfi) f_j f_{j-i}, \\
		S_i^{(3)}[f]
		= {} & -\sum_{j=1}^{i-1} K_3(\bfi-\bfj,\bfj) f_i f_j
		+\sum_{j=i+1}^{\infty} K_3(\bfj-\bfi,\bfi) f_i f_j
		+\sum_{j=i+1}^{\infty} K_3(\bfi,\bfj-\bfi) f_j f_{j-i}, \\
		V_i[f]
		= {} & \gamma(\bfi) f_i .
	\end{aligned}
	\]
	
	We assume that, for $n\in\{1,2,3\}$,
	\[
	K_n(\bfi,\bfj)
	=\bfi^{\alpha_n}\bfj^{\alpha_n}(\bfi+\bfj)^{\beta_n},
	\qquad
	\gamma(\bfi)=\bfi^\delta ,
	\]
	with positive parameters \(\{\alpha_n,\beta_n\}_{n\in\{1,2,3\}},\delta\) satisfying
	\begin{equation}\label{asp:para}		
		\delta>\max_{n\in\{1,2,3\}}(\alpha_n+\beta_n),\qquad
		\delta>2\alpha_1+\beta_1-1.
	\end{equation}

	Define
	\[
	S[f]:=S^{(1)}[f]+S^{(2)}[f]+S^{(3)}[f],
	\qquad
	O[f]:=S[f]-V[f].
	\]
	
	Finally, the operator $O[f]$ admits the gain--loss decomposition
	\begin{equation}\label{eqn:opt_gain_loss}
		O[f]=O^+[f]-O^-[f],
	\end{equation}
	where
	\begin{align*}
		O^+[f]
		= {} &
		\sum_{j=1}^{i-1} K_1(\bfj,\bfi-\bfj) f_j f_{i-j}
		+\sum_{j=i+1}^{\infty} K_2(\bfj-\bfi,\bfi) f_i f_j
		+\sum_{j=i+1}^{\infty} K_2(\bfi,\bfj-\bfi) f_j f_{j-i} \\
		&+
		\sum_{j=i+1}^{\infty} K_3(\bfj-\bfi,\bfi) f_i f_j
		+\sum_{j=i+1}^{\infty} K_3(\bfi,\bfj-\bfi) f_j f_{j-i}, \\
		O^-[f]
		= {} &
		2\sum_{j=1}^{\infty} K_1(\bfi,\bfj) f_i f_j
		+\sum_{j=1}^{i-1} K_2(\bfi-\bfj,\bfj) f_i f_j
		+\sum_{j=1}^{i-1} K_3(\bfi-\bfj,\bfj) f_i f_j
		+\gamma(\bfi) f_i.
	\end{align*}
	
	Throughout the paper, the notation \(A \lesssim B\) denotes the existence of a positive constant \(C\), independent of the relevant variables, such that
	\[
	A \le C B .
	\]
	When the constant \(C\) depends on parameters \(a,b,\dots\), we write \(A \lesssim_{a,b,\dots} B\).
	
	We also adopt the convention \(\mathbb{N}=\{1,2,3,\dots\}\).

	\subsection{Main results}
	
	{\it This section is devoted to the statement of the analytical results established for the numerical scheme \eqref{eqn:discrete_pde}. The proofs of those results will be provided in Subsections \ref{SecTh1}, \ref{SecTh2}, \ref{SecTh3}, \ref{SecTh4}, \ref{SecTh5}, \ref{SecTh6} and \ref{SecTh7}. Numerical confirmations of these results are presented later in Section~\ref{Sec:Numerical}.
	}

	\begin{definition}
		For a function \(f:\mathbb{N}\to\mathbb{R}_{\ge 0}\), we define its \(k\)-th moment by
		\[
		m_k\langle f\rangle := \sum_{i=1}^{\infty} f_i\, \bfi^{\,k},
		\qquad k\ge 0 .
		\]
		
	\end{definition}
	\begin{remark}\label{rmk:high_moment_est_low_moment}
		Note that for \(k_1<k_2\),
		\[
		\bfi^{k_1}
		= i^{k_1} h^{k_1}
		\le i^{k_2} h^{k_1}
		= h^{k_1-k_2}\,\bfi^{k_2}.
		\]
		Consequently, we obtain the natural estimate
		\[
		m_{k_1}\langle f\rangle
		\le h^{\,k_1-k_2}\, m_{k_2}\langle f\rangle,
		\qquad k_1<k_2.
		\]
		
		In particular, if \(m_k\langle f\rangle<\infty\) for some \(k\ge 0\), then
		\[
		m_{k'}\langle f\rangle<\infty
		\quad \text{for all } k'\le k .
		\]
		
	\end{remark}
	Throughout this section, we assume that the initial data \(f_0:\mathbb{N}\to\mathbb{R}\) satisfies the following assumption.
	\begin{assumption}\leavevmode\label{asmp_1}
		\begin{enumerate}
			\item \(f_0(i)\ge 0\) for all \(i\in\mathbb{N}\).
			\item There exists a fixed \(\kappa>\max\{\delta,1,2-\alpha_1-\beta_1\}\) such that
			\[
			m_\kappa\langle f_0\rangle<\infty .
			\]
			By Remark~\ref{rmk:high_moment_est_low_moment}, this condition also implies that
			\(m_1\langle f_0\rangle<\infty\).
		\end{enumerate}
		
	\end{assumption}
	
	Before stating the main results of the paper, we introduce a bound function
	\(\mathcal{B}_k:[0,\infty)\to[0,\infty)\) in order to simplify the presentation.
	For \(k>1\), we define
	\[
	\mathcal{B}_k(x)
	:=\left(
	2\mathcal{C}_k\,
	x^{\,1+\frac{\delta+k-1}{\delta-2\alpha_1-\beta_1+1}
		+\frac{\delta}{k-1}}
	\right)^{\frac{k-1}{\delta+k-1}},
	\]
	where \(\mathcal{C}_k\) is a positive constant depending on
	\(\alpha_1,\beta_1,\delta,\) and \(k\), whose explicit expression is given in
	\eqref{eqn:def_for_C_k} in the proof of Proposition~\ref{prop:moment_est_for_O}.
	
	With this notation, we can now state the main results concerning solutions to
	\eqref{eqn:discrete_pde}.
	\begin{thm}[Existence and Uniqueness]\label{thm:GWP}
		For initial data \(f_0\) satisfying Assumption~\ref{asmp_1}, the discrete equation \ref{eqn:discrete_pde} admits a unique classical solution
		\[
		f(t,\cdot)\in C^1\bigl([0,\infty);\ell^1(\mathbb{N})\bigr).
		\]
		Moreover, for all \(t\in[0,\infty)\), the solution satisfies
		\[
		f(t,i)\ge 0,
		\qquad
		m_1\langle f(t)\rangle \le m_1\langle f_0\rangle,
		\]
		and
		\[
		m_{\kappa}\langle f(t)\rangle
		\le
		\max\!\left\{
		m_{\kappa}\langle f_0\rangle,
		\; 2\,\mathcal{B}_{\kappa}\!\bigl(m_1\langle f_0\rangle\bigr)
		\right\}.
		\]
		
	\end{thm}
	
	\begin{thm}[Lipschitz Continuity]\label{thm:Lipschitz_dependence}
		Suppose that \(f\) and \(g\) are two solutions of \eqref{eqn:discrete_pde} obtained in
		Theorem~\ref{thm:GWP}, with initial data \(f_0\) and \(g_0\) satisfying
		Assumption~\ref{asmp_1}. Then, for all \(t\in[0,\infty)\),
		\[
		\|f(t)-g(t)\|_{\ell^1}
		\le
		\mathrm e^{Ct}\,\|f_0-g_0\|_{\ell^1},
		\]
		where the constant \(C>0\) depends only on
		\[
		\max\!\bigl\{m_1\langle f_0\rangle,\; m_1\langle g_0\rangle\bigr\}
		\quad\text{and}\quad
		\max\!\bigl\{m_{\kappa}\langle f_0\rangle,\; m_{\kappa}\langle g_0\rangle\bigr\}.
		\]
		
	\end{thm}
	\begin{thm}\label{thm:energy_decay}
		Let \(f(t)\) be the solution to \eqref{eqn:discrete_pde} with initial data \(f_0\) satisfying
		Assumption~\ref{asmp_1}. Suppose that the discretization mesh size is \(h>0\).
		Then, for all \(t\ge 0\),
		\[
		m_1\langle f(t)\rangle
		\le
		\mathrm e^{-h^{\delta} t}\, m_1\langle f_0\rangle .
		\]
		
	\end{thm}
	\begin{remark}
		It implies that, the solution is converging to 0 pointwisely in long time.
	\end{remark}
	\begin{thm}[Creation and Propagation of Positivity]\label{thm:creation_of_positivity}
		Let \(f(t)\) be the solution to \eqref{eqn:discrete_pde} with initial data \(f_0\)
		satisfying Assumption~\ref{asmp_1}. Define the index set
		\[
		I:=\{\,i\in\mathbb{N} : f_0(i)>0\,\},
		\]
		and let \(\gcd(I)\) denote the greatest common divisor of the elements of \(I\).
		Then, for all \(t>0\),
		\[
		\begin{cases}
			f(t,i)>0, & i\in \gcd(I)\,\mathbb{N},\\[0.3em]
			f(t,i)=0, & i\notin \gcd(I)\,\mathbb{N}.
		\end{cases}
		\]
		
	\end{thm}
	
	\begin{cor}[Propagation of Polynomial Moments]
		\label{Corr:Poly}	Let \(f(t)\) be the solution to \eqref{eqn:discrete_pde} with initial data \(f_0\)
		satisfying Assumption~\ref{asmp_1}. Suppose that, for a given \(k>\max\{\delta,1,2-\alpha_1-\beta_1\}\),
		\[
		m_k\langle f_0\rangle<\infty .
		\]
		Then, for all \(t\in[0,\infty)\),
		\[
		m_k\langle f(t)\rangle
		\le
		\max\!\left\{
		m_k\langle f_0\rangle,\;
		2\,\mathcal{B}_k\!\bigl(m_1\langle f_0\rangle\bigr)
		\right\}.
		\]
		
	\end{cor}
	\begin{proof}
		This follows directly from Theorem \ref{thm:GWP}.
	\end{proof}
	\begin{thm}[Creation of Polynomial Moments]\label{thm:creation_of_polynomial_moments}
		Let \(f(t)\) be the solution to \eqref{eqn:discrete_pde} with initial data \(f_0\)
		satisfying Assumption~\ref{asmp_1}. Then, for any \(k>1\) and for all \(t>0\),
		\begin{equation}\label{eqn:creation_of_polynomial_moments}
			m_k\langle f(t)\rangle
			\le
			\left(\frac{2(k-1)}{\delta\, t}\right)^{\frac{k-1}{\delta}}
			m_1\langle f_0\rangle
			+ \mathcal{B}_k\!\bigl(m_1\langle f_0\rangle\bigr).
		\end{equation}
		
	\end{thm}
	
	\begin{definition}[Mittag-Leffler Moments]
		
		For \(a\in[1,\infty)\), define the function
		\[
		\mathcal{E}_a(x)
		:= \sum_{k=1}^{\infty}\frac{x^k}{\Gamma(ak+1)} .
		\]
		
		The Mittag--Leffler moment of a function \(f:\mathbb{N}\to\mathbb{R}_{\ge 0}\)
		with parameters \(a\ge 1\) and \(\lambda>0\) is then defined by
		\[
		\mathcal{E}_a^\infty(\lambda)\langle f\rangle
		:= \sum_{i=1}^{\infty} f_i\,\mathcal{E}_a\!\left(\lambda^{a}\bfi\right)
		= \sum_{k=1}^{\infty}
		\frac{m_k\langle f\rangle\,\lambda^{ak}}{\Gamma(ak+1)} .
		\]
		
	\end{definition}
	\begin{remark}\leavevmode\label{rmk:estimate_for_m_l_function}
		\begin{enumerate}
			\item There exist positive constants \(c_a\) and \(C_a\) such that
			\[
			c_a\bigl(\mathrm e^{x^{1/a}}-1\bigr)
			\le \mathcal{E}_a(x)
			\le C_a\bigl(\mathrm e^{x^{1/a}}-1\bigr),
			\qquad x\ge 0 .
			\]
			Consequently,
			\[
			\mathcal{E}_a^\infty(\lambda)\langle f\rangle
			= \sum_{i=1}^{\infty} f_i\,\mathcal{E}_a\!\left(\lambda^a\bfi\right)
			\approx
			\sum_{i=1}^{\infty} f_i\,
			\bigl(\mathrm e^{\lambda\,\bfi^{1/a}}-1\bigr),
			\]
			where the equivalence holds up to multiplicative constants depending only on \(a\).
			
			\item When \(a=1\), we have \(\mathcal{E}_1(x)=\mathrm e^x-1\). Moreover,
			\[
			\mathrm e^{\lambda x}-1
			= \mathrm e^{\frac{\lambda}{2}x}
			\left(\mathrm e^{\frac{\lambda}{2}x}-\mathrm e^{-\frac{\lambda}{2}x}\right)
			= 2\,\mathrm e^{\frac{\lambda}{2}x}\,
			\sinh\!\left(\frac{\lambda}{2}x\right)
			\ge \lambda x\,\mathrm e^{\frac{\lambda}{2}x},
			\]
			for all \(x\ge 0\). As a consequence,
			\[
			\sum_{i=1}^{\infty} f_i(t)\,\lambda\bfi\,
			\mathrm e^{\frac{\lambda}{2}\bfi}
			\le
			\sum_{i=1}^{\infty} f_i(t)\,\mathcal{E}_1(\lambda\bfi)
			=
			\mathcal{E}_1^\infty(\lambda)\langle f\rangle .
			\]
		\end{enumerate}
		
	\end{remark}
	
	\begin{thm}[Propagation of Mittag-Leffler Moments]\label{thm:propagation_of_m_l_tails}
		Suppose that \(\delta>2\alpha_1+\beta_1\).
		Let \(f(t)\) be the solution to \eqref{eqn:discrete_pde} with initial data \(f_0\)
		satisfying Assumption~\ref{asmp_1}. Fix any \(a\in[1,\infty)\) and assume that
		\begin{equation}\label{eqn:con_for_propaga_of_m_l_tail}
			\lambda
			<
			\min\!\left\{
			\left(4C(\sqrt{a})^{\frac{\delta}{\delta-2\alpha_1-\beta_1}}\right)^{-\frac{1}{\delta a}},
			\;
			\left(4\,m_1\langle f_0\rangle\,\mathcal{E}_a(1)\right)^{-\frac{1}{a}}
			\right\},
		\end{equation}
		where the constant \(C>0\) depends only on \(\alpha_1,\beta_1,\delta\).
		
		If the initial data satisfies
		\[
		\sum_{i=1}^{\infty} f_0(i)\,\mathcal{E}_a\!\left(\lambda^a\bfi\right)\le 1,
		\]
		then the solution satisfies, for all \(t\in[0,\infty)\),
		\[
		\sum_{i=1}^{\infty} f_i(t)\,\mathcal{E}_a\!\left(\lambda^a\bfi\right)\le 1 .
		\]
	\end{thm}
	\begin{thm}[Creation of Exponential Moments]\label{thm:creation_of_exp_moments} Suppose that \(\delta>2\alpha_1+\beta_1\) and \(\delta\ge 1\).
		Let \(f(t)\) be the solution to \eqref{eqn:discrete_pde} with initial data \(f_0\)
		satisfying Assumption~\ref{asmp_1}. Then for
		\(\lambda>0\) depending only on
		\(m_1\langle f_0\rangle\), \(\alpha_1\), \(\beta_1\), and \(\delta\) small enough, the solution satisfies, for all \(t\in [0,1]\), 
		\[\sum_{i=1}^{\infty}
		f_i(t)\,\bfi\,
		e^{\lambda t^{1/\lfloor\delta\rfloor}\bfi}	
		\le \frac{1}{2\lambda},
		\]
		and for all \(t\in [1,\infty)\),
		\[\sum_{i=1}^{\infty}
		f_i(t)\,\bfi\,
		e^{\lambda \bfi}
		\le \frac{1}{2\lambda}.\]
	\end{thm}	
	\section{Moments estimates}
	
	\begin{prop}[Weak Formulation]
		For any functions \(f,\varphi:\mathbb{N}\to\mathbb{R}\) such that all terms on the
		right-hand side are well defined and finite, we have the weak formulation
		\begin{equation}\label{eqn:weak_formulation}
			\sum_{i=1}^{\infty} S_i[f]\,\varphi(i)
			=
			\sum_{i,j=1}^{\infty}
			\Bigl[
			K_1(\bfi,\bfj)\,f_i f_j
			-
			\bigl(K_2(\bfi,\bfj)+K_3(\bfi,\bfj)\bigr)\,f_i f_{i+j}
			\Bigr]
			\bigl(\varphi(i+j)-\varphi(i)-\varphi(j)\bigr).
		\end{equation}
		
	\end{prop}
	\begin{proof}
		
		Assume that all terms appearing in the following computations are finite.
		Then, by Fubini’s theorem, we obtain
		\begin{align*}
			\sum_{i=1}^{\infty} S_i^{(1)}[f]\,\varphi(i)
			= {} &
			\sum_{i=1}^{\infty}\sum_{j=1}^{i-1}
			K_1(\bfj,\bfi-\bfj)\,f_j f_{i-j}\,\varphi(i)
			-2\sum_{i=1}^{\infty}\sum_{j=1}^{\infty}
			K_1(\bfi,\bfj)\,f_i f_j\,\varphi(i) \\
			= {} &
			\sum_{i,j=1}^{\infty}
			K_1(\bfi,\bfj)\,f_i f_j\,\varphi(i+j)
			-\sum_{i,j=1}^{\infty}
			K_1(\bfi,\bfj)\,f_i f_j\,[\varphi(i)+\varphi(j)] \\
			= {} &
			\sum_{i,j=1}^{\infty}
			K_1(\bfi,\bfj)\,f_i f_j\,
			\bigl[\varphi(i+j)-\varphi(i)-\varphi(j)\bigr].
		\end{align*}
		
		Next, using the symmetry of the kernel \(K_2\), we compute
		\begin{align*}
			\sum_{i=1}^{\infty} S_i^{(2)}[f]\,\varphi(i)
			= {} &
			-\sum_{i=1}^{\infty}\sum_{j=1}^{i-1}
			K_2(\bfj,\bfi-\bfj)\,f_i f_j\,\varphi(i)
			+\sum_{i=1}^{\infty}\sum_{j=i+1}^{\infty}
			K_2(\bfj-\bfi,\bfi)\,f_i f_j\,\varphi(i) \\
			&\quad
			+\sum_{i=1}^{\infty}\sum_{j=i+1}^{\infty}
			K_2(\bfi,\bfj-\bfi)\,f_j f_{j-i}\,\varphi(i) \\
			= {} &
			-\sum_{i,j=1}^{\infty}
			K_2(\bfi,\bfj)\,f_i f_{i+j}\,\varphi(i+j)
			+\sum_{i,j=1}^{\infty}
			K_2(\bfi,\bfj)\,f_i f_{i+j}\,\varphi(i) \\
			&\quad
			+\sum_{i,j=1}^{\infty}
			K_2(\bfi,\bfj)\,f_i f_{i+j}\,\varphi(j) \\
			= {} &
			-\sum_{i,j=1}^{\infty}
			K_2(\bfi,\bfj)\,f_i f_{i+j}\,
			\bigl[\varphi(i+j)-\varphi(i)-\varphi(j)\bigr].
		\end{align*}

		Similarly, we obtain
		\[
		\sum_{i=1}^{\infty} S_i^{(3)}[f]\,\varphi(i)
		=
		-\sum_{i,j=1}^{\infty}
		K_3(\bfi,\bfj)\,f_i f_{i+j}\,
		\bigl[\varphi(i+j)-\varphi(i)-\varphi(j)\bigr].
		\]
		
		%    	\begin{align*}
			%    		\sum_{i=1}^\infty S_i^{(3)}[f]\varphi(i)= & -\sum_{i=1}^\infty\sum_{j=1}^{i-1} K_3(j, i-j) f_i f_{i-j} \varphi(i) +\sum_{i=1}^\infty\sum_{j=i+1}^{\infty} K_3(j-i, i) f_j f_i\varphi(i) \\
			%    		+& \sum_{i=1}^\infty\sum_{j=1}^{\infty} K_3(i, j) f_j f_{i+j}\varphi(i)\\
			%    		= & -\sum_{i,j=1}^\infty K_3(i, j) f_{j} f_{i+j} \varphi(i+j) +\sum_{i=1}^\infty K_3(i, j) f_j f_{i+j}\varphi(j) \\
			%    		+& \sum_{i,j=1}^\infty K_3(i, j) f_j f_{i+j}\varphi(i)\\
			%    		=&- \sum_{i,j=1}^\infty K_3(i, j) f_j f_{i+j}[\varphi(i+j)-\varphi(i)-\varphi(j)]
			%    	\end{align*}
	\end{proof}	
	\begin{prop}
		Let \(f:\mathbb{N}\to\mathbb{R}_{\ge 0}\), and let \(0\le a<k<b\). Then the following
		moment interpolation inequalities hold:
		\begin{enumerate}
			\item
			\begin{equation}\label{eqn:moment_interpolation}
				m_k\langle f\rangle
				\le
				\bigl(m_a\langle f\rangle\bigr)^{\frac{b-k}{b-a}}
				\bigl(m_b\langle f\rangle\bigr)^{\frac{k-a}{b-a}} .
			\end{equation}
			
			\item
			For any \(\varepsilon>0\), there exists a constant \(C_\varepsilon>0\) such that
			\begin{equation}\label{eqn:moment_interpolation_plus}
				m_k\langle f\rangle
				\le
				C_\varepsilon\, m_a\langle f\rangle
				+\varepsilon\, m_b\langle f\rangle .
			\end{equation}
		\end{enumerate}
		
	\end{prop}
	%    $$m_k  \langle f \rangle \le (m_{k\frac{p}{q}} \langle f \rangle )^q (m_{k\frac{1-p}{1-q}} \langle f \rangle )^{1-q} $$
	%
	%Set $a=k\frac{p}{q}, b=k \frac{1-p}{1-q}$, it becomes clear that, if $a < k <b$, then
	\begin{proof}
		\leavevmode
		\begin{enumerate}
			\item
			Since \(f\ge 0\), by Hölder’s inequality we have
			\begin{align*}
				m_k\langle f\rangle
				&= \sum_{i=1}^{\infty} f_i\,\bfi^{k}
				= \sum_{i=1}^{\infty}
				\bigl(f_i\,\bfi^{a}\bigr)^{\frac{b-k}{b-a}}
				\bigl(f_i\,\bfi^{b}\bigr)^{\frac{k-a}{b-a}} \\
				&\le
				\left(\sum_{i=1}^{\infty} f_i\,\bfi^{a}\right)^{\frac{b-k}{b-a}}
				\left(\sum_{i=1}^{\infty} f_i\,\bfi^{b}\right)^{\frac{k-a}{b-a}} \\
				&=
				\bigl(m_a\langle f\rangle\bigr)^{\frac{b-k}{b-a}}
				\bigl(m_b\langle f\rangle\bigr)^{\frac{k-a}{b-a}},
			\end{align*}
			which proves \eqref{eqn:moment_interpolation}.
			
			\item
			By Young’s inequality, for any \(\varepsilon>0\) there exists a constant
			\(C_\varepsilon>0\) such that, for all \(i\in\mathbb{N}\),
			\[
			\bfi^{k} \le C_\varepsilon\,\bfi^{a} + \varepsilon\,\bfi^{b}.
			\]
			Therefore,
			\begin{align*}
				m_k\langle f\rangle
				&= \sum_{i=1}^{\infty} f_i\,\bfi^{k}
				\le \sum_{i=1}^{\infty}
				f_i\bigl(C_\varepsilon\,\bfi^{a} + \varepsilon\,\bfi^{b}\bigr) \\
				&=
				C_\varepsilon\, m_a\langle f\rangle
				+ \varepsilon\, m_b\langle f\rangle,
			\end{align*}
			which proves \eqref{eqn:moment_interpolation_plus}.
		\end{enumerate}
	\end{proof}
	
	\begin{remark}\label{rmk:eat_the_term_by_constant}
		Inequalities \eqref{eqn:moment_interpolation} and
		\eqref{eqn:moment_interpolation_plus} imply that, if
		\(f\) satisfies
		\[
		m_a\langle f\rangle \le \mathfrak{m}_a,
		\qquad
		m_b\langle f\rangle \le \mathfrak{m}_b,
		\]
		for some positive constants \(\mathfrak{m}_a\) and \(\mathfrak{m}_b\), then for any
		\(k\) with \(a\le k\le b\),
		\[
		m_k\langle f\rangle
		\le C(\mathfrak{m}_a,\mathfrak{m}_b)
		\lesssim_{\mathfrak{m}_a,\mathfrak{m}_b} 1 .
		\]
		
	\end{remark}
	\begin{prop}[Conservation Law]
		
		Let \(f:\mathbb{N}\to\mathbb{R}_{\ge 0}\) satisfy
		\[
		m_{\alpha_n+\beta_n+1}\langle f\rangle<\infty,
		\qquad \forall\, n\in\{1,2,3\}.
		\]
		Then the operator \(S\) conserves the first moment, namely
		\begin{equation}\label{eqn:moment_conservation_S}
			m_1\langle S[f]\rangle = 0 .
		\end{equation}
		
	\end{prop}
	\begin{proof}
		This condition ensures that all terms in the weak formulation
		\eqref{eqn:weak_formulation} are finite. Consequently,
		\eqref{eqn:moment_conservation_S} follows from the identity
		\[
		(\bfi+\bfj)-\bfi-\bfj = 0,
		\qquad \forall\, i,j\in\mathbb{N}.
		\]
		
	\end{proof}
	\begin{lem}[Moment Estimates for \(S\)]
		Let \(f:\mathbb{N}\to\mathbb{R}_{\ge 0}\) be such that all its moments are finite. Then the following estimate hold
%		\begin{itemize}
%			\item For \(k=0\),
%			\begin{equation}
%				m_0\langle S[f]\rangle
%				\le
%				m_{\alpha_2}\langle f\rangle\, m_{\alpha_2+\beta_2}\langle f\rangle
%				+
%				m_{\alpha_3}\langle f\rangle\, m_{\alpha_3+\beta_3}\langle f\rangle .
%			\end{equation}
%			
%			\item 
			for \(k\ge 1\),
			\begin{equation}\label{eqn:moment_k_estimation}
				m_k\langle S[f]\rangle
				\le
				2^{\beta_1+1}\bigl(2^k-2\bigr)\,
				m_{\alpha_1+\beta_1+k-1}\langle f\rangle\,
				m_{\alpha_1+1}\langle f\rangle .
			\end{equation}
%		\end{itemize}
	\end{lem}
	
	\begin{proof}
%		For \(k=0\), we compute
%		\begin{align*}
%			m_0\langle S[f]\rangle
%			&= \sum_{i=1}^{\infty} S_i[f] \\
%			&=
%			\sum_{i,j=1}^{\infty}
%			\Bigl[
%			K_1(\bfi,\bfj)\,f_i f_j
%			- K_2(\bfi,\bfj)\,f_i f_{i+j}
%			- K_3(\bfi,\bfj)\,f_j f_{i+j}
%			\Bigr]
%			\bigl[1-1-1\bigr] \\
%			&\le
%			\sum_{i,j=1}^{\infty}
%			\Bigl(
%			K_2(\bfi,\bfj)\,f_i f_{i+j}
%			+ K_3(\bfi,\bfj)\,f_i f_{i+j}
%			\Bigr),
%		\end{align*}
%		where we used the weak formulation with the test function \(\varphi\equiv 1\).
%		
%		Next, observe that
%		\begin{align*}
%			\sum_{i,j=1}^{\infty} K_2(\bfi,\bfj)\,f_i f_{i+j}
%			&=
%			\sum_{i=1}^{\infty}\sum_{j=i+1}^{\infty}
%			\bfi^{\alpha_2}(\bfj-\bfi)^{\alpha_2}\bfj^{\beta_2}\, f_i f_j \\
%			&\le
%			\sum_{i=1}^{\infty} \bfi^{\alpha_2} f_i
%			\sum_{j=1}^{\infty} \bfj^{\alpha_2+\beta_2} f_j
%			=
%			m_{\alpha_2}\langle f\rangle\,
%			m_{\alpha_2+\beta_2}\langle f\rangle .
%		\end{align*}
%		An analogous estimate holds for the term involving \(K_3\). Therefore,
%		\[
%		m_0\langle S[f]\rangle
%		\le
%		m_{\alpha_2}\langle f\rangle\,
%		m_{\alpha_2+\beta_2}\langle f\rangle
%		+
%		m_{\alpha_3}\langle f\rangle\,
%		m_{\alpha_3+\beta_3}\langle f\rangle .
%		\]
		
		For \(k=1\), recalling \eqref{eqn:moment_conservation_S}, we have
		\[
		m_1\langle S[f]\rangle = 0,
		\]
		which immediately implies \eqref{eqn:moment_k_estimation} for \(k=1\). 
		
		For \(k>1\), the terms involving \(K_2\) and \(K_3\) contribute non-positively.
		Hence,
		\[
		m_k\langle S[f]\rangle
		\le
		\sum_{i=1}^{\infty} S_i^{(1)}[f]\,\bfi^{k}.
		\]
		
		%\begin{equation}
		%    \begin{aligned}
			%        m_k\langle S[f]\rangle &\le \sum_{i=1}^\infty S^{(1)}_i[f] |\bfi|^k\\ 
			%        & \le \sum_{i,j=1}^\infty K_1(i,j) f_i f_j (|\bfi+\bfj|^k-|\bfi|^k -|\bfj|^k)\\
			%        & \le C_k \sum_{i,j=1}^\infty |\bfi|^{\alpha_1}|\bfj|^{\alpha_1}|\bfi+\bfj|^{\beta_1}  f_i f_j(|\bfi|^{k-1}|\bfj|+ |\bfi||\bfj|^{k-1})\\
			%        & \le 2C_k \sum_{i,j=1}^\infty |\bfi|^{\alpha_1}|\bfj|^{\alpha_1}|\bfi+\bfj|^{\beta_1}  f_i f_j|\bfi|^{k-1}|\bfj|\\
			%        & \le 2C'_k \sum_{i,j=1}^\infty |\bfi|^{\alpha_1}|\bfj|^{\alpha_1}(|\bfi|^{\beta_1}+|\bfj|^{\beta_1})  f_i f_j|\bfi|^{k-1}|\bfj|\\
			%        & = 2C'_k (m_{\alpha_1+\beta_1+k-1}\langle f\rangle m_{\alpha_1+1}\langle f\rangle +m_{\alpha_1+k-1} \langle f \rangle m_{\alpha_1+\beta_1+1}  \langle f \rangle  ).
			%    \end{aligned}
		%\end{equation}
		
		Note that for \(k>1\) and \(0\le x\le y\), the following inequality holds:
		\[
		(x+y)^k - x^k - y^k \le (2^k-2)\, x\, y^{k-1}.
		\]
		
		Using the weak formulation, we obtain
		\begin{align*}
			m_k\langle S[f]\rangle
			&\le \sum_{i=1}^{\infty} S_i^{(1)}[f]\,\bfi^k \\
			&\le
			\sum_{i,j=1}^{\infty}
			K_1(\bfi,\bfj)\, f_i f_j
			\bigl[(\bfi+\bfj)^k-\bfi^k-\bfj^k\bigr].
		\end{align*}
		We split the sum into the regions \(i\le j\) and \(j\le i\), and use the symmetry of
		the kernel:
		\begin{align*}
			m_k\langle S[f]\rangle
			\le {} &
			\sum_{i\le j}
			\bfi^{\alpha_1}\bfj^{\alpha_1}(\bfi+\bfj)^{\beta_1}
			f_i f_j
			\bigl[(\bfi+\bfj)^k-\bfi^k-\bfj^k\bigr] \\
			&+
			\sum_{j\le i}
			\bfi^{\alpha_1}\bfj^{\alpha_1}(\bfi+\bfj)^{\beta_1}
			f_i f_j
			\bigl[(\bfi+\bfj)^k-\bfi^k-\bfj^k\bigr].
		\end{align*}
		Applying the above inequality and the bound \(\bfi+\bfj\le 2\max\{\bfi,\bfj\}\),
		we deduce
		\begin{align*}
			m_k\langle S[f]\rangle
			\le {} &
			(2^k-2)\Bigg[
			\sum_{i\le j}
			\bfi^{\alpha_1}\bfj^{\alpha_1}(2\bfj)^{\beta_1}
			f_i f_j\,\bfi\,\bfj^{k-1} \\
			&\hspace{2.2cm}
			+\sum_{j\le i}
			\bfi^{\alpha_1}\bfj^{\alpha_1}(2\bfi)^{\beta_1}
			f_i f_j\,\bfj\,\bfi^{k-1}
			\Bigg].
		\end{align*}
		Combining the two sums yields
		\begin{align*}
			m_k\langle S[f]\rangle
			&\le
			2^{\beta_1+1}(2^k-2)
			\sum_{i,j=1}^{\infty}
			\bfi^{\alpha_1}\bfj^{\alpha_1+\beta_1}
			f_i f_j\,\bfi\,\bfj^{k-1} \\
			&=
			2^{\beta_1+1}(2^k-2)\,
			m_{\alpha_1+1}\langle f\rangle\,
			m_{\alpha_1+\beta_1+k-1}\langle f\rangle .
		\end{align*}
		This concludes the proof of \eqref{eqn:moment_k_estimation}.
		
	\end{proof}
	\begin{prop}\label{prop:moment_est_for_O}
		Let \(f:\mathbb{N}\to\mathbb{R}_{\ge 0}\) be such that all its moments are finite.
		Then for \(k>1\) and \(k>2-\alpha_1-\beta_1\), the following moment estimate holds:
		\begin{equation}\label{eqn:moment_estimate_for_m_O}
			m_k\langle O[f]\rangle
			= m_k\langle (S-V)[f]\rangle
			\le
			\mathcal{C}_k\, m_1^{\,1+\frac{\delta+k-1}{\delta-2\alpha_1-\beta_1+1}}
			-\frac{1}{2}\, m_{\delta+k}\langle f\rangle .
		\end{equation}
		
	\end{prop}
	\begin{proof}
		Note that
		\[
		m_k\langle V[f]\rangle
		=
		\sum_{i=1}^{\infty} \bfi^{\delta} f_i\, \bfi^{k}
		=
		m_{\delta+k}\langle f\rangle .
		\]
		
		It then follows from \eqref{eqn:moment_k_estimation} that, for \(k>1\),
		\[
		m_k\langle (S-V)[f]\rangle
		\le
		2^{\beta_1+1}(2^k-2)\,
		m_{\alpha_1+\beta_1+k-1}\langle f\rangle\,
		m_{\alpha_1+1}\langle f\rangle
		-
		m_{\delta+k}\langle f\rangle .
		\]
		
		Applying the moment interpolation inequality \eqref{eqn:moment_interpolation}, since our assumption implies that \(1<\alpha_1+\beta_1+k-1<\delta+k\) and \(1<\alpha_1+1<\delta+k\),
		we obtain
		\[
		m_{\alpha_1+\beta_1+k-1}\langle f\rangle
		\le
		m_1^{\frac{\delta-\alpha_1-\beta_1+1}{\delta+k-1}}\,
		m_{\delta+k}^{\frac{\alpha_1+\beta_1+k-2}{\delta+k-1}},
		\]
		and
		\[
		m_{\alpha_1+1}\langle f\rangle
		\le
		m_1^{\frac{\delta+k-\alpha_1-1}{\delta+k-1}}\,
		m_{\delta+k}^{\frac{\alpha_1}{\delta+k-1}} .
		\]
		%\begin{align*}
		%    m_{\alpha_1+\beta_1+k-1} &\le m_1^{\frac{\delta -\alpha_1-\beta_1+1}{\delta+k-1}} m_{\delta+k}^{\frac{\alpha_1+\beta_1+k-2}{\delta+k-1}}    \\
		%    m_{\alpha_1+1} &\le m_1^{\frac{\delta +k -\alpha_1-1}{\delta+k-1}} m_{\delta+k}^{\frac{\alpha_1}{\delta+k-1}} 
		%    \\
		%    m_{\alpha_1+k-1} &\le m_1^{\frac{\delta -\alpha_1+1}{\delta+k-1}} m_{\delta+k}^{\frac{\alpha_1+k-2}{\delta+k-1}}   \\
		%    m_{\alpha_1+\beta_1+1} &\le m_1^{\frac{\delta +k -\alpha_1-\beta_1-1}{\delta+k-1}} m_{\delta+k}^{\frac{\alpha_1+\beta_1}{\delta+k-1}}
		%\end{align*}
		
		Thus,
		\[
		m_k\langle (S-V)[f]\rangle
		\le
		2^{\beta_1+1}(2^k-2)\,
		m_1^{\frac{2\delta+k-2\alpha_1-\beta_1}{\delta+k-1}}\,
		m_{\delta+k}^{\frac{2\alpha_1+\beta_1+k-2}{\delta+k-1}}
		-
		m_{\delta+k}.
		\]
		
		Since \(\delta>2\alpha_1+\beta_1-1\), we have
		\[
		\frac{2\alpha_1+\beta_1+k-2}{\delta+k-1}<1 .
		\]
		Therefore, Young’s inequality yields
		\begin{align*}
			m_k\langle (S-V)[f]\rangle
			&\le
			\mathcal{C}_k\,
			m_1^{\frac{2\delta+k-2\alpha_1-\beta_1}{\delta-2\alpha_1-\beta_1+1}}
			-\frac{1}{2}\, m_{\delta+k} \\
			&=
			\mathcal{C}_k\,
			m_1^{\,1+\frac{\delta+k-1}{\delta-2\alpha_1-\beta_1+1}}
			-\frac{1}{2}\, m_{\delta+k},
		\end{align*}
		where a direct computation shows that
		\begin{equation}\label{eqn:def_for_C_k}
			\mathcal{C}_k
			=
			\frac{\delta-2\alpha_1-\beta_1+1}{\delta+k-1}\,
			\bigl[2^{\beta_1+1}(2^k-2)\bigr]^{\frac{\delta+k-1}{\delta-2\alpha_1-\beta_1+1}}
			\left(
			\frac{2(2\alpha_1+\beta_1+k-2)}{\delta+k-1}
			\right)^{\frac{2\alpha_1+\beta_1+k-2}{\delta-2\alpha_1-\beta_1+1}} .
		\end{equation}
		
	\end{proof}
	%	\begin{remark}
		%		For the sake of completeness, we provide the explicit expression of \(\mathcal{C}_k\),
		%		as this constant will appear frequently in the results of the subsequent two sections.
		%		We note that \(\mathcal{C}_k\) exhibits super-exponential growth with respect to \(k\).
		%		
		%	\end{remark}
	\begin{remark}Suppose further that \(f\not\equiv 0\). Then, by the moment interpolation inequality
		\eqref{eqn:moment_interpolation}, and \(k>1\), we have
		\[
		m_k\langle f\rangle
		\le
		m_1^{\frac{\delta}{\delta+k-1}}\,
		m_{\delta+k}^{\frac{k-1}{\delta+k-1}} .
		\]
		
		Since \(m_1\langle f\rangle\neq 0\), it follows that
		\[
		m_{\delta+k}\langle f\rangle
		\ge
		m_1^{-\frac{\delta}{k-1}}\,
		m_k^{\frac{\delta+k-1}{k-1}} .
		\]
		Substituting this bound into \eqref{eqn:moment_estimate_for_m_O} yields
		\begin{equation}\label{eqn:moment_estimate_m_O_when_zero}
			m_k\langle (S-V)[f]\rangle
			\le
			\mathcal{C}_k\,
			m_1^{\,1+\frac{\delta+k-1}{\delta-2\alpha_1-\beta_1+1}}
			-\frac{1}{2}\,
			m_1^{-\frac{\delta}{k-1}}\,
			m_k^{\frac{\delta+k-1}{k-1}},
			\qquad f\not\equiv 0 .
		\end{equation}
		
	\end{remark}

	\begin{prop}
		Let \(f,g\in \ell^1(\mathbb{N})\) be nonnegative functions satisfying
		\[
		\max\!\left\{m_1\langle f\rangle,\; m_1\langle g\rangle\right\}
		\le \mathfrak{m}_1<\infty,
		\qquad
		\max\!\left\{m_{\kappa}\langle f\rangle,\; m_{\kappa}\langle g\rangle\right\}
		\le \mathfrak{m}_{\kappa}<\infty,
		\]
		for some \(\kappa>\delta\).
		Then there exists a constant \(C(\mathfrak{m}_1,\mathfrak{m}_{\kappa})>0\),
		depending only on \(\mathfrak{m}_1\) and \(\mathfrak{m}_{\kappa}\), such that
		\begin{equation}\label{eqn:Holder_estimate_pre_for_S}
			\|S[f]-S[g]\|_{\ell^1}
			\le
			C(\mathfrak{m}_1,\mathfrak{m}_{\kappa})\,\|f-g\|_{\ell^1}
			+
			m_\delta\langle |f-g|\rangle .
		\end{equation}
		
	\end{prop}
	\begin{proof}
		For such \(f\) and \(g\), we write
		\[
		\|S^{(1)}[f]-S^{(1)}[g]\|_{\ell^1}
		=
		\sum_{i=1}^{\infty}
		\bigl|S_i^{(1)}[f]-S_i^{(1)}[g]\bigr|
		=
		\sum_{i=1}^{\infty}
		\bigl(S_i^{(1)}[f]-S_i^{(1)}[g]\bigr)\, s(i),
		\]
		where \(s:\mathbb{N}\to\{-1,1\}\) is a sign function.
		
		Using the weak formulation \eqref{eqn:weak_formulation}, we obtain
		\begin{align*}
			\|S^{(1)}[f]-S^{(1)}[g]\|_{\ell^1}
			&=
			\sum_{i,j=1}^{\infty}
			K_1(\bfi,\bfj)\,
			\bigl(f_i f_j-g_i g_j\bigr)
			\bigl(s(i+j)-s(i)-s(j)\bigr) \\
			&\le
			3\sum_{i,j=1}^{\infty}
			K_1(\bfi,\bfj)\,
			\bigl|f_i f_j-g_i g_j\bigr|,
		\end{align*}
		where we used the bound
		\(|s(i+j)-s(i)-s(j)|\le 3\).
		
		%	\begin{align*}
			%		&\left\|S^{(1)}[f]-S^{(1)}[g]\right\|
			%		\le\sum_{i=1}^\infty \left|\sum_{j=1}^{i-1} K_1(j, i-j) f_j f_{i-j}- \sum_{j=1}^{i-1} K_1(j, i-j) g_j g_{i-j} \right|\\
			%		+&2\sum_{i=1}^\infty \left|-\sum_{j=1}^{\infty} K_1(i, j) f_i f_{j}+ \sum_{j=1}^{\infty} K_1(i, j) g_i g_{j} \right| \\
			%		\le & \sum_{i=1}^\infty \sum_{j=1}^{i-1} K_1(j, i-j) \left|f_j f_{i-j}-g_j g_{i-j} \right|+2\sum_{i=1}^\infty \sum_{j=1}^{\infty} K_1(i, j) \left|f_i f_{j}-g_i g_{j} \right| \\
			%		=& 3\sum_{k=1}^\infty\sum_{j=1}^{\infty} K_1(j, k) \left|f_j f_k-g_j g_k \right|,
			%	\end{align*}
		Since \(\delta>\alpha_1+\beta_1\), by Remark~\ref{rmk:eat_the_term_by_constant}
		and the interpolation estimate \eqref{eqn:moment_interpolation_plus}, we obtain
		\begin{equation}\label{eqn:holder_est_for_S1}
			\begin{aligned}
				\sum_{i=1}^{\infty}\sum_{j=1}^{\infty}
				K_1(\bfi,\bfj)\,\bigl|f_i f_j-g_i g_j\bigr|
				&\le
				\sum_{i=1}^{\infty}\sum_{j=1}^{\infty}
				\bfi^{\alpha_1}\bfj^{\alpha_1}
				\bigl(\bfi^{\beta_1}+\bfj^{\beta_1}\bigr)
				\bigl|f_i f_j-g_i g_j\bigr| \\
				&\le
				2\sum_{i=1}^{\infty}\sum_{j=1}^{\infty}
				\bfi^{\alpha_1}\bfj^{\alpha_1+\beta_1}
				\bigl(|f_i-g_i|\,|f_j|
				+|g_i|\,|f_j-g_j|\bigr) \\
				&=
				2\,m_{\alpha_1+\beta_1}\langle f\rangle\,
				m_{\alpha_1}\langle |f-g|\rangle
				+
				2\,m_{\alpha_1}\langle g\rangle\,
				m_{\alpha_1+\beta_1}\langle |f-g|\rangle \\
				&\le
				C_1(\mathfrak{m}_1,\mathfrak{m}_{\kappa})\,
				m_0\langle |f-g|\rangle
				+\frac{1}{3}\,m_\delta\langle |f-g|\rangle .
			\end{aligned}
		\end{equation}
		
		%	
		%	Here the last inequality follows from Remark \ref{rmk:eat_the_term_by_constant}.
		%	
		Similarly, using the weak formulation \eqref{eqn:weak_formulation}, there exists a sign
		function \(s:\mathbb{N}\to\{-1,1\}\) such that
		\begin{align*}
			\|S^{(2)}[f]-S^{(2)}[g]\|_{\ell^1}
			&=
			\sum_{i,j=1}^{\infty}
			-\,K_2(\bfi,\bfj)\,
			\bigl(f_i f_{i+j}-g_i g_{i+j}\bigr)
			\bigl(s(i+j)-s(i)-s(j)\bigr) \\
			&\le
			3\sum_{i,j=1}^{\infty}
			K_2(\bfi,\bfj)\,
			\bigl|f_i f_{i+j}-g_i g_{i+j}\bigr| \\
			&=
			3\sum_{i=1}^{\infty}\sum_{j=1}^{i-1}
			K_2(\bfj,\bfi-\bfj)\,
			\bigl|f_i f_j-g_i g_j\bigr| .
		\end{align*}
		
		%	
		%	\begin{align*}
			%		\left\|S^{(2)}[f]-S^{(2)}[g]\right\|&\le \sum_{i=1}^\infty \left|-\sum_{j=1}^{i-1} K_2(j, i-j) f_i f_{j}+ \sum_{j=1}^{i-1} K_2(j, i-j) g_i g_{j} \right|\\
			%		+&\sum_{i=1}^\infty \left|\sum_{j=i+1}^{\infty} K_2(j-i, i) f_j f_{j-i}- \sum_{j=i+1}^{\infty} K_2(j-i, i) g_j g _{j-i} \right|\\
			%		+&\sum_{i=1}^\infty \left|\sum_{j=1}^{\infty} K_2(i, j) f_i f_{i+j}- \sum_{j=1}^{\infty} K_2(i, j) g_i g _{i+j} \right|\\
			%		=& I_1+I_2+I_3.
			%	\end{align*}
		%	
		%	We have that
		%	\begin{align*}
			%		I_1=&\sum_{i=1}^\infty \left|-\sum_{j=1}^{i-1} K_2(j, i-j) f_i f_{j}+ \sum_{j=1}^{i-1} K_2(j, i-j) g_i g_{j} \right|\\
			%		\le & \sum_{i=1}^\infty \sum_{j=1}^{i-1} K_2(j, i-j) \left|f_i f_{j}-g_i g_{j} \right|\\
			%		= &\sum_{k=1}^\infty \sum_{j=1}^\infty K_2(j,k) \left|f_{k+j} f_{j}-g_{k+j} g_{j} \right|
			%	\end{align*}
		%	
		%	Also
		%	\begin{align*}
			%		I_2=&\sum_{i=1}^\infty \left|\sum_{j=i+1}^{\infty} K_2(j-i, i) f_j f_{j-i}- \sum_{j=i+1}^{\infty} K_2(j-i, i) g_j g _{j-i} \right|\\
			%		\le & \sum_{i=1}^\infty \sum_{j=i+1}^{\infty} K_2(j-i, i) \left|f_j f_{j-i}-g_j g_{j-i} \right|\\
			%		= & \sum_{i=1}^\infty \sum_{k=1}^\infty K_2(k,i)|f_{k+i}f_k-g_{k+i}g_k|
			%	\end{align*}
		%	\begin{align*}
			%		I_3=&\sum_{i=1}^\infty \left|\sum_{j=1}^{\infty} K_2(i, j) f_i f_{i+j}- \sum_{j=1}^{\infty} K_2(i, j) g_i g _{i+j} \right|\\
			%		\le & \sum_{i=1}^\infty \sum_{j=1}^\infty K_2(i,j)|f_{i+j}f_i-g_{i+j}g_i|
			%	\end{align*}
		Similarly, since \(\delta>\alpha_2+\beta_2\), we obtain
		\begin{equation}\label{eqn:holder_est_for_S2}
			\begin{aligned}
				\sum_{i=1}^{\infty}\sum_{j=1}^{i-1}
				K_2(\bfj,\bfi-\bfj)\,
				\bigl|f_i f_j-g_i g_j\bigr|
				&\le
				\sum_{i=1}^{\infty}\sum_{j=1}^{\infty}
				\bfj^{\alpha_2}\bfi^{\alpha_2+\beta_2}
				\bigl(|f_i-g_i|\,|f_j|
				+|f_j-g_j|\,|g_i|\bigr) \\
				&\le
				m_{\alpha_2+\beta_2}\langle f\rangle\,
				m_{\alpha_2}\langle |f-g|\rangle
				+
				m_{\alpha_2}\langle g\rangle\,
				m_{\alpha_2+\beta_2}\langle |f-g|\rangle \\
				&\le
				C_2(\mathfrak{m}_1,\mathfrak{m}_{\kappa})\,
				m_0\langle |f-g|\rangle
				+\frac{1}{3}\,m_\delta\langle |f-g|\rangle .
			\end{aligned}
		\end{equation}
		
		The same argument applies to \(S^{(3)}\), yielding
		\begin{equation}\label{eqn:holder_est_for_S3}
			\|S^{(3)}[f]-S^{(3)}[g]\|_{\ell^1}
			\le
			C_3(\mathfrak{m}_1,\mathfrak{m}_{\kappa})\,
			m_0\langle |f-g|\rangle
			+\frac{1}{3}\,m_\delta\langle |f-g|\rangle .
		\end{equation}
		
		Combining \eqref{eqn:holder_est_for_S1},
		\eqref{eqn:holder_est_for_S2}, and
		\eqref{eqn:holder_est_for_S3}, and applying the triangle inequality,
		we obtain \eqref{eqn:Holder_estimate_pre_for_S}.
		
	\end{proof}
	
	\begin{prop}[Holder Continuity]\label{prop:Holder_est}
		Let \(f,g\in \ell^1(\mathbb{N})\) be nonnegative functions satisfying
		\[
		\max\!\left\{m_1\langle f\rangle,\; m_1\langle g\rangle\right\}
		\le \mathfrak{m}_1<\infty,
		\qquad
		\max\!\left\{m_{\kappa}\langle f\rangle,\; m_{\kappa}\langle g\rangle\right\}
		\le \mathfrak{m}_{\kappa}<\infty,
		\]
		for some \(\kappa>\delta\).
		Then
		\[
		\|O[f]-O[g]\|_{\ell^1}
		\lesssim_{\mathfrak{m}_1,\mathfrak{m}_{\kappa}}
		\|f-g\|_{\ell^1}^{\frac{\kappa-\delta}{\kappa}}
		+\|f-g\|_{\ell^1}.
		\]
		
	\end{prop}
	\begin{proof}
		Estimate \eqref{eqn:Holder_estimate_pre_for_S} yields
		\begin{align*}
			\|O[f]-O[g]\|_{\ell^1}
			&\le
			\|S[f]-S[g]\|_{\ell^1}
			+\|V[f]-V[g]\|_{\ell^1} \\
			&\le
			C(\mathfrak{m}_1,\mathfrak{m}_{\kappa})\,\|f-g\|_{\ell^1}
			+ m_\delta\langle |f-g|\rangle
			+\sum_{i=1}^{\infty} \bfi^{\delta}\,|f_i-g_i| \\
			&\le
			C(\mathfrak{m}_1,\mathfrak{m}_{\kappa})\,\|f-g\|_{\ell^1}
			+2\,m_\delta\langle |f-g|\rangle .
		\end{align*}
		
		Applying the moment interpolation inequality \eqref{eqn:moment_interpolation}, we further obtain
		\begin{align*}
			\|O[f]-O[g]\|_{\ell^1}
			&\le
			C(\mathfrak{m}_1,\mathfrak{m}_{\kappa})\,\|f-g\|_{\ell^1}
			+2\bigl(m_{\kappa}\langle f\rangle+m_{\kappa}\langle g\rangle\bigr)^{\frac{\delta}{\kappa}}
			\,m_0\langle |f-g|\rangle^{\frac{\kappa-\delta}{\kappa}} \\
			&\lesssim_{\mathfrak{m}_1,\mathfrak{m}_{\kappa}}
			\|f-g\|_{\ell^1}^{\frac{\kappa-\delta}{\kappa}}
			+\|f-g\|_{\ell^1}.
		\end{align*}
		
		%	\begin{equation}\label{eqn:holder_est_for_V}
			%		\begin{aligned}
				%			\left\|V[f]-V[g]\right\|_{\ell^1}&=\sum_{i=1}^\infty \left|V[f]-V[g]\right|= m_\delta \langle \left|f-g\right| \rangle\\
				%			&\le (m_{\kappa} \langle f \rangle+m_{\kappa} \langle g \rangle)^{\frac{\delta}{\kappa}}m_0\langle \left|f-g \right|\rangle^{\frac{\kappa-\delta}{\kappa}}  \\
				%			&\lesssim_{\mathfrak{m}_1,\mathfrak{m}_{\kappa}}\left\|f-g\right\|_{\ell^1}^{\frac{\kappa-\delta}{\kappa}}.
				%		\end{aligned}
			%	\end{equation}
		%	
		%	Note that
		%	\[\frac{\kappa-\delta}{\kappa}\le \frac{\kappa-\alpha_n-\beta_n}{\kappa},\quad\forall n\in \{1,2,3\},\]
		%	by \(\delta\ge \alpha_n+\beta_n\) for all \(n\). 
		%	
		%	
		%	Hence combining \eqref{eqn:holder_est_for_S1}, \eqref{eqn:holder_est_for_S2}, \eqref{eqn:holder_est_for_S3} and \eqref{eqn:holder_est_for_V}, and by Young's inequality, we have that
		%	\[\left\|O[f]-O[g]\right\|_{\ell^1}\lesssim_{\mathfrak{m}_1,\mathfrak{m}_{\kappa}}\left\|f-g\right\|_{\ell^1}^{\frac{\kappa-\delta}{\kappa}}+\left\|f-g\right\|_{\ell^1}.\]
	\end{proof}
	
	%\begin{remark}\label{rmk:moment_est_for_l1_diff_S}
	%	It can be observed from the proof that, for \(f,g\in\ell^1(\mathbb{N})\), we would have
	%	\begin{align*}
		%		\left\|S[f]-S[g]\right\|_{\ell^1}&\le 6 m_{\alpha_1+\beta_1}  \langle f \rangle m_{\alpha_1} \langle \left|f-g\right| \rangle+6 m_{\alpha_1}  \langle g \rangle m_{\alpha_1+\beta_1} \langle \left|f-g \right|\rangle\\
		%		&+3m_{\alpha_2+\beta_2} \langle f \rangle m_{\alpha_2} \langle \left|f-g\right| \rangle+3m_{\alpha_2} \langle f \rangle m_{\alpha_2+\beta_2} \langle \left|f-g \right|\rangle\\
		%		&+3m_{\alpha_3+\beta_3} \langle f \rangle m_{\alpha_3} \langle \left|f-g\right| \rangle+3m_{\alpha_3} \langle f \rangle m_{\alpha_3+\beta_3} \langle \left|f-g \right|\rangle.
		%	\end{align*}
	%\end{remark}
	
	\section{Well-posedness}
	
	In this section, we prove Theorems~\ref{thm:GWP},
	\ref{thm:Lipschitz_dependence}, and \ref{thm:energy_decay}.
	Before proving Theorem~\ref{thm:GWP}, we first establish the following proposition.

	\begin{prop}[Existence of Viable Set]\label{prop:existence_viable_set}
		Let \(\mathfrak{m}_1\ge 0\), let \(\kappa>\max\{\delta,1,2-\alpha_1-\beta_1\}\), and let \(\mathfrak{m}_{\kappa}\) satisfy
		\[
		\mathfrak{m}_{\kappa}\ge 2\,\mathcal{B}_{\kappa}(\mathfrak{m}_1).
		\]
		Define the subset
		\[
		\mathcal{S}
		:=
		\left\{
		f\in \ell^1(\mathbb{N}) :
		f\ge 0,\;
		m_1\langle f\rangle\le \mathfrak{m}_1,\;
		m_{\kappa}\langle f\rangle\le \mathfrak{m}_{\kappa}
		\right\}.
		\]
		Then \(\mathcal{S}\) satisfies the tangency condition
		\begin{equation}\label{eqn:sub_tangeny_condition}
			\forall\, f\in\mathcal{S},
			\qquad
			\liminf_{\tau\to 0^{+}}
			\frac{1}{\tau}\,
			\mathrm{dist}_{\ell^1}\bigl(f+\tau O[f],\,\mathcal{S}\bigr)
			=0 .
		\end{equation}
		
	\end{prop}
	\begin{proof}
		Since \eqref{eqn:sub_tangeny_condition} is trivial for \(f\equiv0\), we henceforth
		restrict attention to the case \(f\not\equiv 0\).
		Fix \(0\not\equiv f\in\mathcal{S}\), and define the truncated sequence
		\[
		f_N := (f_i\,\mathbf{1}_{\{i\le N\}})_{i\in\mathbb{N}},
		\qquad N\in\mathbb{N}.
		\]
		There exists \(N_1\in\mathbb{N}\) such that \(f_N\not\equiv 0\) for all \(N\ge N_1\).
		
		For such \(N\), define
		\[
		w_N := f + \tau\, O[f_N],
		\]
		and our goal is to show that \(w_N\in\mathcal{S}\) for \(N\) sufficiently large
		and \(\tau>0\) sufficiently small. 
		%		The precise choice of
		%		\(\mathfrak{m}_{\kappa}\) will be specified later.
		
		Since \(f_N\) has finite support, it is clear that \(O[f_N]\in\ell^1(\mathbb{N})\),
		and therefore \(w_N\in\ell^1(\mathbb{N})\).
		
		We recall the gain--loss decomposition in \eqref{eqn:opt_gain_loss},
		\[
		O[f] = O^+[f] - O^-[f],
		\]
		where
		\begin{align*}
			O^-[f](i)
			&= 2\sum_{j=1}^{\infty} K_1(\bfi,\bfj)\,f_if_j
			+ \sum_{j=1}^{i-1} K_2(\bfi-\bfj,\bfj)\,f_if_j
			+ \sum_{j=1}^{i-1} K_3(\bfi-\bfj,\bfj)\,f_if_j
			+ \bfi^{\delta}f_i \\
			&\le
			\left[C
			\left(\bfi^{\alpha_1+\beta_1} m_{\alpha_1}\langle f\rangle
			+ \bfi^{\alpha_1} m_{\alpha_1+\beta_1}\langle f\rangle
			+ \bfi^{\alpha_2+\beta_2} m_{\alpha_2}\langle f\rangle
			+ \bfi^{\alpha_3+\beta_3} m_{\alpha_3}\langle f\rangle\right)
			+ \bfi^{\delta}\right] f_i,
		\end{align*}
		for some constant \(C>0\).
		
		Since \(0\le f_N\le f\), we have
		\[
		m_k\langle f_N\rangle \le m_k\langle f\rangle
		\le C(\mathfrak{m}_1,\mathfrak{m}_{\kappa}),
		\]
		for all relevant \(k\).
		Moreover, for \(i\le N\), we have \(\bfi\le Nh\). Consequently, for all \(i\),
		\begin{align*}
			O^-[f_N](i)
			&\le
			\left[C(\mathfrak{m}_1,\mathfrak{m}_{\kappa})
			\Bigl(
			(Nh)^{\alpha_1+\beta_1}
			+ (Nh)^{\alpha_1}
			+ (Nh)^{\alpha_2+\beta_2}
			+ (Nh)^{\alpha_3+\beta_3}
			\Bigr)
			+ (Nh)^{\delta}\right]f_N(i) \\
			&\le
			C(\mathfrak{m}_1,\mathfrak{m}_{\kappa},h,N) f_N(i).
		\end{align*}
		
		Since
		\[
		w_N = f + \tau\, O[f_N]
		\ge f - \tau\, O^-[f_N]
		= f\bigl(1-\tau\,C(\mathfrak{m}_1,\mathfrak{m}_{\kappa},h,N)\bigr),
		\]
		it follows that, for \(\tau>0\) sufficiently small, depending on
		\(\mathfrak{m}_1\), \(\mathfrak{m}_{\kappa}\), \(h\), and \(N\), we have
		\[
		w_N \ge 0 .
		\]
		
		Next, since \(f_N\) has finite support, all its moments are finite. By the
		conservation of the first moment for the operator \(S\)
		(cf. \eqref{eqn:moment_conservation_S}), we obtain
		\[
		\sum_{i=1}^{\infty} S[f_N](i)\,\bfi = 0 .
		\]
		Consequently,
		\[m_1\langle w_N\rangle= \sum_{i=1}^{\infty}
		\bigl(f_i - \tau\,\bfi^{\delta} f_N(i)\bigr)\,\bfi\le \sum_{i=1}^{\infty} f_i\,\bfi
		\le \mathfrak{m}_1 .\]

		%	Then by Proposition \ref{prop:moment_est_for_O}, by taking \(k=\kappa
		%	\), we have that,
		%	$$m_{k}  \langle O[f_N] \rangle 
		%	\le C_{k}'' m_1\left\langle f_N\right\rangle^{\frac{2\delta + k-2\alpha_1-\beta_1}{\delta-2\alpha_1-\beta_1+1}}-\frac{1}{2} m_{\delta+k}\left\langle f_N\right\rangle$$
		%	
		%	By taking \(k=\delta+1\), we have that
		%	$$m_{\kappa}  \langle O[f_N] \rangle \le C''_{\kappa} m_1\left\langle f_N\right\rangle^{\frac{2\delta+\kappa -2\alpha_1-\beta_1}{\delta-2\alpha_1-\beta_1+1}}-\frac{1}{2} m_{\delta+\kappa}\left\langle f_N\right\rangle.$$
		%	
		%	Note that
		%	$$m_{\kappa}\left\langle f_N\right\rangle \le m_1\left\langle f_N\right\rangle^\frac{\delta+\kappa-\kappa}{\delta+\kappa-1}m_{\delta+\kappa}\left\langle f_N\right\rangle^\frac{\kappa-1}{\delta+\kappa-1}, $$
		%	then since \(f_N\ne0\), we have that \(m_1\left\langle f_N\right\rangle\ne 0\), thus
		%	$$m_{\delta+\kappa}\left\langle f_N\right\rangle \ge m_1\left\langle f_N\right\rangle^{-\frac{\delta}{\kappa-1}} m_{\kappa}\left\langle f_N\right\rangle^{\frac{\delta+\kappa-1}{\kappa-1}}.$$
		%	
		
		Since \(f_N\not\equiv 0\), by \(\kappa\ge \max\{1,2-\alpha_1-\beta_1\}\), applying \eqref{eqn:moment_estimate_m_O_when_zero} with
		\(k=\kappa\) yields
		\begin{equation}\label{eqn:diff_inequality_for_moments}
			\begin{aligned}
				m_{\kappa}\langle O[f_N]\rangle
				&\le
				\mathcal{C}_{\kappa}\,
				m_1\langle f_N\rangle^{\,1+\frac{\delta+\kappa-1}{\delta-2\alpha_1-\beta_1+1}}
				-\frac{1}{2}\,
				m_1\langle f_N\rangle^{-\frac{\delta}{\kappa-1}}\,
				m_{\kappa}\langle f_N\rangle^{\frac{\delta+\kappa-1}{\kappa-1}} \\
				&\le
				\mathcal{C}_{\kappa}\,
				\mathfrak{m}_1^{\,1+\frac{\delta+\kappa-1}{\delta-2\alpha_1-\beta_1+1}}
				-\frac{1}{2}\,
				\mathfrak{m}_1^{-\frac{\delta}{\kappa-1}}\,
				m_{\kappa}\langle f_N\rangle^{\frac{\delta+\kappa-1}{\kappa-1}} \\
				&=
				\mathfrak{m}_1^{-\frac{\delta}{\kappa-1}}
				\left(
				\mathcal{C}_{\kappa}\,
				\mathfrak{m}_1^{\,1+\frac{\delta+\kappa-1}{\delta-2\alpha_1-\beta_1+1}
					+\frac{\delta}{\kappa-1}}
				-\frac{1}{2}\,
				m_{\kappa}\langle f_N\rangle^{\frac{\delta+\kappa-1}{\kappa-1}}
				\right).
			\end{aligned}
		\end{equation}
		
		Define the function
		\[
		F(m_{\kappa})
		:=
		\mathcal{C}_{\kappa}\,
		\mathfrak{m}_1^{\,1+\frac{\delta+\kappa-1}{\delta-2\alpha_1-\beta_1+1}
			+\frac{\delta}{\kappa-1}}
		-\frac{1}{2}\,
		m_{\kappa}^{\frac{\delta+\kappa-1}{\kappa-1}} .
		\]
		Then \(F(m_{\kappa})<0\) whenever
		\[
		m_{\kappa}\langle f_N\rangle
		>
		\left(
		2\,\mathcal{C}_{\kappa}\,
		\mathfrak{m}_1^{\,1+\frac{\delta+\kappa-1}{\delta-2\alpha_1-\beta_1+1}
			+\frac{\delta}{\kappa-1}}
		\right)^{\frac{\kappa-1}{\delta+\kappa-1}}
		=
		\mathcal{B}_{\kappa}(\mathfrak{m}_1).
		\]
		
		In particular, we always have the uniform upper bound
		\[
		m_{\kappa}\langle O[f_N]\rangle
		\le
		\mathcal{C}_{\kappa}\,
		\mathfrak{m}_1^{\,1+\frac{\delta+\kappa-1}{\delta-2\alpha_1-\beta_1+1}} .
		\]

		%	Then we can define \(\bar{\mathfrak{m}}_{\kappa}\) as
		%	$$\bar{\mathfrak{m}}_{\kappa} := 2\left( 2\mathcal{C}_{\kappa} \mathfrak{m}_1^{1+\frac{\delta+\kappa-1 }{\delta-2\alpha_1-\beta_1+1}+\frac{\delta}{\kappa-1}}\right)^{\frac{\kappa-1}{\delta+\kappa-1}}., $$
		%	which is only depending on \(\mathfrak{m}_{1}\), then 
		
		Let \(\mathfrak{m}_{\kappa}\ge 2\,\mathcal{B}_{\kappa}(\mathfrak{m}_1)\). We distinguish two cases.
		
		\begin{enumerate}
			\item Suppose that \(m_{\kappa}\langle f\rangle \le \mathcal{B}_{\kappa}(\mathfrak{m}_1)\). Then
			\begin{align*}
				m_{\kappa}\langle w_N\rangle
				&=
				\sum_{i=1}^{\infty}
				\bigl(f_i+\tau\,O[f_N](i)\bigr)\,\bfi^{\kappa} \\
				&\le
				\mathcal{B}_{\kappa}(\mathfrak{m}_1)
				+\tau\,\mathcal{C}_{\kappa}\,
				\mathfrak{m}_1^{\,1+\frac{\delta+\kappa-1}{\delta-2\alpha_1-\beta_1+1}} \\
				&\le
				2\,\mathcal{B}_{\kappa}(\mathfrak{m}_1)
				\le
				\mathfrak{m}_{\kappa},
			\end{align*}
			for \(\tau>0\) sufficiently small.
			
			\item Suppose that \(m_{\kappa}\langle f\rangle > \mathcal{B}_{\kappa}(\mathfrak{m}_1)\). Then there exists
			\(N_2\in\mathbb{N}\) such that, for all \(N>N_2\),
			\[
			m_{\kappa}\langle f_N\rangle>\mathcal{B}_{\kappa}(\mathfrak{m}_1).
			\]
			Consequently, \(m_{\kappa}\langle O[f_N]\rangle<0\). Since \(f\in\mathcal{S}\), we obtain
			\begin{align*}
				m_{\kappa}\langle w_N\rangle
				&=
				\sum_{i=1}^{\infty}
				\bigl(f_i+\tau\,O[f_N](i)\bigr)\,\bfi^{\kappa}
				<
				\sum_{i=1}^{\infty} f_i\,\bfi^{\kappa} \\
				&=
				m_{\kappa}\langle f\rangle
				\le
				\mathfrak{m}_{\kappa}.
			\end{align*}
		\end{enumerate}
		
		In both cases, we conclude that
		\[
		m_{\kappa}\langle w_N\rangle \le \mathfrak{m}_{\kappa}.
		\]
		Therefore, for \(N>\max\{N_1,N_2\}\) and for \(\tau>0\) sufficiently small (depending on
		\(\mathfrak{m}_1\), \(\mathfrak{m}_{\kappa}\), \(h\), and \(N\)), we have
		\[
		w_N = f+\tau O[f_N]\in\mathcal{S}.
		\]
		
		Finally, given any \(\varepsilon>0\), there exists \(N_3\in\mathbb{N}\) such that
		\(\|f-f_N\|_{\ell^1}<\varepsilon\) for all \(N>N_3\). For
		\(N>\max\{N_1,N_2,N_3\}\) and \(\tau>0\) sufficiently small, by Proposition~\ref{prop:Holder_est},
		\begin{align*}
			\frac{1}{\tau}\,
			\mathrm{dist}_{\ell^1}\bigl(f+\tau O[f],\mathcal{S}\bigr)
			&\le
			\frac{1}{\tau}\,
			\bigl\|f+\tau O[f]-f-\tau O[f_N]\bigr\|_{\ell^1} \\
			&\le
			\|O[f]-O[f_N]\|_{\ell^1} \\
			&\le
			C(\mathfrak{m}_1,\mathfrak{m}_{\kappa})
			\bigl(\varepsilon^{\frac{\kappa-\delta}{\kappa}}+\varepsilon\bigr).
		\end{align*}
		Since \(\varepsilon>0\) is arbitrary, we conclude that
		\[
		\liminf_{\tau\to 0^{+}}
		\frac{1}{\tau}\,
		\mathrm{dist}_{\ell^1}\bigl(f+\tau O[f],\mathcal{S}\bigr)
		=0,
		\qquad
		\forall\, f\in\mathcal{S}.
		\]
	\end{proof}

	\subsection{Proof of Theorem \ref{thm:GWP}}
	\label{SecTh1}
	Since \(m_{\kappa}\langle f_0\rangle<\infty\), Remark~\ref{rmk:high_moment_est_low_moment}
	implies that \(m_1\langle f_0\rangle<\infty\).
	We set
	\[
	\mathfrak{m}_1 := m_1\langle f_0\rangle .
	\]
	By Proposition~\ref{prop:existence_viable_set}, choosing
	\[
	\mathfrak{m}_{\kappa}
	:=
	\max\!\left\{
	m_{\kappa}\langle f_0\rangle,\;
	2\,\mathcal{B}_{\kappa}(\mathfrak{m}_1)
	\right\}
	\ge
	2\,\mathcal{B}_{\kappa}(\mathfrak{m}_1),
	\]
	the set
	\[
	\mathcal{S}
	=
	\left\{
	f\in \ell^1(\mathbb{N}) :
	f\ge 0,\;
	m_1\langle f\rangle\le \mathfrak{m}_1,\;
	m_{\kappa}\langle f\rangle\le \mathfrak{m}_{\kappa}
	\right\}
	\]
	satisfies the tangency condition \eqref{eqn:sub_tangeny_condition}.
	
	Moreover, Proposition~\ref{prop:Holder_est} establishes the Hölder continuity of
	\(O\). In particular, we have
	\[
	C := \sup_{f\in\mathcal{S}} \|O[f]\|_{\ell^1} < \infty .
	\]
	
	Fix any \(\varepsilon\in(0,1)\). The Hölder continuity of \(O\) further implies
	that there exists \(\tau_0>0\) such that, for any \(g_1,g_2\in\mathcal{S}\) with
	\(\|g_1-g_2\|_{\ell^1} < (C+1)\tau_0\), we have
	\[
	\|O[g_1]-O[g_2]\|_{\ell^1} < \varepsilon .
	\]
	
	Then, for any \(f_1\in\mathcal{S}\), the tangency condition
	\eqref{eqn:sub_tangeny_condition} guarantees the existence of
	\(0<\tau<\tau_0\) and \(f_2\in\mathcal{S}\) such that
	\begin{equation}\label{eqn:tangency_middle_step}
		\frac{1}{\tau}\,
		\bigl\|f_1 - f_2 + \tau O[f_1]\bigr\|_{\ell^1}
		< \varepsilon .
	\end{equation}
	
	It follows from \eqref{eqn:tangency_middle_step} that
	\[
	\|f_1-f_2\|_{\ell^1}
	<
	\bigl(\|O[f_1]\|_{\ell^1}+\varepsilon\bigr)\tau
	<
	(C+1)\tau_0,
	\]
	and therefore
	\[
	\|O[f_1]-O[f_2]\|_{\ell^1}<\varepsilon .
	\]
	
	For \(t\in[0,\tau]\), define
	\[
	f(t):=f_1+\frac{t}{\tau}(f_2-f_1).
	\]
	A direct computation shows that, for all \(t\in[0,\tau]\),
	\begin{equation}\label{eqn:piecewise_construction}
		f(t)\in\mathcal{S},\qquad
		\left\|\frac{d}{dt}f(t)-O[f(t)]\right\|_{\ell^1}\le 2\varepsilon,
		\qquad
		f(t)\ \text{is Lipschitz with constant } C+1 .
	\end{equation}
	
	Fix \(T>0\). For any \(\varepsilon\in(0,1)\), we construct a piecewise affine
	approximate solution \(f^\varepsilon(t)\) on \([0,T]\) as follows.
	
	Starting from the initial data \(f_0\in\mathcal{S}\), the previous construction
	yields an affine function \(f^\varepsilon(t)\) on \([0,\tau_1]\) for some
	\(\tau_1>0\), such that \(f^\varepsilon(0)=f_0\) and
	\eqref{eqn:piecewise_construction} holds.
	By iterating this procedure, we extend \(f^\varepsilon\) to a larger interval
	\([0,\tau_2]\) with \(\tau_2>\tau_1\), while preserving
	\eqref{eqn:piecewise_construction}.
	To continue this process, we employ a transfinite induction argument.
	
	Suppose that for some countable ordinal \(\alpha\), a piecewise affine
	approximate solution \(f^\varepsilon\) has been constructed on
	\([0,\tau_\alpha]\) with \(\tau_\alpha\le T\).
	\begin{enumerate}
		\item If \(\alpha\) is a successor ordinal, we extend \(f^\varepsilon\) from
		\([0,\tau_{\alpha-1}]\) to \([0,\tau_\alpha]\) using the above construction.
		\item If \(\alpha\) is a limit ordinal, we define
		\(\tau_\alpha:=\sup_{\gamma<\alpha}\tau_\gamma\) and set
		\(f^\varepsilon(\tau_\alpha)
		=\lim_{t\to\tau_\alpha^-}f^\varepsilon(t)\).
		The limit exists due to the Lipschitz continuity of \(f^\varepsilon\).
	\end{enumerate}
	Since the sequence \((\tau_\alpha)\) is strictly increasing, this transfinite
	procedure terminates at some countable ordinal \(\alpha^*\) such that
	\(\tau_{\alpha^*}=T\).
	Thus, we obtain an approximate solution on \([0,T]\).
	By construction, \(f^\varepsilon\) is Lipschitz continuous and differentiable
	almost everywhere.
	
	Next, consider a sequence of approximate solutions \(f^\varepsilon\) with
	\(\varepsilon\to 0\).
	For any two approximate solutions \(f^{\varepsilon_1}\) and
	\(f^{\varepsilon_2}\) with \(\varepsilon_1,\varepsilon_2<\varepsilon/4\), the map
	\(t\mapsto\|f^{\varepsilon_1}(t)-f^{\varepsilon_2}(t)\|_{\ell^1}\) is
	differentiable for almost every \(t\in[0,T]\).
	For such \(t\), we compute
	\begin{equation}\label{eqn:one_side_lipschitz_estimate_argument}
		\begin{aligned}
			\frac{d}{dt}\|f^{\varepsilon_1}(t)-f^{\varepsilon_2}(t)\|_{\ell^1}
			&=
			\sum_{i=1}^\infty
			\bigl(\dot f_i^{\varepsilon_1}(t)-\dot f_i^{\varepsilon_2}(t)\bigr)
			\operatorname{sgn}\bigl(f_i^{\varepsilon_1}(t)-f_i^{\varepsilon_2}(t)\bigr) \\
			&\le
			\sum_{i=1}^\infty
			\bigl(O[f^{\varepsilon_1}](i)-O[f^{\varepsilon_2}](i)\bigr)
			\operatorname{sgn}\bigl(f_i^{\varepsilon_1}-f_i^{\varepsilon_2}\bigr)
			+2(\varepsilon_1+\varepsilon_2) \\
			&\le
			\sum_{i=1}^\infty
			|S[f^{\varepsilon_1}](i)-S[f^{\varepsilon_2}](i)|
			-\sum_{i=1}^\infty
			|f_i^{\varepsilon_1}-f_i^{\varepsilon_2}|\bfi^\delta
			+\varepsilon .
		\end{aligned}
	\end{equation}
	
	Applying estimate \eqref{eqn:Holder_estimate_pre_for_S}, we obtain
	\[
	\frac{d}{dt}
	\|f^{\varepsilon_1}(t)-f^{\varepsilon_2}(t)\|_{\ell^1}
	\le
	C(\mathfrak{m}_1,\mathfrak{m}_\kappa)
	\|f^{\varepsilon_1}(t)-f^{\varepsilon_2}(t)\|_{\ell^1}
	+\varepsilon .
	\]
	By Gr\"onwall’s inequality,
	\[
	\|f^{\varepsilon_1}(t)-f^{\varepsilon_2}(t)\|_{\ell^1}
	\le
	\varepsilon\,\frac{e^{Ct}}{C},
	\qquad
	C=C(\mathfrak{m}_1,\mathfrak{m}_\kappa).
	\]
	
	This shows that \(\{f^\varepsilon\}\) is a Cauchy sequence in
	\(C([0,T];\ell^1)\).
	Hence \(f^\varepsilon\to f\) uniformly on \([0,T]\).
	By continuity of \(O\), the limit \(f\) is a \(C^1\) solution of
	\eqref{eqn:discrete_pde}.
	Since \(T>0\) is arbitrary, the solution exists globally and satisfies
	\(f(t)\in\mathcal{S}\) for all \(t\in[0,\infty)\).

	\subsection{Proof of Theorem \ref{thm:Lipschitz_dependence}.}
	\label{SecTh2}
	Proceeding as in \eqref{eqn:one_side_lipschitz_estimate_argument}, we compute
	\begin{align*}
		\frac{d}{dt}\|f(t)-g(t)\|_{\ell^1}
		&=
		\sum_{i=1}^\infty
		\bigl(\dot f_i(t)-\dot g_i(t)\bigr)
		\operatorname{sgn}\bigl(f_i(t)-g_i(t)\bigr) \\
		&\le
		\sum_{i=1}^\infty
		\bigl(O[f(t)](i)-O[g(t)](i)\bigr)
		\operatorname{sgn}\bigl(f_i(t)-g_i(t)\bigr)\\
		&\le
		\sum_{i=1}^\infty
		|S[f(t)](i)-S[g(t)](i)|
		-\sum_{i=1}^\infty
		|f_i(t)-g_i(t)|\bfi^\delta.
	\end{align*}
	%	\begin{align*}
		%		\frac{d}{dt}\|f(t)-g(t)\|_{\ell^1}
		%		&=
		%		\lim_{h\to 0}
		%		\frac{\|(f(t)-g(t))+h(\dot f(t)-\dot g(t))\|_{\ell^1}
			%			-\|f(t)-g(t)\|_{\ell^1}}{h} \\
		%		&=
		%		\lim_{h\to 0}
		%		\frac{\|(f(t)-g(t))+h(O[f(t)]-O[g(t)])\|_{\ell^1}
			%			-\|f(t)-g(t)\|_{\ell^1}}{h} \\
		%		&=
		%		\lim_{h\to 0}
		%		\frac{1}{h}
		%		\sum_{i=1}^\infty
		%		\Bigl(
		%		\bigl|(f_i(t)-g_i(t))+h(O[f(t)](i)-O[g(t)](i))\bigr|
		%		-\bigl|f_i(t)-g_i(t)\bigr|
		%		\Bigr) \\
		%		&\le
		%		\sum_{i=1}^\infty
		%		\bigl(O[f(t)](i)-O[g(t)](i)\bigr)
		%		\operatorname{sgn}\!\bigl(f_i(t)-g_i(t)\bigr) \\
		%		&\le
		%		\sum_{i=1}^\infty
		%		\bigl|S[f(t)](i)-S[g(t)](i)\bigr|
		%		-
		%		\sum_{i=1}^\infty
		%		|f_i(t)-g_i(t)|\,\bfi^\delta .
		%	\end{align*}
	
	Applying estimate \eqref{eqn:Holder_estimate_pre_for_S}, we obtain
	\[
	\frac{d}{dt}\|f(t)-g(t)\|_{\ell^1}
	\le
	C\,
	\|f(t)-g(t)\|_{\ell^1},
	\]
	where the constant \(C>0\) depends only on
	\(\max\left\lbrace  m_1\langle f_0\rangle,m_1\langle g_0\rangle\right\rbrace\) and 
	\(\max\left\lbrace m_{\kappa}\langle f_0\rangle, m_{\kappa}\langle g_0\rangle\right\rbrace\).
	Therefore, by Gr\"onwall’s inequality, for all \(t\in[0,\infty)\),
	\[
	\|f(t)-g(t)\|_{\ell^1}
	\le
	e^{Ct}
	\|f_0-g_0\|_{\ell^1}.
	\]
	
	\subsection{Proof of Theorem \ref{thm:energy_decay}}		\label{SecTh3}
	
	We consider the truncated initial data
	\[
	f_0^{(N)}(i) := f_0(i)\,\mathbf{1}_{\{i\le N\}} .
	\]
	Let \(f^{(N)}(t)\) denote the solution obtained from Theorem~\ref{thm:GWP}
	with initial data \(f_0^{(N)}\).
	Since \(f_0^{(N)}\) has finite support, all moments of \(f^{(N)}(t)\) are finite
	for every \(t\in[0,\infty)\).
	Therefore, by the moment conservation property \eqref{eqn:moment_conservation_S},
	we have
	\[
	m_1\langle S[f^{(N)}(t)]\rangle = 0,
	\qquad \forall\, t\in[0,\infty).
	\]
	
	Consequently,
	\begin{align*}
		\frac{d}{dt} m_1\langle f^{(N)}(t)\rangle
		&=
		m_1\langle S[f^{(N)}(t)]\rangle
		-
		m_1\langle V[f^{(N)}(t)]\rangle \\
		&=
		-
		m_{\delta+1}\langle f^{(N)}(t)\rangle
		\le
		- h^{\delta}\, m_1\langle f^{(N)}(t)\rangle ,
	\end{align*}
	where the last inequality follows from Remark~\ref{rmk:high_moment_est_low_moment}.
	
	Applying Gr\"onwall’s inequality yields
	\[
	m_1\langle f^{(N)}(t)\rangle
	\le
	e^{-h^{\delta} t}\, m_1\langle f^{(N)}(0)\rangle
	\le
	e^{-h^{\delta} t}\, m_1\langle f_0\rangle .
	\]
	
	By construction, \(f_0^{(N)} \to f_0\) in \(\ell^1(\mathbb{N})\).
	Hence, Theorem~\ref{thm:Lipschitz_dependence} implies that
	\(f^{(N)}(t) \to f(t)\) in \(\ell^1(\mathbb{N})\) for each \(t\ge 0\).
	Finally, applying Fatou’s lemma, we obtain
	\[
	m_1\langle f(t)\rangle
	\le
	e^{-h^{\delta} t}\, m_1\langle f_0\rangle ,
	\qquad \forall\, t\in[0,\infty).
	\]

	\section{Creation and propagation of positivity and polynomial poments}
	
	In this section, we prove Theorems~\ref{thm:creation_of_positivity}
	and~\ref{thm:creation_of_polynomial_moments}.
	
	\begin{prop}[Propagation of Positivity]\label{prop:propagation_of_positivity}
		Let \(f(t)\) be the solution to
		\eqref{eqn:discrete_pde} with nonzero initial data \(f_0\ge 0\),
		\(f_0\not\equiv 0\), satisfying \(m_{\kappa}\langle f_0\rangle<\infty\) for some
		\(\kappa>\delta\).
		Then, for every index \(i\) such that \(f_0(i)>0\), we have
		\[
		f(t,i)>0,
		\qquad
		\forall\, t\in[0,\infty).
		\]
		
	\end{prop}
	\begin{proof}
		We recall the gain-loss decomposition of \(O[f]\) from
		\eqref{eqn:opt_gain_loss}:
		\[
		O[f]=O^{+}[f]-O^{-}[f].
		\]
		
		For the loss term, we have
		\begin{align*}
			O^{-}[f](i)
			&=
			2\sum_{j=1}^{\infty} K_1(i,j)\, f_i f_j
			+\sum_{j=1}^{i-1} K_2(i-j,j)\, f_i f_j
			+\sum_{j=1}^{i-1} K_3(i-j,j)\, f_i f_j
			+\gamma(i) f_i \\
			&\le
			\Bigl[
			2^{\beta_1}\bfi^{\alpha_1+\beta_1} m_{\alpha_1}\langle f\rangle
			+2^{\beta_1}\bfi^{\alpha_1} m_{\alpha_1+\beta_1}\langle f\rangle
			+\bfi^{\alpha_2+\beta_2} m_{\alpha_2}\langle f\rangle
			+\bfi^{\alpha_3+\beta_3} m_{\alpha_3}\langle f\rangle
			+\bfi^{\delta}
			\Bigr] f_i .
		\end{align*}
		
		By Theorem~\ref{thm:GWP}, both \(m_1\langle f(t)\rangle\) and
		\(m_{\kappa}\langle f(t)\rangle\) are uniformly bounded in time.
		Consequently, for each fixed \(i\), there exists a constant \(C_i>0\),
		depending only on \(m_1\langle f_0\rangle\) and
		\(m_{\kappa}\langle f_0\rangle\), such that
		\begin{equation}\label{eqn:propagation_estimation_loss_term}
			O^{-}[f](i)
			\le
			C\bigl(1+\bfi^{\delta}\bigr) f_i
			\le
			C_i f_i .
		\end{equation}
		
		Dropping the nonnegative gain term \(O^{+}[f]\), we obtain
		\[
		\frac{\partial f_i}{\partial t}
		= O[f](i)
		\ge -C_i f_i .
		\]
		Therefore, if \(f_0(i)>0\), it follows that
		\[
		f(t,i)>0,
		\qquad
		\forall\, t\in[0,\infty),
		\]
		which proves the propagation of positivity at each initially positive
		index.
		
	\end{proof}
	\subsection{Proof of Theorem \ref{thm:creation_of_positivity}}		\label{SecTh4}
	
	If \(f_0\equiv 0\), then the result follows trivially from the fact that
	\(\gcd(I)=0\).
	Hence, we only consider the case \(f_0\not\equiv 0\).
	
	We first show that if the initial datum \(f_0\) is positive at indices
	\(i\) and \(j\) (with \(i\) possibly equal to \(j\)), then the solution
	immediately generates positive values at the indices
	\(i+j\) and \(|i-j|\).
	
	Assume that \(f_0(i)>0\) and \(f_0(j)>0\).
	If \(f_0(i+j)>0\), then by Proposition~\ref{prop:propagation_of_positivity}
	we have \(f(t,i+j)>0\) for all \(t\in[0,\infty)\), and the claim follows.
	We therefore consider the case \(f_0(i+j)=0\).
	In this case, we estimate
	\[
	\frac{\partial f_{i+j}}{\partial t}
	\ge
	K_1(\bfi,\bfj)\,f_i f_j - O^{-}[f](i+j),
	\]
	which, by \eqref{eqn:propagation_estimation_loss_term}, implies
	\[
	\frac{\partial f_{i+j}}{\partial t}
	\ge
	K_1(\bfi,\bfj)\,f_i f_j - C_{i+j} f_{i+j}.
	\]
	
	Fix any \(\tau>0\).
	By Proposition~\ref{prop:propagation_of_positivity}, the components
	\(f_i(t)\) and \(f_j(t)\) remain strictly positive on \([0,\tau]\).
	Therefore, there exists a constant \(c>0\) such that
	\[
	\frac{\partial f_{i+j}(t)}{\partial t}
	\ge
	c - C_{i+j} f_{i+j}(t),
	\qquad
	\forall\, t\in[0,\tau].
	\]
	Solving this differential inequality yields
	\[
	f_{i+j}(t)
	\ge
	\frac{c}{C_{i+j}}\bigl(1-e^{-C_{i+j}t}\bigr)
	>0,
	\qquad
	\forall\, t\in(0,\tau].
	\]
	Combining this with Proposition~\ref{prop:propagation_of_positivity}, we
	conclude that \(f(t,i+j)>0\) for all \(t\in(0,\infty)\).
	
	Next, consider the index \(|i-j|\ge 1\).
	Without loss of generality, assume \(i<j\).
	As before, it suffices to consider the case \(f_0(j-i)=0\).
	In this case, the claim follows analogously by considering the estimate
	\[
	\frac{\partial f_{j-i}}{\partial t}
	\ge
	K_2(\bfj-\bfi,\bfi)\,f_i f_j - O^{-}[f](j-i),
	\]
	and proceeding as above.

	We now inductively define a sequence of index sets by
	\[
	I_0 := I, 
	\qquad 
	I_{n+1} := \{\, i+j,\ |i-j| \in \mathbb{N} : i,j \in I_n \,\} \cup I_n .
	\]
	
	By the argument above and an induction on \(n\), it follows that
	\(f(t,i)>0\) for all \(t\in(0,\infty)\) and all \(i\in I_n\), for every
	\(n\ge 0\).
	
	Next, consider the ideal generated by \(I\) in \(\mathbb{Z}\), denoted by
	\(\langle I\rangle\).
	By construction,
	\[
	\bigcup_{n=0}^\infty I_n = \langle I\rangle \cap \mathbb{N}.
	\]
	Since \(\mathbb{Z}\) is a principal ideal domain, we have
	\(\langle I\rangle = \langle \gcd(I)\rangle\).
	Consequently,
	\[
	\bigcup_{n=0}^\infty I_n
	= \langle \gcd(I)\rangle \cap \mathbb{N}
	= \gcd(I)\,\mathbb{N}.
	\]
	Therefore, the solution \(f(t)\) is strictly positive for all
	\(t\in(0,\infty)\) and all indices \(i \in \gcd(I)\mathbb{N}\).
	
	We now show that \(f(t,i)=0\) for all \(i \notin \gcd(I)\mathbb{N}\).
	For equation \eqref{eqn:discrete_pde}, suppose that
	\(f_k(t)\equiv 0\) for all \(k \notin \gcd(I)\mathbb{N}\).
	For indices of the form \(k=\gcd(I)\,i\) with \(i\in\mathbb{N}\), define
	\(g_i(t) := f_{\gcd(I)i}(t)\).
	A direct computation shows that \(g_i\) satisfies
	\begin{align*}
		\frac{\partial g_i}{\partial t}
		&=
		\gcd(I)^{2\alpha_1+\beta_1}
		\Bigg[
		\sum_{j=1}^{i-1} K_1(\bfj,\bfi-\bfj) g_j g_{i-j}
		-2\sum_{j=1}^{\infty} K_1(\bfi,\bfj) g_i g_j
		\Bigg] \\
		&\quad
		+\gcd(I)^{2\alpha_2+\beta_2}
		\Bigg[
		-\sum_{j=1}^{i-1} K_2(\bfi-\bfj,\bfj) g_i g_j
		+\sum_{j=i+1}^{\infty} K_2(\bfj-\bfi,\bfi) g_i g_j
		+\sum_{j=i+1}^{\infty} K_2(\bfi,\bfj-\bfi) g_j g_{j-i}
		\Bigg] \\
		&\quad
		+\gcd(I)^{2\alpha_3+\beta_3}
		\Bigg[
		-\sum_{j=1}^{i-1} K_3(\bfi-\bfj,\bfj) g_i g_j
		+\sum_{j=i+1}^{\infty} K_3(\bfj-\bfi,\bfi) g_i g_j
		+\sum_{j=i+1}^{\infty} K_3(\bfi,\bfj-\bfi) g_j g_{j-i}
		\Bigg] \\
		&\quad
		-\gcd(I)^{\delta}\,\gamma(\bfi)\,g_i,
		\qquad
		g_i(0)=f_{\gcd(I)i}(0).
	\end{align*}
	This system has the same structure as \eqref{eqn:discrete_pde}, up to
	multiplicative constants.
	Hence, the global well-posedness result established in
	Theorem~\ref{thm:GWP} applies verbatim, and the system for \(g_i\) admits a
	unique global solution.
	
	Define
	\[
	\tilde f(t,i) :=
	\begin{cases}
		g_i(t), & i\in \gcd(I)\mathbb{N},\\
		0, & i\notin \gcd(I)\mathbb{N}.
	\end{cases}
	\]
	By construction, \(\tilde f\) is a solution of \eqref{eqn:discrete_pde}
	with initial datum \(f_0\).
	By uniqueness, we conclude that \(\tilde f = f\).
	In particular,
	\[
	f(t,i)=0
	\qquad
	\text{for all } t\in(0,\infty) \text{ and all } i\notin \gcd(I)\mathbb{N}.
	\]

	\begin{remark} This proposition shows that when $\gcd(I)\neq 1$, the dynamics of the equation are confined to the sublattice $\gcd(I)\mathbb{N}$. In particular, the original equation reduces to an equivalent system posed on $\gcd(I)\mathbb{N}$.
		
	\end{remark}
	\subsection{Proof of Theorem \ref{thm:creation_of_polynomial_moments}}		\label{SecTh5}
	For the given initial datum $f_0$, we introduce the truncated initial data
	\[
	f_0^{(N)}(i):=f_0(i)\mathbf{1}_{i\le N}.
	\]
	Clearly,
	\[
	m_1\left\langle f_0^{(N)} \right\rangle \le \mathfrak{m}_1,
	\qquad
	\mathfrak{m}_1:=m_1\left\langle f_0 \right\rangle.
	\]
	
	By Theorem~\ref{thm:GWP}, for each $N\in\mathbb{N}$ there exists a global solution
	$f^{(N)}(t)$ associated with the initial datum $f_0^{(N)}$, and all its moments
	are finite:
	\[
	m_k\left\langle f^{(N)}(t)\right\rangle<\infty,
	\qquad \forall\, t\in[0,\infty),\ \forall\, k\ge0.
	\]
	
	Since $f_0\not\equiv 0$, we have $f_0^{(N)}\not\equiv 0$ for $N$ sufficiently large. By
	Proposition~\ref{prop:propagation_of_positivity}, it follows that
	\[
	m_1\left\langle f^{(N)}(t)\right\rangle>0,
	\qquad \forall\, t\in[0,\infty).
	\]
	Proceeding as in the derivation of \eqref{eqn:diff_inequality_for_moments}, we obtain
	\begin{align*}
		\frac{d}{dt} m_k\left\langle f^{(N)}(t)\right\rangle
		&= m_k\left\langle O[f^{(N)}](t)\right\rangle \\
		&\le \mathcal{C}_k\, m_1\left\langle f^{(N)}(t)\right\rangle^{1+\frac{\delta+k-1}{\delta-2\alpha_1-\beta_1+1}}
		-\frac{1}{2}\,
		m_1\left\langle f^{(N)}(t)\right\rangle^{-\frac{\delta}{k-1}}
		m_k\left\langle f^{(N)}(t)\right\rangle^{\frac{\delta+k-1}{k-1}} \\
		&\le \mathcal{C}_k\,\mathfrak{m}_1^{1+\frac{\delta+k-1}{\delta-2\alpha_1-\beta_1+1}}
		-\frac{1}{2}\,\mathfrak{m}_1^{-\frac{\delta}{k-1}}
		m_k\left\langle f^{(N)}(t)\right\rangle^{\frac{\delta+k-1}{k-1}} .
	\end{align*}
	
	Setting
	\[
	y(t):=m_k\left\langle f^{(N)}(t)\right\rangle,
	\]
	we arrive at the differential inequality
	\[
	\frac{dy}{dt}
	\le
	\mathcal{C}_k\,\mathfrak{m}_1^{1+\frac{\delta+k-1}{\delta-2\alpha_1-\beta_1+1}}
	-\frac{1}{2}\,\mathfrak{m}_1^{-\frac{\delta}{k-1}}
	y^{\frac{\delta+k-1}{k-1}}.
	\]
	
	A standard comparison argument shows that a supersolution is given by
	\begin{align*}
		Y(t)
		&=
		\frac{m_k\left\langle f_0^{(N)}\right\rangle}
		{\left(1+\frac{1}{2}\mathfrak{m}_1^{-\frac{\delta}{k-1}}
			\frac{\delta}{k-1}
			\left[m_k\left\langle f_0^{(N)}\right\rangle\right]^{\frac{\delta}{k-1}} t
			\right)^{\frac{k-1}{\delta}}}
		+
		\left(
		2\mathcal{C}_k\,\mathfrak{m}_1^{1+\frac{\delta+k-1}{\delta-2\alpha_1-\beta_1+1}
			+\frac{\delta}{k-1}}
		\right)^{\frac{k-1}{\delta+k-1}} \\
		&=
		\frac{1}{\left(1+\frac{1}{2}\mathfrak{m}_1^{-\frac{\delta}{k-1}}
			\frac{\delta}{k-1} t\right)^{\frac{k-1}{\delta}}}
		+\mathcal{B}_k(\mathfrak{m}_1) \\
		&<
		\left(\frac{2(k-1)}{\delta t}\right)^{\frac{k-1}{\delta}}\mathfrak{m}_1
		+\mathcal{B}_k(\mathfrak{m}_1).
	\end{align*}
	
	Consequently, for all $t\in(0,\infty)$ and all $N$ sufficiently large,
	\begin{equation}\label{eqn:creation_of_poly_mom_cut_off_version}
		m_k\left\langle f^{(N)}(t)\right\rangle
		<
		\left(\frac{2(k-1)}{\delta t}\right)^{\frac{k-1}{\delta}}\mathfrak{m}_1
		+\mathcal{B}_k(\mathfrak{m}_1).
	\end{equation}
	
	Finally, since $f_0^{(N)}\to f_0$ in $\ell^1(\mathbb{N})$, Theorem~\ref{thm:Lipschitz_dependence}
	implies that for any fixed $t\in(0,\infty)$,
	\[
	f^{(N)}(t)\to f(t)\quad \text{in }\ell^1(\mathbb{N}).
	\]
	By Fatou’s lemma and the uniform bound \eqref{eqn:creation_of_poly_mom_cut_off_version},
	we conclude that
	\[
	m_k\left\langle f(t)\right\rangle
	\le
	\left(\frac{2(k-1)}{\delta t}\right)^{\frac{k-1}{\delta}}\mathfrak{m}_1
	+\mathcal{B}_k(\mathfrak{m}_1),
	\qquad \forall\, t\in(0,\infty).
	\]

	\section{Creation and propogation of Mittag-Leffler moments}
	
	In this section, we prove Theorems~\ref{thm:propagation_of_m_l_tails}
	and~\ref{thm:creation_of_exp_moments}.
	
	\begin{definition}
		For a function $f:\mathbb{N}\to\mathbb{R}_{\ge 0}$, and for any
		$n\in\mathbb{N}$, $\lambda>0$, $\rho>0$, and $a\in[1,\infty)$, we define
		the truncated Mittag--Leffler moments and their shifted counterparts by
		\[
		\mathcal{E}^n_a(\lambda)\langle f\rangle
		:=\sum_{k=1}^n \frac{m_k\langle f\rangle\,\lambda^{ak}}{\Gamma(ak+1)},
		\qquad
		\mathcal{E}^n_{a,\rho}(\lambda)\langle f\rangle
		:=\sum_{k=1}^n \frac{m_{k+\rho}\langle f\rangle\,\lambda^{ak}}{\Gamma(ak+1)}.
		\]
		
	\end{definition}
	\begin{lem}
		For any $n\in\mathbb{N}\cup\{\infty\}$, any $0<\rho_1<\rho_2$, any $\lambda>0$,
		and any $a\in[1,\infty)$, the following interpolation inequality holds:
		\begin{equation}\label{eqn:m_l_interpolation_prod}
			\mathcal{E}^n_{a,\rho_1}(\lambda)\langle f\rangle
			\le
			\left[
			\mathcal{E}^n_{a}(\lambda)\langle f\rangle
			\right]^{\frac{\rho_2-\rho_1}{\rho_2}}
			\left[
			\mathcal{E}^n_{a,\rho_2}(\lambda)\langle f\rangle
			\right]^{\frac{\rho_1}{\rho_2}}.
		\end{equation}
		
		%	\begin{enumerate}
			%		\item 
			
			%		\item \begin{equation}\label{eqn:m_l_interpolation_sum}
				%			\calE^n_{a,\rho_1}(\lambda,t)\le C_\varepsilon\calE^n_{a}(\lambda,t)+\varepsilon\calE^n_{a,\rho_2}(\lambda,t).
				%		\end{equation}
			%	\end{enumerate}
	\end{lem}
	\begin{proof}
		
		By the moment interpolation inequality \eqref{eqn:moment_interpolation}, we have
		\[
		m_{k+\rho_1}\langle f\rangle
		\le
		m_k\langle f\rangle^{\frac{\rho_2-\rho_1}{\rho_2}}
		\, m_{k+\rho_2}\langle f\rangle^{\frac{\rho_1}{\rho_2}} .
		\]
		Therefore,
		\begin{align*}
			\mathcal{E}^n_{a,\rho_1}(\lambda)\langle f\rangle
			&=
			\sum_{k=1}^n
			\frac{m_{k+\rho_1}\langle f\rangle\,\lambda^{ak}}{\Gamma(ak+1)} \\
			&\le
			\sum_{k=1}^n
			\frac{
				m_k\langle f\rangle^{\frac{\rho_2-\rho_1}{\rho_2}}
				\, m_{k+\rho_2}\langle f\rangle^{\frac{\rho_1}{\rho_2}}
				\, \lambda^{ak}
			}{
				\Gamma(ak+1)
			} \\
			&\le
			\left(
			\sum_{k=1}^n
			\frac{m_k\langle f\rangle\,\lambda^{ak}}{\Gamma(ak+1)}
			\right)^{\frac{\rho_2-\rho_1}{\rho_2}}
			\left(
			\sum_{k=1}^n
			\frac{m_{k+\rho_2}\langle f\rangle\,\lambda^{ak}}{\Gamma(ak+1)}
			\right)^{\frac{\rho_1}{\rho_2}} \\
			&=
			\left[
			\mathcal{E}^n_a(\lambda)\langle f\rangle
			\right]^{\frac{\rho_2-\rho_1}{\rho_2}}
			\left[
			\mathcal{E}^n_{a,\rho_2}(\lambda)\langle f\rangle
			\right]^{\frac{\rho_1}{\rho_2}} .
		\end{align*}
		
		The same argument applies verbatim in the case $n=\infty$, which completes the proof of
		\eqref{eqn:m_l_interpolation_prod}.
		
		%	By Young's inequality, for any \(\varepsilon>0\), there exists \(C_\varepsilon\) s.t.
		%	\[\left|\bfi\right|^{\rho_1}<C_\varepsilon +\varepsilon \left|\bfi\right|^{\rho_2}\quad \forall i,h,\]
		%	hence
		%	\begin{align*}
			%		\mathcal{E}^n_{a,\rho_1}(\lambda,t)&=\sum_{k=1}^n\frac{m_{k+\rho_1}(t)\lambda^{ak}}{\Gamma(ak+1)}=\sum_{k=1}^n\frac{\sum_{i=1}^\infty f_i(\bfi)^{k+\rho_1}\lambda^{ak}}{\Gamma(ak+1)}\\
			%		&\le C_\varepsilon \sum_{k=1}^n\frac{\sum_{i=1}^\infty f_i(\bfi)^{k}\lambda^{ak}}{\Gamma(ak+1)}+\varepsilon\sum_{k=1}^n\frac{\sum_{i=1}^\infty f_i(\bfi)^{k+\rho_2}\lambda^{ak}}{\Gamma(ak+1)}\\
			%		&=C_\varepsilon\calE^n_{a}(\lambda,t)+\varepsilon\calE^n_{a,\rho_2}(\lambda,t).
			%	\end{align*}
	\end{proof}
	\begin{prop}
		For $f:\mathbb{N}\to\mathbb{R}_{\ge 0}$ be such that all its moments are finite. For any integer $k\ge 1$, the following
		combinatorial moment estimate holds:
		\begin{equation}\label{eqn:combin_estimate_for_moment}
			m_k\langle S[f]\rangle
			\le
			2^{\beta_1+1}
			\sum_{l=1}^{k-1}
			\binom{k}{l}\,
			m_{\alpha_1+l}\langle f\rangle\,
			m_{\alpha_1+\beta_1+k-l}\langle f\rangle .
		\end{equation}
		
	\end{prop}
	\begin{proof}
		By \eqref{eqn:moment_conservation_S}, the estimate
		\eqref{eqn:combin_estimate_for_moment} is trivial for $k=1$.
		We therefore assume $k\ge 2$. Using the weak formulation and the binomial
		expansion of $(\bfi+\bfj)^k$, we obtain
		\begin{align*}
			m_k\langle S[f]\rangle
			&\le
			\sum_{l=1}^{k-1} \binom{k}{l}
			\sum_{i,j=1}^\infty
			K_1(\bfi,\bfj)\, f_i f_j\, \bfi^{\,l}\bfj^{\,k-l}.
		\end{align*}
		Splitting the sum into the regions $i\le j$ and $j\le i$, and using the bound
		$(\bfi+\bfj)^{\beta_1}\le 2^{\beta_1}\max\{\bfi^{\beta_1},\bfj^{\beta_1}\}$,
		we find
		\begin{align*}
			m_k\langle S[f]\rangle
			&\le
			\sum_{l=1}^{k-1} \binom{k}{l}
			\Bigg(
			\sum_{i\le j}
			\bfi^{\alpha_1}\bfj^{\alpha_1}(\bfi+\bfj)^{\beta_1}
			f_i f_j\, \bfi^{\,l}\bfj^{\,k-l}  \\
			&\hspace{3.5cm}
			+
			\sum_{j\le i}
			\bfi^{\alpha_1}\bfj^{\alpha_1}(\bfi+\bfj)^{\beta_1}
			f_i f_j\, \bfi^{\,l}\bfj^{\,k-l}
			\Bigg) \\
			&\le
			2^{\beta_1}
			\sum_{l=1}^{k-1} \binom{k}{l}
			\Bigg(
			\sum_{i\le j}
			\bfi^{\alpha_1+l}\bfj^{\alpha_1+\beta_1+k-l} f_i f_j
			+
			\sum_{j\le i}
			\bfi^{\alpha_1+\beta_1+l}\bfj^{\alpha_1+k-l} f_i f_j
			\Bigg) \\
			&\le
			2^{\beta_1+1}
			\sum_{l=1}^{k-1} \binom{k}{l}
			\sum_{i\le j}
			\bfi^{\alpha_1+l}\bfj^{\alpha_1+\beta_1+k-l} f_i f_j .
		\end{align*}
		Finally, by symmetry and the definition of moments, we conclude that
		\[
		m_k\langle S[f]\rangle
		\le
		2^{\beta_1+1}
		\sum_{l=1}^{k-1}
		\binom{k}{l}\,
		m_{\alpha_1+l}\langle f\rangle\,
		m_{\alpha_1+\beta_1+k-l}\langle f\rangle,
		\]
		which proves \eqref{eqn:combin_estimate_for_moment}.
		
	\end{proof}
	\begin{lem}
		For any integers $k\ge 2$ and $1\le l<k$, and for any $a\ge 1$, the following estimate holds:
		\begin{equation}\label{eqn:combin_estimate_gamma}
			\binom{k}{l}\,
			\frac{\Gamma(al+1)\Gamma(a(k-l)+1)}{\Gamma(ak+1)}
			\le 2\sqrt{a}.
		\end{equation}
		
	\end{lem}
	\begin{proof}
		For any $1\le l<k$, we write
		\[
		\binom{k}{l}\frac{\Gamma(al+1)\Gamma(a(k-l)+1)}{\Gamma(ak+1)}
		=
		\frac{\Gamma(k+1)}{\Gamma(l+1)\Gamma(k-l+1)}
		\frac{\Gamma(al+1)\Gamma(a(k-l)+1)}{\Gamma(ak+1)} .
		\]
		
		For $a\ge 1$, we recall Stirling’s formula: for $x>0$,
		\[
		\Gamma(x)
		=
		\left(\frac{x}{e}\right)^x
		\sqrt{\frac{2\pi}{x}}
		\,e^{\frac{\theta_x}{12x}},
		\qquad
		\Gamma(x+1)
		=
		\left(\frac{x}{e}\right)^x
		\sqrt{2\pi x}
		\,e^{\frac{\theta_x}{12x}},
		\]
		for some $\theta_x\in(0,1)$.
		
		Applying Stirling’s formula to each Gamma function yields
		\begin{align*}
			&\binom{k}{l}\frac{\Gamma(al+1)\Gamma(a(k-l)+1)}{\Gamma(ak+1)} \\
			=&
			\frac{\left(\frac{k}{e}\right)^k\sqrt{2\pi k}\,e^{\frac{\theta_k}{12k}}}
			{\left(\frac{l}{e}\right)^l\sqrt{2\pi l}\,e^{\frac{\theta_l}{12l}}
				\left(\frac{k-l}{e}\right)^{k-l}\sqrt{2\pi(k-l)}\,e^{\frac{\theta_{k-l}}{12(k-l)}}}
			\\
			&\quad\times
			\frac{
				\left(\frac{al}{e}\right)^{al}\sqrt{2\pi al}\,e^{\frac{\theta_{al}}{12al}}
				\left(\frac{a(k-l)}{e}\right)^{a(k-l)}\sqrt{2\pi a(k-l)}\,e^{\frac{\theta_{a(k-l)}}{12a(k-l)}}
			}
			{
				\left(\frac{ak}{e}\right)^{ak}\sqrt{2\pi ak}\,e^{\frac{\theta_{ak}}{12ak}}
			}.
		\end{align*}
		Collecting constants and using $\theta_x\le 1$, we obtain
		\[
		\binom{k}{l}\frac{\Gamma(al+1)\Gamma(a(k-l)+1)}{\Gamma(ak+1)}
		\le
		e^{1/4}\sqrt{a}\,
		\frac{k^k}{l^l(k-l)^{k-l}}
		\frac{l^{al}(k-l)^{a(k-l)}}{k^{ak}}
		=
		e^{1/4}\sqrt{a}\,
		\frac{l^{(a-1)l}(k-l)^{(a-1)(k-l)}}{k^{(a-1)k}}.
		\]
		
		Define
		\[
		f(x,y):=x^{(a-1)x}y^{(a-1)y},
		\qquad
		\ln f(x,y)=(a-1)(x\ln x+y\ln y).
		\]
		Since $x\mapsto x\ln x$ is convex on $(0,\infty)$, for $1\le l<k$ we have
		\[
		l\ln l+(k-l)\ln(k-l)
		\le
		1\ln 1+(k-1)\ln(k-1)
		\le
		k\ln k.
		\]
		Consequently,
		\[
		\ln f(l,k-l)
		\le
		(a-1)k\ln k,
		\qquad
		f(l,k-l)\le k^{(a-1)k}.
		\]
		
		Therefore, for all $a\ge 1$ and $1\le l<k$,
		\[
		\binom{k}{l}\frac{\Gamma(al+1)\Gamma(a(k-l)+1)}{\Gamma(ak+1)}
		\le
		e^{1/4}\sqrt{a}
		\le
		2\sqrt{a},
		\]
		which proves \eqref{eqn:combin_estimate_gamma}.

	\end{proof}
	\begin{prop}For any $n\in\mathbb{N}\cup\{\infty\}$, any $\lambda>0$, any $a\ge 1$, and any
		$f:\mathbb{N}\to\mathbb{R}_{\ge 0}$, the following estimate holds:
		\begin{equation}\label{eqn:combin_estimate_for_moment_Gamma}
			\sum_{k=1}^n \sum_{l=1}^{k-1}
			\binom{k}{l}
			\frac{
				m_{\alpha_1+l}\langle f\rangle\,
				m_{\alpha_1+\beta_1+k-l}\langle f\rangle\,
				\lambda^{ak}
			}{
				\Gamma(ak+1)
			}
			\le
			2\sqrt{a}\,
			\mathcal{E}_{a,\alpha_1}^n(\lambda)\langle f\rangle\,
			\mathcal{E}_{a,\alpha_1+\beta_1}^n(\lambda)\langle f\rangle.
		\end{equation}
		
	\end{prop}
	\begin{proof}
		It follows from \eqref{eqn:combin_estimate_gamma} that
		\begin{equation}\label{eqn:middle_step_est_for_combin}
			\begin{aligned}
				&\sum_{k=1}^n \sum_{l=1}^{k-1}
				\binom{k}{l}
				\frac{
					m_{\alpha_1+l}\langle f\rangle\,
					m_{\alpha_1+\beta_1+k-l}\langle f\rangle\,
					\lambda^{ak}
				}{
					\Gamma(ak+1)
				} \\
				\le{}&
				2\sqrt{a}
				\sum_{k=1}^n \sum_{l=1}^{k-1}
				\frac{m_{\alpha_1+l}\langle f\rangle\,\lambda^{al}}{\Gamma(al+1)}
				\frac{m_{\alpha_1+\beta_1+k-l}\langle f\rangle\,\lambda^{a(k-l)}}{\Gamma(a(k-l)+1)} \\
				\le{}&
				2\sqrt{a}
				\left(
				\sum_{i=1}^n
				\frac{m_{\alpha_1+i}\langle f\rangle\,\lambda^{ai}}{\Gamma(ai+1)}
				\right)
				\left(
				\sum_{j=1}^n
				\frac{m_{\alpha_1+\beta_1+j}\langle f\rangle\,\lambda^{aj}}{\Gamma(aj+1)}
				\right) \\
				={}&
				2\sqrt{a}\,
				\mathcal{E}_{a,\alpha_1}^n(\lambda)\langle f\rangle\,
				\mathcal{E}_{a,\alpha_1+\beta_1}^n(\lambda)\langle f\rangle .
			\end{aligned}
		\end{equation}
		
	\end{proof}
	\subsection{Proof of Theorem \ref{thm:propagation_of_m_l_tails}}		\label{SecTh6}
	
	We aim to show that, for $\lambda>0$ sufficiently small, the set
	\[
	\mathcal{T}
	:=
	\Bigl\{
	f\in \ell^1(\mathbb{N}) :
	f\ge 0,\ 
	m_1\langle f\rangle\le \mathfrak{m}_1,\ 
	m_{\kappa}\langle f\rangle\le \mathfrak{m}_{\kappa},\ 
	\mathcal{E}_a^\infty(\lambda)\langle f\rangle\le 1
	\Bigr\}
	\]
	satisfies the sub-tangency condition
	\begin{equation}\label{eqn:sub_tangency_for_mittag}
		\forall f\in \mathcal{T},\qquad
		\liminf_{\tau\to 0^+}
		\frac{1}{\tau}\,
		\mathrm{dist}_{\ell^1}(f+\tau O[f],\mathcal{T})
		=0 .
	\end{equation}
	
	Recalling Proposition~\ref{prop:existence_viable_set}, for
	$\mathfrak{m}_1:=m_1\langle f_0\rangle$, we accordingly fix \(\mathfrak{m}_{\kappa}
	:=
	\max\bigl\{m_{\kappa}\langle f_0\rangle,\,2\mathcal{B}_k\left(\mathfrak{m}_1\right)
	\bigr\}.\) Then the set
	\[
	\mathcal{S}
	=
	\Bigl\{
	f\in \ell^1(\mathbb{N}) :
	f\ge 0,\ 
	m_1\langle f\rangle\le \mathfrak{m}_1,\ 
	m_{\kappa}\langle f\rangle\le \mathfrak{m}_{\kappa}
	\Bigr\}
	\]
	satisfies the sub-tangency condition \eqref{eqn:sub_tangeny_condition}.
	
	As in the proof of Proposition~\ref{prop:existence_viable_set}, it suffices
	to consider nonzero functions $f\in\mathcal{T}$. There exists $N_0$ such that,
	for all $N>N_0$ and for $\tau>0$ sufficiently small (depending on
	$\mathfrak{m}_1,\mathfrak{m}_{\kappa},h$, and $N$), the truncated function
	$f_N(i):=f_i\,\mathbf{1}_{\{i\le N\}}$ is nonzero, and the perturbation
	\[
	w_N:=f+\tau O[f_N]
	\]
	satisfies
	\[
	w_N\ge 0,\qquad
	m_1\langle w_N\rangle\le \mathfrak{m}_1,\qquad
	m_{\kappa}\langle w_N\rangle\le \mathfrak{m}_{\kappa}.
	\]
	
	Using the combinatorial estimate \eqref{eqn:combin_estimate_for_moment}, we
	obtain
	\begin{align*}
		&\mathcal{E}_a^\infty(\lambda)\langle O[f_N]\rangle
		=
		\sum_{i=1}^\infty O[f_N](i)\,\mathcal{E}_a(\lambda^a\bfi)
		=
		\sum_{k=1}^\infty
		\frac{m_k\langle O[f_N]\rangle\,\lambda^{ak}}{\Gamma(ak+1)} \\
		&\qquad\le
		\sum_{k=1}^\infty
		\frac{m_k\langle S[f_N]\rangle\,\lambda^{ak}}{\Gamma(ak+1)}
		-
		\sum_{k=1}^\infty
		\frac{m_k\langle V[f_N]\rangle\,\lambda^{ak}}{\Gamma(ak+1)} \\
		&\qquad\le
		2^{\beta_1+1}
		\sum_{k=1}^\infty
		\sum_{l=1}^{k-1}
		\binom{k}{l}
		\frac{
			m_{\alpha_1+l}\langle f_N\rangle\,
			m_{\alpha_1+\beta_1+k-l}\langle f_N\rangle\,
			\lambda^{ak}
		}{
			\Gamma(ak+1)
		}
		-
		\sum_{k=1}^\infty
		\frac{m_{\delta+k}\langle f_N\rangle\,\lambda^{ak}}{\Gamma(ak+1)} .
	\end{align*}

	It follows from \eqref{eqn:combin_estimate_for_moment_Gamma} that
	\begin{align*}
		\mathcal{E}_{a}^\infty(\lambda)\langle O[f_N]\rangle
		&\le
		2^{\beta_1+2}\sqrt{a}\,
		\mathcal{E}_{a,\alpha_1}^\infty(\lambda)\langle f_N\rangle\,
		\mathcal{E}_{a,\alpha_1+\beta_1}^\infty(\lambda)\langle f_N\rangle
		-
		\mathcal{E}_{a,\delta}^\infty(\lambda)\langle f_N\rangle .
	\end{align*}
	Note that
	\[
	\mathcal{E}_{a,\delta}^\infty(\lambda)\langle f_N\rangle
	=
	\sum_{i=1}^\infty f_N(i)\,\bfi^\delta\,\mathcal{E}_a(\lambda^a\bfi)
	=
	\sum_{i=1}^N f_N(i)\,\bfi^\delta\,\mathcal{E}_a(\lambda^a\bfi)
	<\infty,
	\]
	so all terms above are well defined and the computation is justified.
	
	Recalling the interpolation inequality \eqref{eqn:m_l_interpolation_prod}, we have
	\[
	\mathcal{E}_{a,\alpha_1}^\infty(\lambda)
	\le
	\bigl(\mathcal{E}_{a}^\infty(\lambda)\bigr)^{\frac{\delta-\alpha_1}{\delta}}
	\bigl(\mathcal{E}_{a,\delta}^\infty(\lambda)\bigr)^{\frac{\alpha_1}{\delta}},
	\qquad
	\mathcal{E}_{a,\alpha_1+\beta_1}^\infty(\lambda)
	\le
	\bigl(\mathcal{E}_{a}^\infty(\lambda)\bigr)^{\frac{\delta-\alpha_1-\beta_1}{\delta}}
	\bigl(\mathcal{E}_{a,\delta}^\infty(\lambda)\bigr)^{\frac{\alpha_1+\beta_1}{\delta}}.
	\]
	Since $f\in\mathcal{T}$, we have
	\[
	\mathcal{E}_a^\infty(\lambda)\langle f_N\rangle
	\le
	\mathcal{E}_a^\infty(\lambda)\langle f\rangle
	\le 1.
	\]
	Therefore,
	\begin{equation}\label{eqn:middle_step_propaga_m_l}
		\begin{aligned}
			\mathcal{E}_{a}^\infty(\lambda)\langle O[f_N]\rangle
			&\le
			2^{\beta_1+2}\sqrt{a}\,
			\bigl(\mathcal{E}_{a,\delta}^\infty(\lambda)\langle f_N\rangle\bigr)^{\frac{2\alpha_1+\beta_1}{\delta}}
			-
			\mathcal{E}_{a,\delta}^\infty(\lambda)\langle f_N\rangle \\
			&\le
			\frac{C}{2}\,
			(\sqrt{a})^{\frac{\delta}{\delta-2\alpha_1-\beta_1}}
			-
			\frac{1}{2}\,
			\mathcal{E}_{a,\delta}^\infty(\lambda)\langle f_N\rangle ,
		\end{aligned}
	\end{equation}
	where the second inequality follows from Young’s inequality, using the assumption
	$\delta>2\alpha_1+\beta_1$. The constant $C>0$ depends only on
	$\alpha_1,\beta_1$, and $\delta$.
	
	Finally, recalling the definition of $\mathcal{E}_{a,\delta}^\infty(\lambda)$, we observe that
	\[
	\mathcal{E}_{a,\delta}^\infty(\lambda)\langle f_N\rangle
	=
	\sum_{k=1}^\infty
	\frac{m_{k+\delta}\langle f_N\rangle\,\lambda^{ak}}{\Gamma(ak+1)}
	\ge
	\sum_{k=1}^\infty
	\sum_{\{\,i:\,\bfi\ge \lambda^{-a}\,\}}
	f_N(i)\,
	\frac{\bfi^{k+\delta}\lambda^{ak}}{\Gamma(ak+1)}.
	\]

	Note that
	\[
	\bfi \ge \lambda^{-a} \;\Longrightarrow\; \bfi^{k+\delta}\ge \frac{\bfi^{k}}{\lambda^{\delta a}},
	\qquad
	\bfi < \lambda^{-a} \;\Longrightarrow\; \bfi^{k}\le \frac{\bfi}{\lambda^{(k-1)a}} .
	\]
	It follows that
	\begin{align*}
		&\mathcal{E}_{a,\delta}^\infty(\lambda)\langle f_N\rangle
		\ge
		\frac{1}{\lambda^{\delta a}}
		\sum_{k=1}^\infty
		\sum_{\{\,i:\,\bfi\ge \lambda^{-a}\,\}}
		f_N(i)\,
		\frac{\bfi^{k}\lambda^{ak}}{\Gamma(ak+1)} \\
		&\qquad=
		\frac{1}{\lambda^{\delta a}}
		\left(
		\sum_{k=1}^\infty \sum_{i=1}^\infty
		f_N(i)\frac{\bfi^{k}\lambda^{ak}}{\Gamma(ak+1)}
		-
		\sum_{k=1}^\infty
		\sum_{\{\,i:\,\bfi<\lambda^{-a}\,\}}
		f_N(i)\frac{\bfi^{k}\lambda^{ak}}{\Gamma(ak+1)}
		\right) \\
		&\qquad\ge
		\frac{1}{\lambda^{\delta a}}
		\left(
		\mathcal{E}_a^\infty(\lambda)\langle f_N\rangle
		-
		\sum_{k=1}^\infty
		\sum_{i=1}^\infty
		f_N(i)\,
		\frac{\bfi}{\lambda^{(k-1)a}}
		\frac{\lambda^{ak}}{\Gamma(ak+1)}
		\right) \\
		&\qquad=
		\frac{1}{\lambda^{\delta a}}
		\left(
		\mathcal{E}_a^\infty(\lambda)\langle f_N\rangle
		-
		\lambda^a m_1\langle f_N\rangle
		\sum_{k=1}^\infty \frac{1}{\Gamma(ak+1)}
		\right) \\
		&\qquad\ge
		\frac{1}{\lambda^{\delta a}}
		\mathcal{E}_a^\infty(\lambda)\langle f_N\rangle
		-
		\frac{\mathfrak{m}_1}{\lambda^{(\delta-1)a}}
		\mathcal{E}_a(1),
	\end{align*}
	where we used $m_1\langle f_N\rangle\le \mathfrak{m}_1$.
	
	Combining this with \eqref{eqn:middle_step_propaga_m_l}, we obtain
	\begin{align*}
		\mathcal{E}_a^\infty(\lambda)\langle O[f_N]\rangle
		&\le
		\frac{C}{2}
		\left(\sqrt{a}\right)^{\frac{\delta}{\delta-2\alpha_1-\beta_1}}
		+
		\frac{\mathfrak{m}_1}{2\lambda^{(\delta-1)a}}
		\mathcal{E}_a(1)
		-
		\frac{1}{2\lambda^{\delta a}}
		\mathcal{E}_a^\infty(\lambda)\langle f_N\rangle \\
		&=
		\frac{1}{2\lambda^{\delta a}}
		\Bigl(
		\lambda^{\delta a} C
		\left(\sqrt{a}\right)^{\frac{\delta}{\delta-2\alpha_1-\beta_1}}
		+
		\lambda^a \mathfrak{m}_1 \mathcal{E}_a(1)
		-
		\mathcal{E}_a^\infty(\lambda)\langle f_N\rangle
		\Bigr).
	\end{align*}
	
	Choosing $\lambda$ as in \eqref{eqn:con_for_propaga_of_m_l_tail}, we ensure that
	\[
	\lambda^{\delta a} C
	\left(\sqrt{a}\right)^{\frac{\delta}{\delta-2\alpha_1-\beta_1}}
	+
	\lambda^a \mathfrak{m}_1 \mathcal{E}_a(1)
	<
	\frac{1}{2}.
	\]
	
	We now distinguish two cases.
	
	\begin{enumerate}
		\item
		If $\mathcal{E}_a^\infty(\lambda)\langle f\rangle\le \tfrac12$, then
		\begin{align*}
			\mathcal{E}_a^\infty(\lambda)\langle w_N\rangle
			&=
			\mathcal{E}_a^\infty(\lambda)\langle f\rangle
			+
			\tau\,\mathcal{E}_a^\infty(\lambda)\langle O[f_N]\rangle \\
			&\le
			\frac12
			+
			\tau
			\left(
			\frac{C}{2}
			\left(\sqrt{a}\right)^{\frac{\delta}{\delta-2\alpha_1-\beta_1}}
			+
			\frac{\mathfrak{m}_1}{2\lambda^{(\delta-1)a}}
			\mathcal{E}_a(1)
			\right)
			\le 1,
		\end{align*}
		provided $\tau>0$ is sufficiently small.
		
		\item
		If $\mathcal{E}_a^\infty(\lambda)\langle f\rangle>\tfrac12$, then for $N$ large enough
		$\mathcal{E}_a^\infty(\lambda)\langle f_N\rangle>\tfrac12$, and hence
		\[
		\mathcal{E}_a^\infty(\lambda)\langle w_N\rangle
		=
		\mathcal{E}_a^\infty(\lambda)\langle f\rangle
		+
		\tau\,\mathcal{E}_a^\infty(\lambda)\langle O[f_N]\rangle
		\le
		\mathcal{E}_a^\infty(\lambda)\langle f\rangle
		\le 1.
		\]
	\end{enumerate}
	
	In both cases, we conclude that $w_N\in\mathcal{T}$. Arguing as in the proof of
	Proposition~\ref{prop:existence_viable_set}, the Hölder continuity of $O$ implies
	the sub-tangency condition \eqref{eqn:sub_tangency_for_mittag}.
	Since $\mathcal{T}$ is a closed convex subset of $\mathcal{S}$, the arguments in
	Theorem~\ref{thm:GWP} apply verbatim with $\mathcal{S}$ replaced by $\mathcal{T}$.
	
	Finally, since $f_0\in\mathcal{T}$ and the solution is unique, we obtain
	\[
	\sum_{i=1}^\infty f_i(t)\,\mathcal{E}_a(\lambda^a\bfi)\le 1,
	\qquad
	\forall\, t\in[0,\infty).
	\]
	
	\subsection{Proof of Theorem \ref{thm:creation_of_exp_moments}}		\label{SecTh7}
	To simplify the notation, we set
	\[
	\mathfrak{d}:=\lfloor \delta \rfloor,
	\qquad
	\mathfrak{m}_1:=m_1\langle f_0\rangle .
	\]
	
	It follows from Theorem~\ref{thm:energy_decay} and
	Theorem~\ref{thm:creation_of_polynomial_moments} that, for any $k\ge 1$,
	\begin{align*}
		m_k\left\langle f(t)\right\rangle
		&\le
		\left(\frac{2(k-1)}{\delta t}\right)^{\frac{k-1}{\delta}}
		\mathfrak{m}_1
		+\mathcal{B}_k(\mathfrak{m}_1) \\
		&\le
		C_k(\mathfrak{m}_1)\left(t^{-\frac{k-1}{\delta}}+1\right).
	\end{align*}
	
	Since $\mathfrak{d}=\lfloor \delta \rfloor\le \delta$, we have
	\[
	\frac{k-1}{\delta}\le \frac{k-1}{\mathfrak{d}}.
	\]
	Consequently, for $t\in(0,1]$,
	\[
	t^{-\frac{k-1}{\delta}}
	\le
	t^{-\frac{k-1}{\mathfrak{d}}},
	\]
	and therefore
	\begin{equation}\label{eqn:estimate_for_mk_in_creation_of_exp_mom}
		m_k\left\langle f(t)\right\rangle
		\le
		C_k(\mathfrak{m}_1)
		\left(t^{-\frac{k-1}{\mathfrak{d}}}+1\right).
	\end{equation}
	
	As a result, for any $\lambda>0$, since $\mathfrak{d}\ge 1$, we obtain
	\begin{align}
		\mathcal{E}_1^n\!\left(\lambda t^{\frac{1}{\mathfrak{d}}}\right)\left\langle f(t)\right\rangle
		&:=
		\sum_{k=1}^n
		\frac{m_k\left\langle f(t)\right\rangle\bigl(\lambda t^{\frac{1}{\mathfrak{d}}}\bigr)^k}{k!} \notag \\
		&\le
		\sum_{k=1}^n
		\frac{
			C_k(\mathfrak{m}_1)
			\left(t^{-\frac{k-1}{\mathfrak{d}}}+1\right)
			\left(\lambda t^{\frac{1}{\mathfrak{d}}}\right)^k
		}{k!} \notag \\
		&\le
		\tilde{C}_n(\mathfrak{m}_1,\lambda)\, t^{\frac{1}{\mathfrak{d}}},
		\label{eqn:estimate_for_small_time_m_l}
	\end{align}
	for all $t\in(0,1]$.

	For an integer $n>\mathfrak{d}$, $0<\lambda\le 1$, and $0<\theta\le 1$, we define
	\[
	T_n
	:=
	\sup\left\{
	t\in(0,1]:
	\mathcal{E}_1^n\!\left(\lambda t^{\frac{1}{\mathfrak{d}}}\right)\left\langle f(t)\right\rangle
	\le
	t^{\frac{1-\theta}{\mathfrak{d}}}
	\right\}.
	\]
	
	By \eqref{eqn:estimate_for_small_time_m_l}, the quantity $T_n$ is well defined and satisfies
	$T_n>0$ for every $n$.
	Our goal is to show that, for $\lambda>0$ sufficiently small, one has
	\[
	T_n=1
	\qquad
	\text{for all } n>\mathfrak{d} \text{ and all } \theta\in(0,1].
	\]
	
	For $n\ge 1$, differentiating $\mathcal{E}_1^n(\lambda t^{1/\mathfrak{d}})\left\langle f(t)\right\rangle$ yields
	\begin{equation}\label{eqn:core_estimate_in_creation_exp_mom}
		\begin{aligned}
			\frac{d}{dt}\mathcal{E}_1^n\!\left(\lambda t^{\frac{1}{\mathfrak{d}}}\right)\left\langle f(t)\right\rangle
			&=
			\frac{d}{dt}
			\sum_{k=1}^n
			m_k\left\langle f(t)\right\rangle\frac{(\lambda t^{\frac{1}{\mathfrak{d}}})^k}{k!} \\
			&=
			\sum_{k=1}^n
			m_k\left\langle O[f(t)]\right\rangle\frac{(\lambda t^{\frac{1}{\mathfrak{d}}})^k}{k!}
			+
			\frac{\lambda}{\mathfrak{d}\, t^{\frac{\mathfrak{d}-1}{\mathfrak{d}}}}
			\sum_{k=1}^n
			m_k\left\langle f(t)\right\rangle\frac{(\lambda t^{\frac{1}{\mathfrak{d}}})^{k-1}}{(k-1)!}.
		\end{aligned}
	\end{equation}
	
	Proceeding as in the proof of Theorem~\ref{thm:propagation_of_m_l_tails}, we obtain
	\begin{align*}
		\sum_{k=1}^n
		m_k\left\langle O[f(t)]\right\rangle\frac{(\lambda t^{\frac{1}{\mathfrak{d}}})^k}{k!}
		&\le
		2^{\beta_1+2}
		\mathcal{E}_{1,\alpha_1}^n\!\left(\lambda t^{\frac{1}{\mathfrak{d}}}\right)
		\mathcal{E}_{1,\alpha_1+\beta_1}^n\!\left(\lambda t^{\frac{1}{\mathfrak{d}}}\right)
		-
		\mathcal{E}_{1,\delta}^n\!\left(\lambda t^{\frac{1}{\mathfrak{d}}}\right) \\
		&\le
		2^{\beta_1+2}
		\left(\mathcal{E}_1^n\!\left(\lambda t^{\frac{1}{\mathfrak{d}}}\right)\right)^{\frac{2\delta-2\alpha_1-\beta_1}{\delta}}
		\left(\mathcal{E}_{1,\delta}^n\!\left(\lambda t^{\frac{1}{\mathfrak{d}}}\right)\right)^{\frac{2\alpha_1+\beta_1}{\delta}}
		-
		\mathcal{E}_{1,\delta}^n\!\left(\lambda t^{\frac{1}{\mathfrak{d}}},t\right).
	\end{align*}
	
	On the interval $(0,T_n]$, by definition of $T_n$ we have
	\[
	\mathcal{E}_1^n\!\left(\lambda t^{\frac{1}{\mathfrak{d}}}\right)\left\langle f(t)\right\rangle
	\le
	t^{\frac{1-\theta}{\mathfrak{d}}}
	\le 1,
	\]
	and therefore
	\begin{equation}\label{eqn:estimate_1_in_creation_of_exp_mom}
		\begin{aligned}
			\sum_{k=1}^n
			m_k\left\langle O[f(t)]\right\rangle\frac{(\lambda t^{\frac{1}{\mathfrak{d}}})^k}{k!}
			&\le
			2^{\beta_1+2}
			\left(\mathcal{E}_{1,\delta}^n\!\left(\lambda t^{\frac{1}{\mathfrak{d}}}\right)\right)^{\frac{2\alpha_1+\beta_1}{\delta}}
			-
			\mathcal{E}_{1,\delta}^n\!\left(\lambda t^{\frac{1}{\mathfrak{d}}}\right) \\
			&\le
			C
			-
			\frac{1}{2}
			\mathcal{E}_{1,\delta}^n\!\left(\lambda t^{\frac{1}{\mathfrak{d}}}\right)\left\langle f(t)\right\rangle,
		\end{aligned}
	\end{equation}
	where the constant $C>0$ depends only on $\alpha_1$, $\beta_1$, and $\delta$.

	For the remaining term in \eqref{eqn:core_estimate_in_creation_exp_mom},
	recalling \eqref{eqn:estimate_for_mk_in_creation_of_exp_mom}, we obtain
	\begin{equation}\label{eqn:estimate_2_in_creation_of_exp_mom}
		\begin{aligned}
			&\frac{\lambda}{\mathfrak{d}\, t^{\frac{\mathfrak{d}-1}{\mathfrak{d}}}}
			\sum_{k=1}^n
			m_k\left\langle f(t)\right\rangle\frac{\left(\lambda t^{\frac{1}{\mathfrak{d}}}\right)^{k-1}}{(k-1)!} \\
			&\quad=
			\frac{\lambda}{\mathfrak{d}\, t^{\frac{\mathfrak{d}-1}{\mathfrak{d}}}}
			\sum_{k=\mathfrak{d}+1}^n
			m_k\left\langle f(t)\right\rangle\frac{\left(\lambda t^{\frac{1}{\mathfrak{d}}}\right)^{k-1}}{(k-1)!}
			+
			\frac{\lambda}{\mathfrak{d}\, t^{\frac{\mathfrak{d}-1}{\mathfrak{d}}}}
			\sum_{k=1}^{\mathfrak{d}}
			m_k\left\langle f(t)\right\rangle\frac{\left(\lambda t^{\frac{1}{\mathfrak{d}}}\right)^{k-1}}{(k-1)!}.
		\end{aligned}
	\end{equation}
	
	Using the uniform bound on low-order moments, the second sum is controlled by
	\[
	\frac{\lambda}{\mathfrak{d}\, t^{\frac{\mathfrak{d}-1}{\mathfrak{d}}}}
	\sum_{k=1}^{\mathfrak{d}}
	m_k\left\langle f(t)\right\rangle\frac{\left(\lambda t^{\frac{1}{\mathfrak{d}}}\right)^{k-1}}{(k-1)!}
	\le
	C(\mathfrak{m}_1)\,\frac{\lambda^{\mathfrak{d}}}{t^{\frac{\mathfrak{d}-1}{\mathfrak{d}}}}.
	\]
	
	For the first sum, we rewrite
	\begin{align*}
		&\frac{\lambda}{\mathfrak{d}\, t^{\frac{\mathfrak{d}-1}{\mathfrak{d}}}}
		\sum_{k=\mathfrak{d}+1}^n
		m_k\left\langle f(t)\right\rangle\frac{\left(\lambda t^{\frac{1}{\mathfrak{d}}}\right)^{k-1}}{(k-1)!} \\
		&\quad=
		\frac{\lambda^{\mathfrak{d}}}{\mathfrak{d}}
		\sum_{k=1}^{n-\mathfrak{d}}
		m_{k+\mathfrak{d}}\left\langle f(t)\right\rangle
		\frac{\left(\lambda t^{\frac{1}{\mathfrak{d}}}\right)^k}{(k+\mathfrak{d}-1)!}
		\;\le\;
		\frac{\lambda^{\mathfrak{d}}}{\mathfrak{d}}
		\sum_{k=1}^n
		m_{k+\mathfrak{d}}\left\langle f(t)\right\rangle
		\frac{\left(\lambda t^{\frac{1}{\mathfrak{d}}}\right)^k}{k!}.
	\end{align*}
	
	Consequently,
	\[
	\frac{\lambda}{\mathfrak{d}\, t^{\frac{\mathfrak{d}-1}{\mathfrak{d}}}}
	\sum_{k=1}^n
	m_k\left\langle f(t)\right\rangle\frac{\left(\lambda t^{\frac{1}{\mathfrak{d}}}\right)^{k-1}}{(k-1)!}
	\le
	\frac{\lambda^{\mathfrak{d}}}{\mathfrak{d}}\,
	\mathcal{E}_{1,\mathfrak{d}}^n\!\left(\lambda t^{\frac{1}{\mathfrak{d}}}\right)\left\langle f(t)\right\rangle
	+
	C(\mathfrak{m}_1)\,\frac{\lambda^{\mathfrak{d}}}{t^{\frac{\mathfrak{d}-1}{\mathfrak{d}}}}.
	\]
	
	Combining \eqref{eqn:core_estimate_in_creation_exp_mom},
	\eqref{eqn:estimate_1_in_creation_of_exp_mom}, and
	\eqref{eqn:estimate_2_in_creation_of_exp_mom}, we arrive at
	\begin{equation}\label{eqn:final_diff_ineq_creation_exp}
		\frac{d}{dt}\mathcal{E}_1^n\!\left(\lambda t^{\frac{1}{\mathfrak{d}}}\right)\left\langle f(t)\right\rangle
		\le
		C
		+
		C(\mathfrak{m}_1)\,\frac{\lambda^{\mathfrak{d}}}{t^{\frac{\mathfrak{d}-1}{\mathfrak{d}}}}
		-
		\left(
		\frac{1}{2}
		-
		\frac{\lambda^{\mathfrak{d}}}{\mathfrak{d}}
		\right)
		\mathcal{E}_{1,\mathfrak{d}}^n\!\left(\lambda t^{\frac{1}{\mathfrak{d}}}\right)\left\langle f(t)\right\rangle,
	\end{equation}
	where the constants depend only on $\alpha_1,\beta_1,\delta$ and $\mathfrak{m}_1$.
	
	For $\lambda < \left(\frac{\mathfrak{d}}{4}\right)^{1/\mathfrak{d}}$, the previous estimate yields
	\[
	\frac{d}{dt}\mathcal{E}_{1}^n\!\left(\lambda t^{\frac{1}{\mathfrak{d}}}\right)\left\langle f(t)\right\rangle
	\le
	C
	+
	C(\mathfrak{m}_1)\frac{\lambda^{\mathfrak{d}}}{t^{\frac{\mathfrak{d}-1}{\mathfrak{d}}}}
	-
	\frac{1}{4}\mathcal{E}_{1,\mathfrak{d}}^n\!\left(\lambda t^{\frac{1}{\mathfrak{d}}}\right)\left\langle f(t)\right\rangle.
	\]
	
	Since $\lambda \le 1$, similar as in the Proof of Theorem \ref{thm:propagation_of_m_l_tails}, we estimate
	\begin{align*}
		&\mathcal{E}_{1,\mathfrak{d}}^n\!\left(\lambda t^{\frac{1}{\mathfrak{d}}}\right)\left\langle f(t)\right\rangle
		=
		\sum_{k=1}^n
		m_{k+\mathfrak{d}}\left\langle f(t)\right\rangle\frac{\left(\lambda t^{\frac{1}{\mathfrak{d}}}\right)^k}{k!} \\
		&\quad =
		\frac{1}{t\lambda^{\mathfrak{d}}}
		\sum_{k=\mathfrak{d}+1}^{n+\mathfrak{d}}
		m_k\left\langle f(t)\right\rangle\frac{\left(\lambda t^{\frac{1}{\mathfrak{d}}}\right)^k}{(k-\mathfrak{d})!}\ge
		\frac{1}{t\lambda^{\mathfrak{d}}}
		\sum_{k=\mathfrak{d}+1}^{n}
		m_k\left\langle f(t)\right\rangle\frac{\left(\lambda t^{\frac{1}{\mathfrak{d}}}\right)^k}{k!} \\
		&\quad=
		\frac{1}{t\lambda^{\mathfrak{d}}}
		\mathcal{E}_1^n\!\left(\lambda t^{\frac{1}{\mathfrak{d}}}\right)\left\langle f(t)\right\rangle
		-
		\frac{1}{t\lambda^{\mathfrak{d}}}
		\sum_{k=1}^{\mathfrak{d}}
		m_k\left\langle f(t)\right\rangle\frac{\left(\lambda t^{\frac{1}{\mathfrak{d}}}\right)^k}{k!} \\
		&\quad\ge
		\frac{1}{t\lambda^{\mathfrak{d}}}
		\mathcal{E}_1^n\!\left(\lambda t^{\frac{1}{\mathfrak{d}}}\right)\left\langle f(t)\right\rangle
		-
		\frac{C'(\mathfrak{m}_1)}{t^{\frac{\mathfrak{d}-1}{\mathfrak{d}}}\lambda^{\mathfrak{d}-1}} .
	\end{align*}
	
	Since $t\le 1$ and $\lambda\le 1$, combining the two estimates yields
	\[
	\frac{d}{dt}\mathcal{E}_1^n\!\left(\lambda t^{\frac{1}{\mathfrak{d}}},t\right)
	\le
	\frac{C''(\mathfrak{m}_1)}{t^{\frac{\mathfrak{d}-1}{\mathfrak{d}}}\lambda^{\mathfrak{d}-1}}
	-
	\frac{1}{4}\frac{1}{t\lambda^{\mathfrak{d}}}
	\mathcal{E}_1^n\!\left(\lambda t^{\frac{1}{\mathfrak{d}}},t\right).
	\]
	
	Choosing $\lambda < \frac{1}{4C''(\mathfrak{m}_1)}$, we ensure that
	\[
	\frac{\mathfrak{d}C''(\mathfrak{m}_1)\lambda}{\frac{1}{4}\mathfrak{d}+\lambda^{\mathfrak{d}}} < 1,
	\]
	which implies
	\[
	\mathcal{E}_{1,\mathfrak{d}}^n\!\left(\lambda t^{\frac{1}{\mathfrak{d}}},t\right)
	<
	t^{\frac{1}{\mathfrak{d}}},
	\qquad \forall\, t\in(0,T_n].
	\]
	
	By continuity of $\mathcal{E}_{1,\mathfrak{d}}^n$, this forces $T_n=1$ for all
	$n>\mathfrak{d}$ and all $\theta\in(0,1]$. Hence, for all $t\in(0,1]$,
	\[
	\mathcal{E}_1^n\!\left(\lambda t^{\frac{1}{\mathfrak{d}}},t\right)
	<
	t^{\frac{1-\theta}{\mathfrak{d}}},
	\qquad \forall\, n>\mathfrak{d},\ \theta\in(0,1].
	\]
	
	Letting $\theta\to 0$ and $n\to\infty$, we obtain
	\begin{equation}\label{eqn:estimate_final_exp_creation}
		\sum_{i=1}^\infty
		f_i(t)\,\mathcal{E}_1\!\left(\lambda t^{\frac{1}{\mathfrak{d}}}\bfi\right)
		\le
		t^{\frac{1}{\mathfrak{d}}}.
	\end{equation}
	
	Recalling Remark~\ref{rmk:estimate_for_m_l_function}, we further deduce
	\[
	\sum_{i=1}^\infty
	f_i(t)\,\lambda t^{\frac{1}{\mathfrak{d}}}\bfi
	e^{\frac{\lambda t^{\frac{1}{\mathfrak{d}}}}{2}\bfi}
	\le
	\sum_{i=1}^\infty
	f_i(t)\,\mathcal{E}_1\!\left(\lambda t^{\frac{1}{\mathfrak{d}}}\bfi\right)
	\le
	t^{\frac{1}{\mathfrak{d}}},
	\]
	and therefore, for all $t\in(0,1]$,
	\[
	\sum_{i=1}^\infty
	f_i(t)\,\bfi
	e^{\frac{\lambda t^{\frac{1}{\mathfrak{d}}}}{2}\bfi}
	\le
	\frac{1}{\lambda}.
	\]
	
	For $t\ge 1$, since estimate \eqref{eqn:estimate_final_exp_creation} implies that
	\[
	\sum_{i=1}^\infty
	f_i(1)\,\mathcal{E}_1(\lambda\bfi)
	\le 1.
	\]
	Taking $\lambda$ small enough as in \eqref{eqn:con_for_propaga_of_m_l_tail}, Theorem~\ref{thm:propagation_of_m_l_tails} implies that for all
	$t\ge 1$,
	\[
	\sum_{i=1}^\infty
	f_i(t)\,\bfi e^{\frac{\lambda}{2}\bfi}
	\le \frac{1}{\lambda}.
	\]
	
	The conclusion follows upon replacing $\lambda$ by $2\lambda$ and recalling
	$\mathfrak{d}=\lfloor\delta\rfloor$.

	\section{Numerical Tests}\label{Sec:Numerical}
	
	In this section, 
	we present numerical experiments for the parameter choices
	\[
	\alpha_n=\beta_n=0.1, \qquad \delta=0.4.
	\]
	
	The computational domain is taken to be $[0,50]$ with spatial mesh size $h=0.1$.
	For clarity of presentation, we display the numerical solutions only on the
	interval $[0,10]$, since the solution becomes negligibly small outside this
	range.
	
	We perform a forward Euler scheme in time with the time step $\Delta t=0.001$.
	Because the dissipation induced by the $V$ term leads to rapid decay of the
	solution, our numerical investigation focuses on this short-time regime.
	
	Our objective is to numerically verify several key analytical properties of the model, namely the strict positivity of solutions established in Theorem~\ref{thm:creation_of_positivity}, the propagation of polynomial moments proved in Corollary~\ref{Corr:Poly}, and the propagation of Mittag--Leffler moments established in Theorem~\ref{thm:propagation_of_m_l_tails}. 
	
	We emphasize that, in all numerical simulations, the computational domain must be truncated, which implies that the initial data are necessarily compactly supported. As a consequence, the numerical experiments cannot directly capture the creation of polynomial or Mittag--Leffler moments from  initial data that are not compactly supported. Therefore, the results concerning the creation of such moments should be understood as purely theoretical in nature, while the numerical simulations serve to illustrate the preservation and qualitative behavior of these moments within the compactly supported initial data setting.
	
	%	Computation shows that, the constant \(C>0\) appears in \eqref{eqn:middle_step_propaga_m_l} has the following explicit formula,
	%	\[C=2\frac{\delta-2\alpha_1-\beta_1}{\delta}\left(2\frac{2\alpha_1+\beta_1}{\delta}\right)^{\frac{2\alpha_1+\beta_1}{\delta-2\alpha_1-\beta_1}}\left(2^{\beta_1+2}\right)^{\frac{\delta}{\delta-2\alpha_1-\beta_1}}.\]
	%	
	%	Hence the upper bound for \(\lambda\) in \eqref{eqn:con_for_propaga_of_m_l_tail} is computable if the initial data is given. In the following numerical test, we will also verify the result in Theorem \ref{thm:propagation_of_m_l_tails} by comparing the behavior of the Mittag–Leffler moment for values of \(\lambda\) below and above the computed upper bound.
	\subsection{Test 1} 
	We choose the initial data \(f_0\) as follows
	\begin{figure}[H]
		\centering
		\begin{minipage}{0.49\textwidth} 
			
			%	\[\psi(x)=\begin{cases}
				%		e^{\frac{1}{\left|x-1\right|^2-1}+1}& 0< x< 2\\
				%		0 & x>2
				%	\end{cases},\]
			\[\psi(x)=\begin{cases}
				\frac{1}{2}\left(1+ \cos(\frac{\pi x}{5})\right)& 0< x< 5\\
				0 & x>5
			\end{cases},\]
			and
			\[f_0(i)=\psi(ih),\quad i\in\mathbb{N}.\]
		\end{minipage}
		\hfill
		\begin{minipage}{0.49\textwidth}
			\centering
			\includegraphics[width=\textwidth]{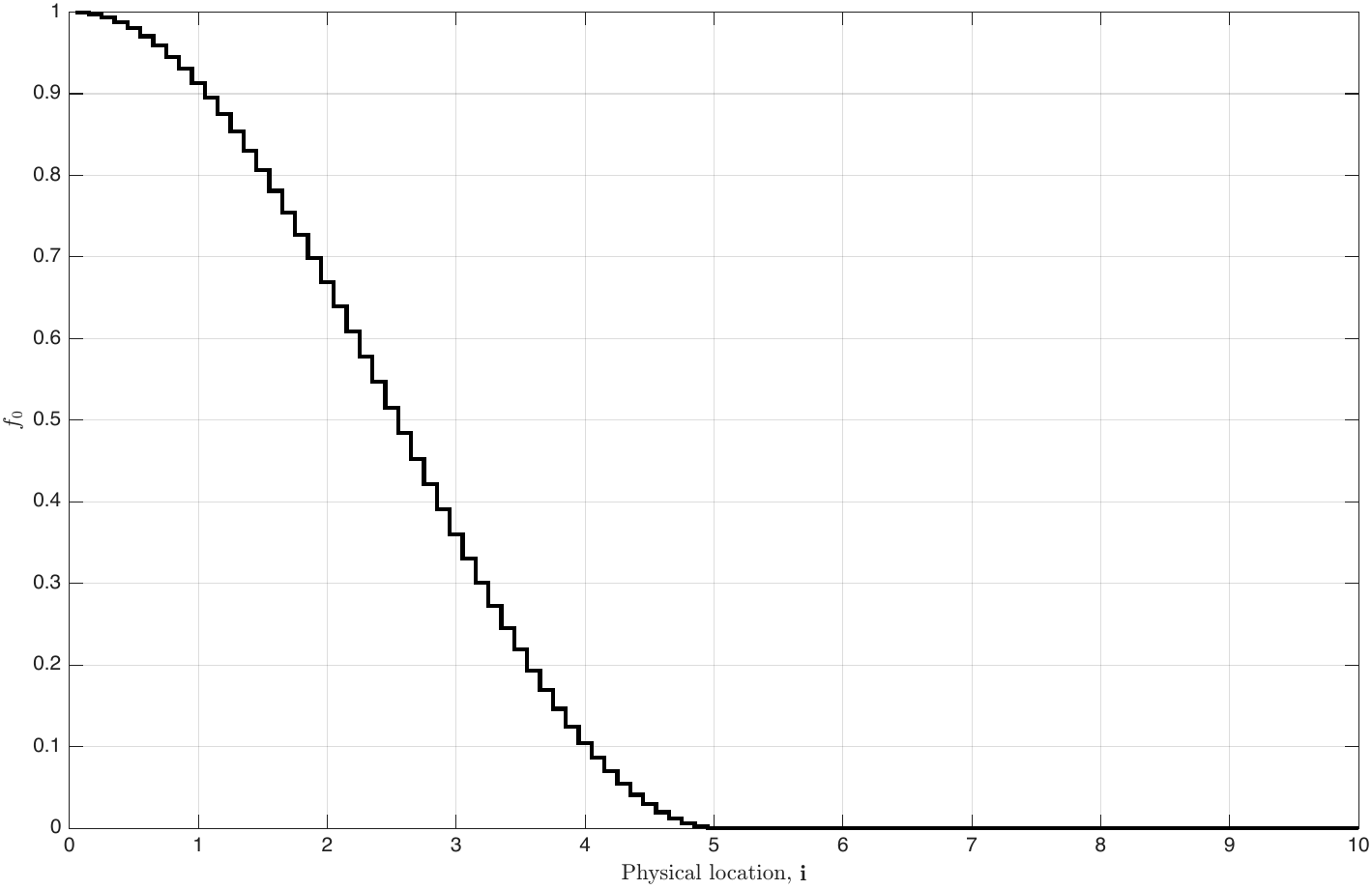}
			\caption{Initial data \(f_0\).}
		\end{minipage}
	\end{figure}
	
	Figure~\ref{Fisol1} displays the numerical solution at times $T=1$ and $T=10$.
	The results confirm the strict positivity of the solution established in
	Theorem~\ref{thm:creation_of_positivity} , as well as the propagation of
	polynomial moments proved in Corollary~\ref{Corr:Poly} and the propagation of
	Mittag--Leffler moments established in
	Theorem~\ref{thm:propagation_of_m_l_tails}.
	
	To illustrate these properties more clearly, we additionally plot the time
	evolution of several moments. Figure~\ref{Fim1m21} shows the first- and second-order
	moments, $m_1\langle f\rangle$ and $m_2\langle f\rangle$, while
	Figure~\ref{Fim3m41} presents the third- and fourth-order moments,
	$m_3\langle f\rangle$ and $m_4\langle f\rangle$, over the time interval
	$[0,1]$.
	
	Finally, Figure~\ref{Fiml1} depicts the evolution of the Mittag--Leffler moments
	corresponding to the parameters $a=1$ with $\lambda=0.1$ and $\lambda=1$,
	again over the time interval $[0,1]$. 
	\begin{figure}[H]
		\centering
		\begin{minipage}{0.49\textwidth}
			\centering
			\includegraphics[width=\textwidth]{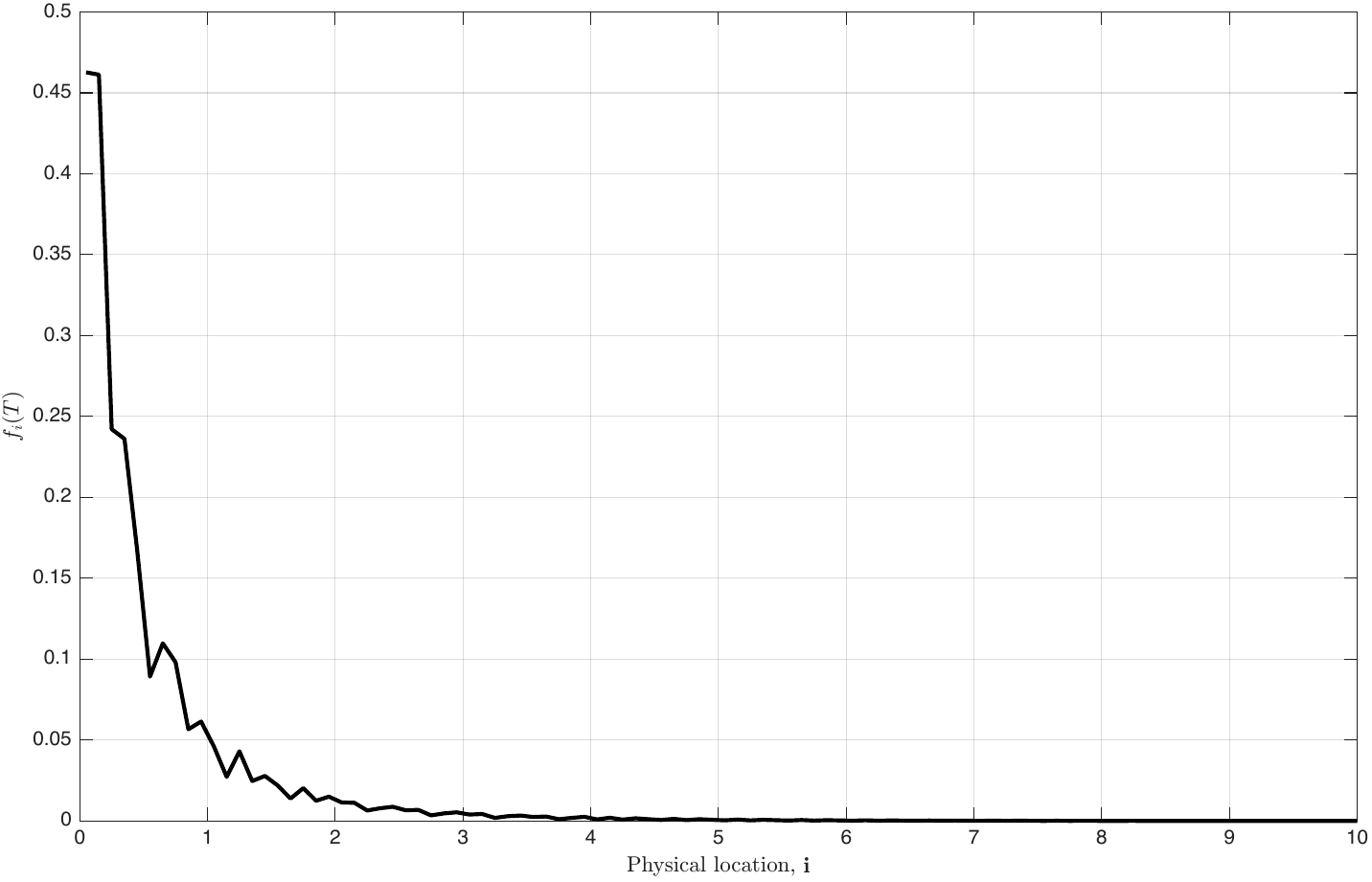}
		\end{minipage}
		\hfill
		\begin{minipage}{0.49\textwidth}
			\centering
			\includegraphics[width=\textwidth]{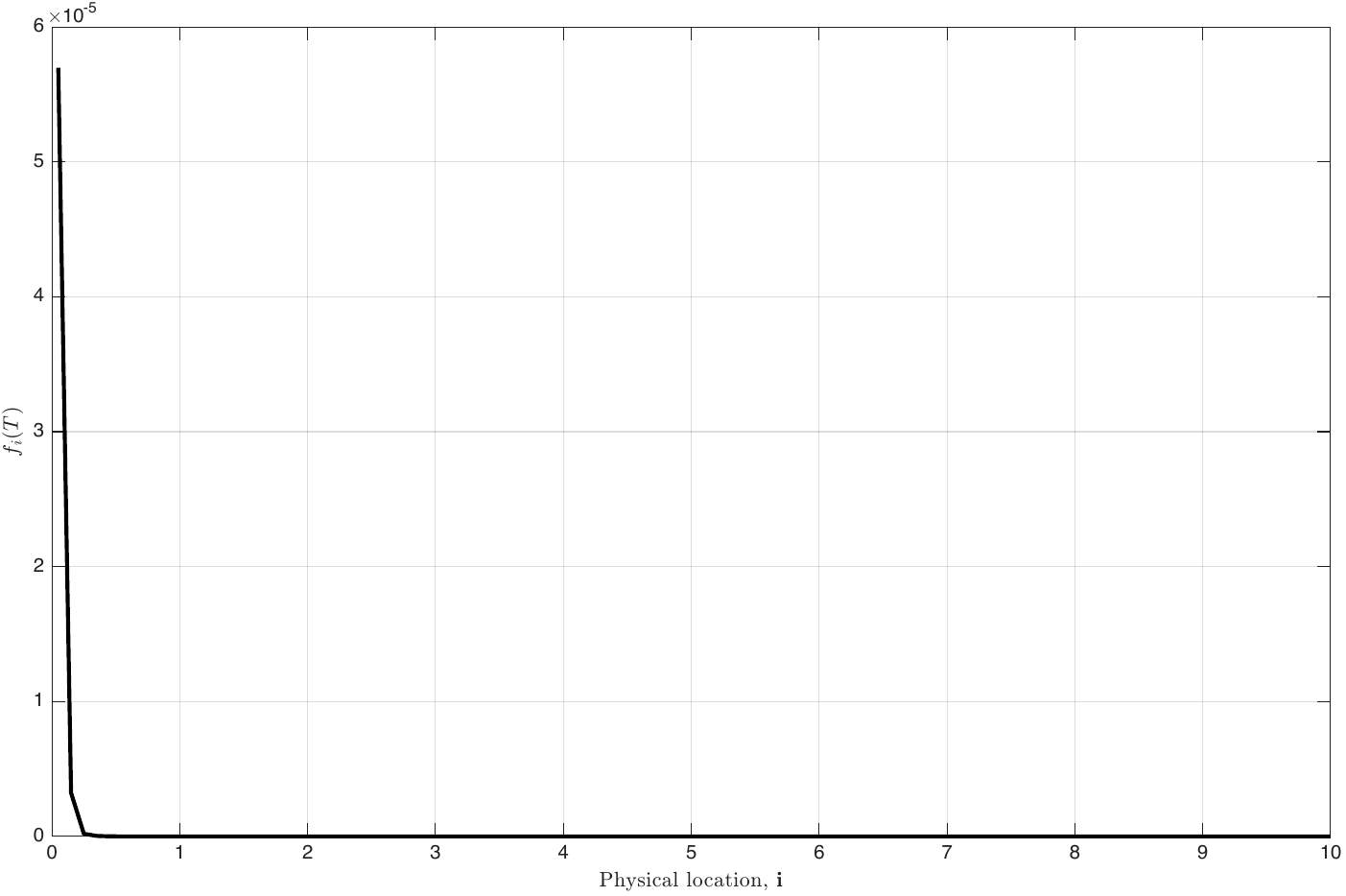}
		\end{minipage}
		\caption{Solution at \(T=1\) (left) and at \(T=10\) (right).}	\label{Fisol1}
	\end{figure}
	\begin{figure}[H]
		\centering
		\begin{minipage}{0.49\textwidth}
			\centering
			\includegraphics[width=\textwidth]{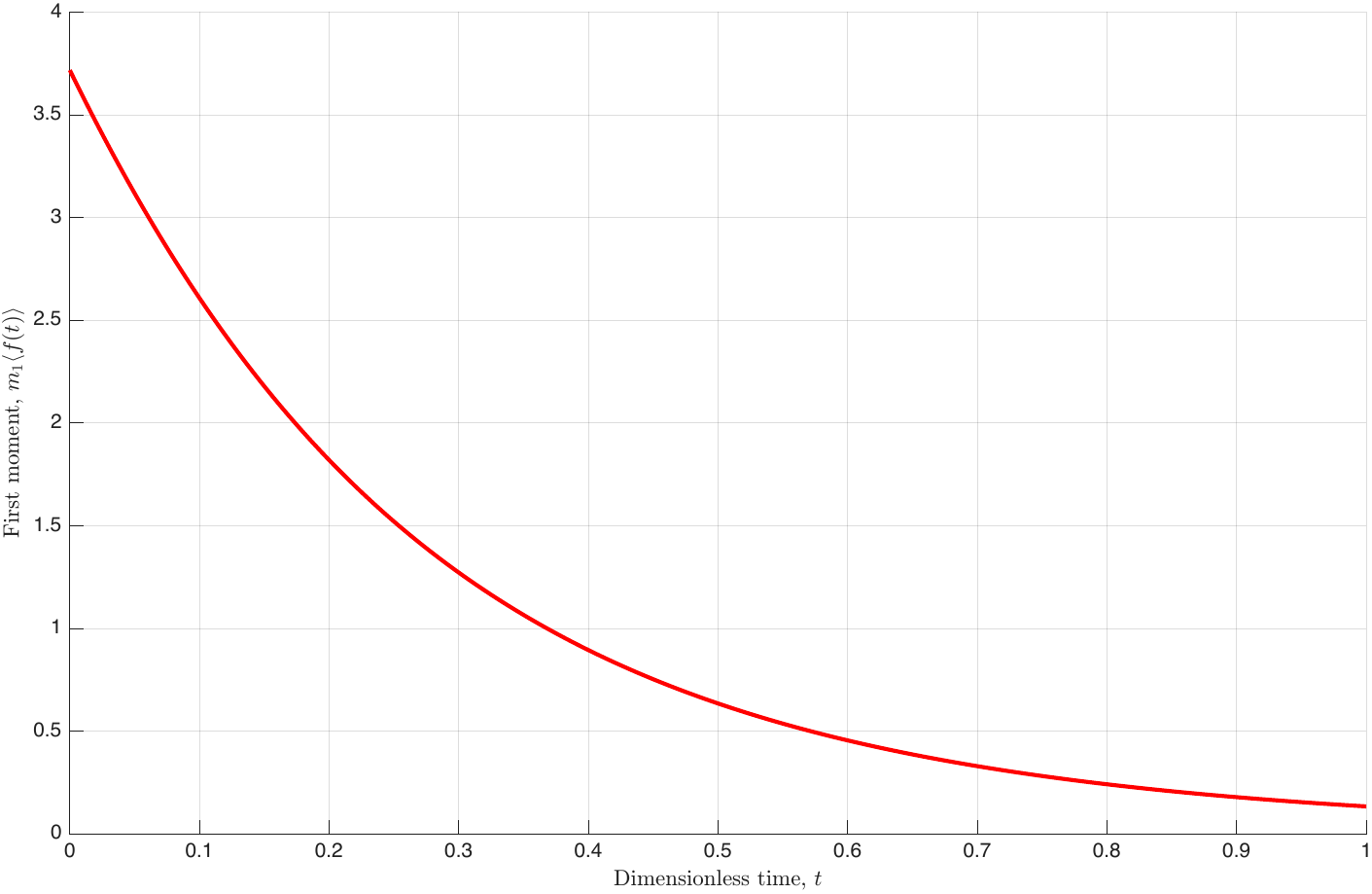}
		\end{minipage}
		\hfill
		\begin{minipage}{0.49\textwidth}
			\centering
			\includegraphics[width=\textwidth]{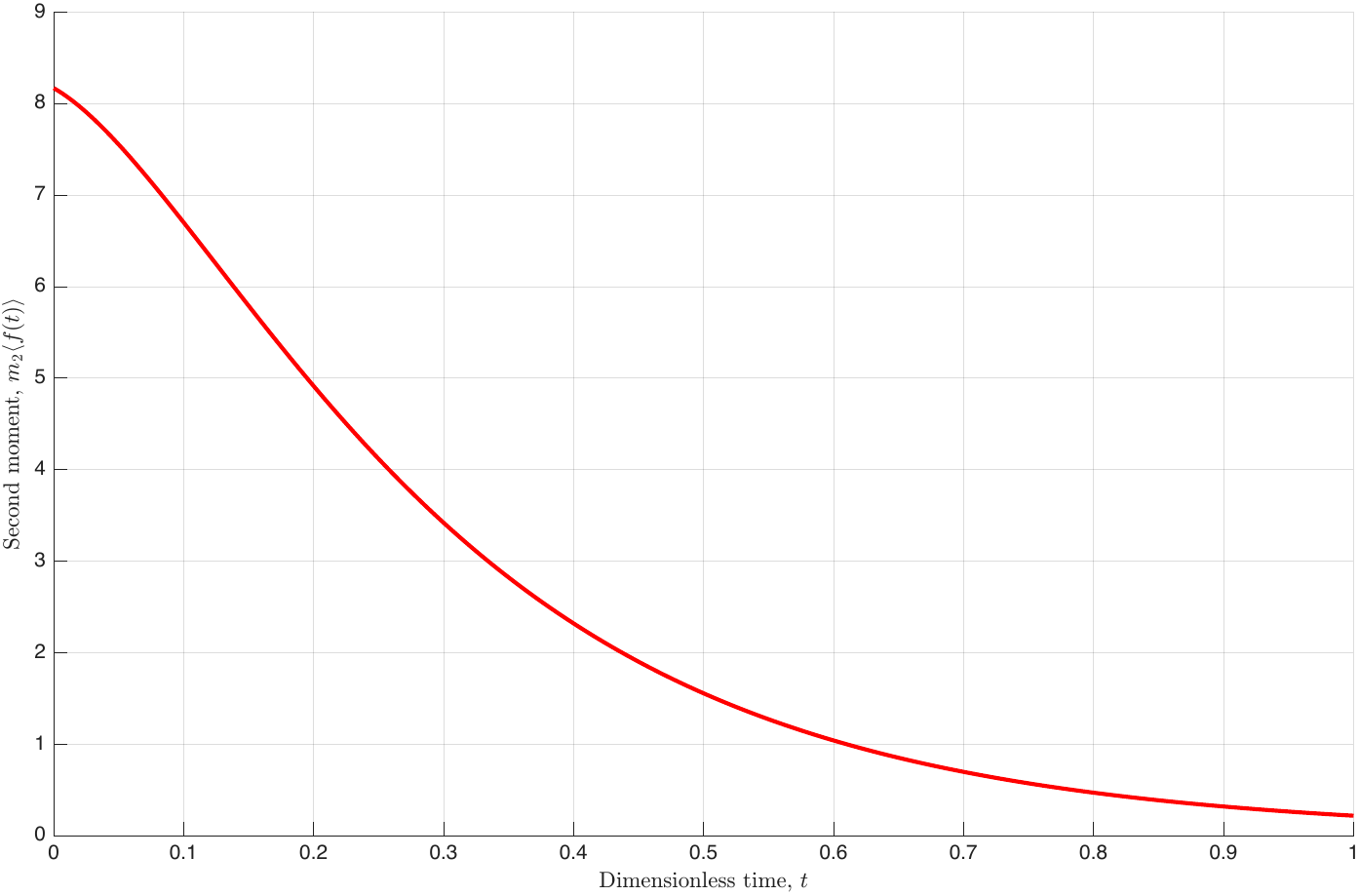}
		\end{minipage}
		\caption{\(m_1\left\langle f(t)\right\rangle\) (left) and \(m_2\left\langle f(t)\right\rangle\) (right) on \([0,1]\).}
		\label{Fim1m21}
	\end{figure}
	
	\begin{figure}[H]
		\centering
		\begin{minipage}{0.49\textwidth}
			\centering
			\includegraphics[width=\textwidth]{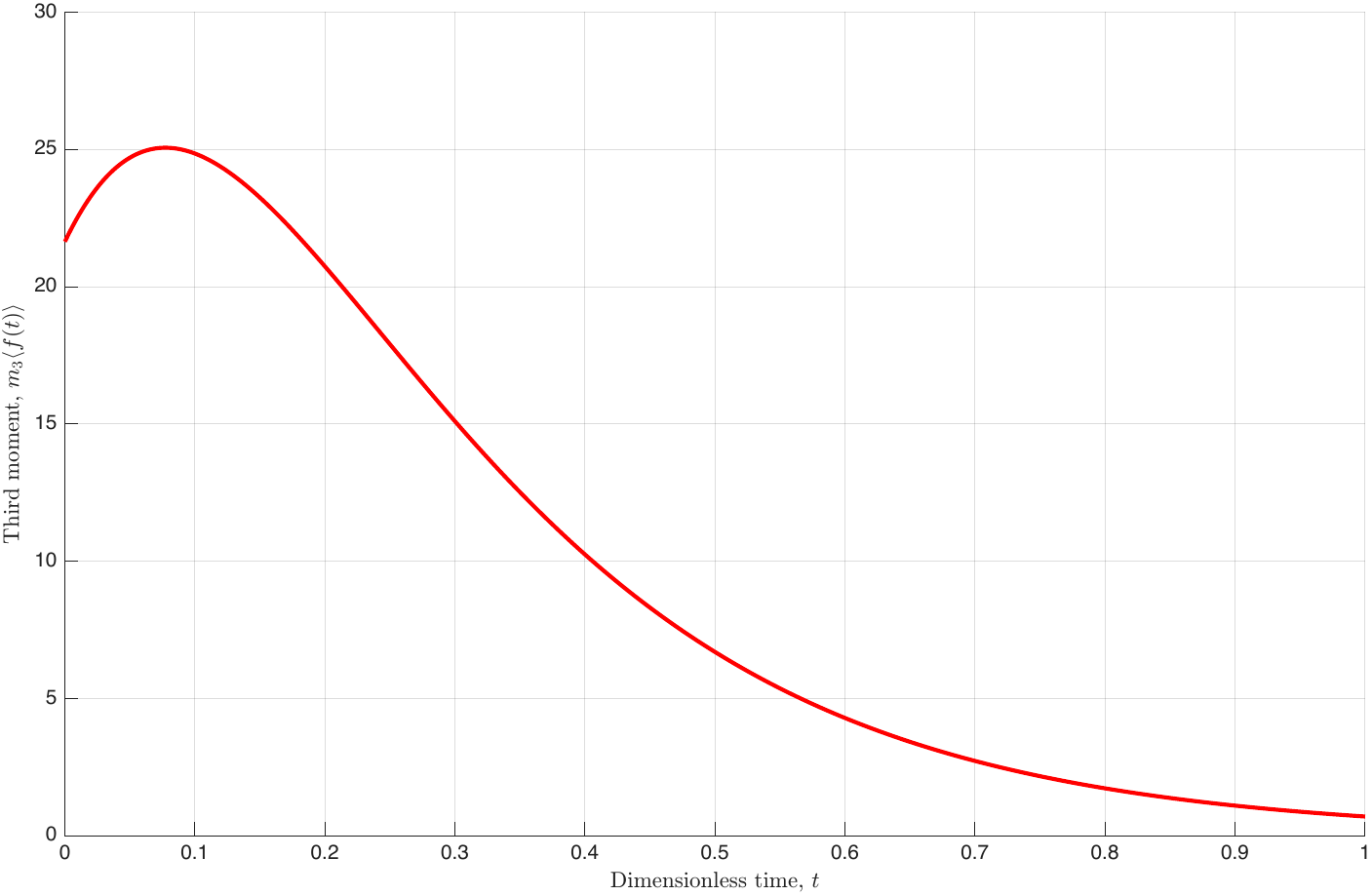}
		\end{minipage}
		\hfill
		\begin{minipage}{0.49\textwidth}
			\centering
			\includegraphics[width=\textwidth]{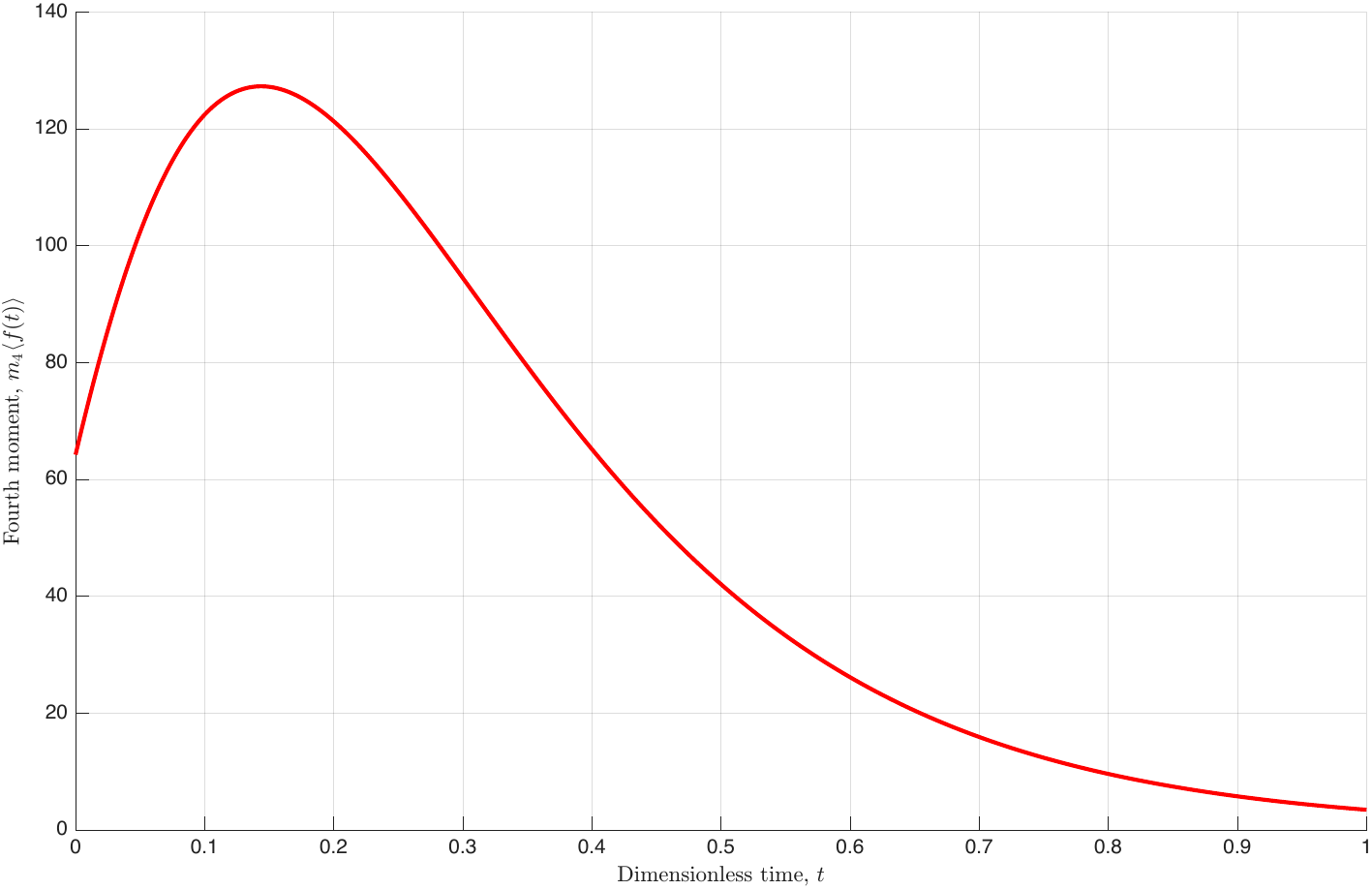}
		\end{minipage}
		\caption{\(m_3\left\langle f(t)\right\rangle\) (left) and \(m_4\left\langle f(t)\right\rangle\) (right) on \([0,1]\).}
		\label{Fim3m41}
	\end{figure}
	
	\begin{figure}[H]%[htbp]
		\centering
		\begin{minipage}{0.49\textwidth}
			\centering
			\includegraphics[width=\textwidth]{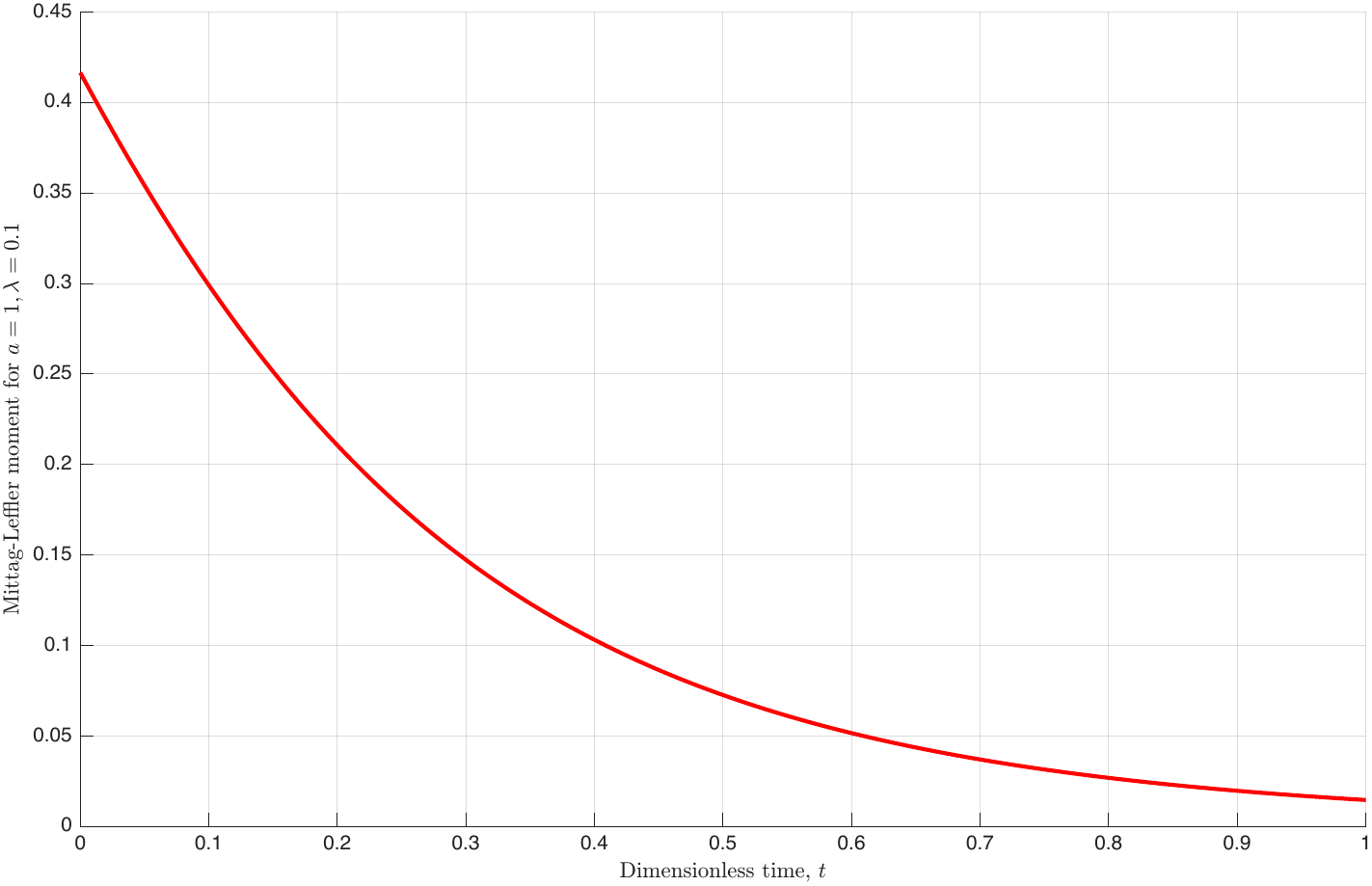}
		\end{minipage}
		\hfill
		\begin{minipage}{0.49\textwidth}
			\centering
			\includegraphics[width=\textwidth]{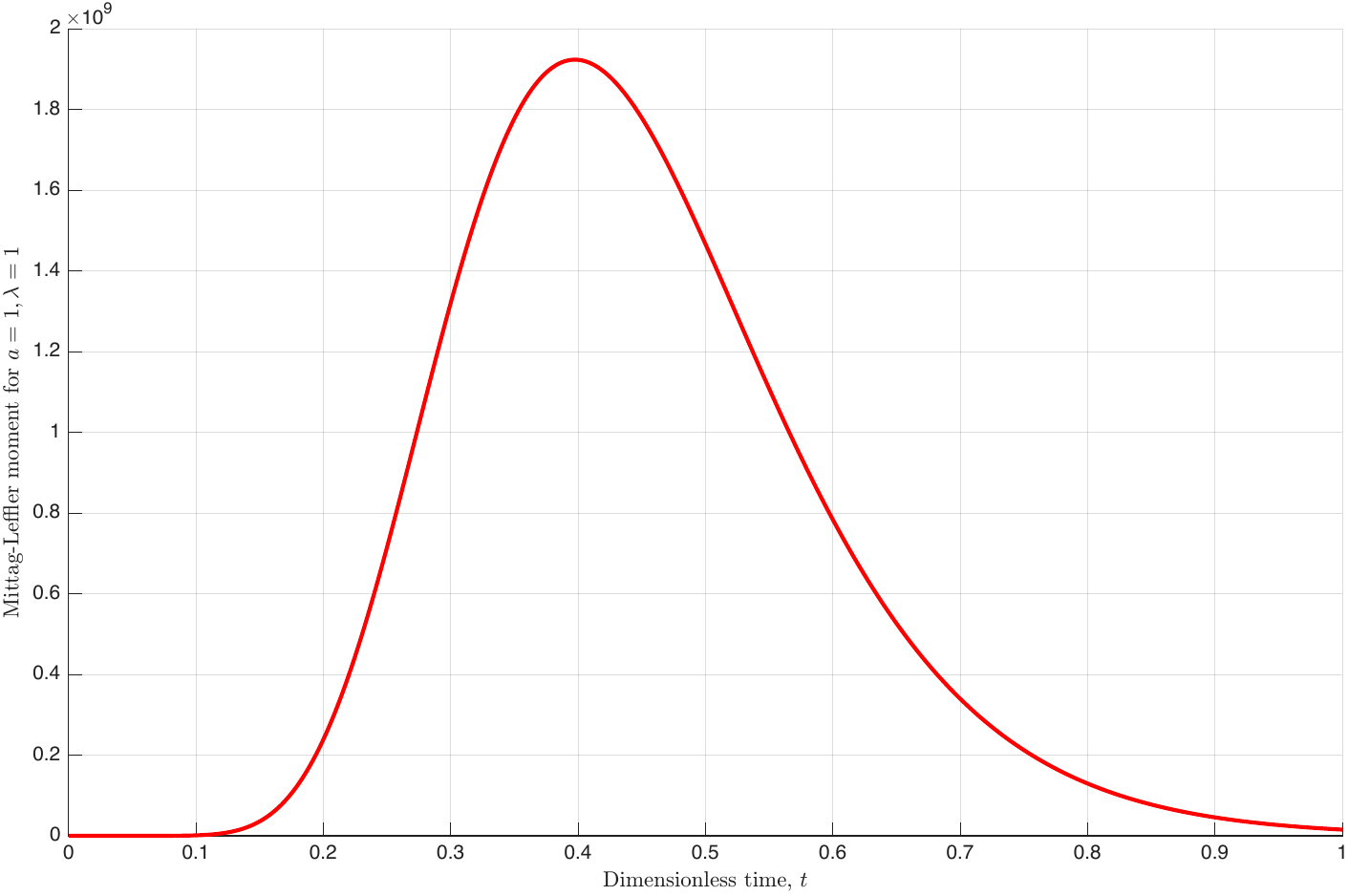}
		\end{minipage}
		\caption{Mittag-Leffler moments \(\mathcal{E}_a^\infty(\lambda)\left\langle f(t)\right\rangle\) for \(a=1,\lambda=0.1\) (left) and for \(a=1,\lambda=1\) (right) on \([0,1]\).}
		\label{Fiml1}
	\end{figure}

	\subsection{Test 2}	 
	
	In this experiment, we consider a different initial datum given by the
	indicator function
	\begin{figure}[H]
		\centering
		\begin{minipage}{0.49\textwidth}
			\[
			\chi_{[3,5]}(x)=
			\begin{cases}
				1, & 3 \le x \le 5,\\
				0, & \text{otherwise},
			\end{cases}
			\]
			and define the discrete initial data by
			\[
			f_0(i)=\chi_{[3,5]}(ih).
			\]
		\end{minipage}
		\hfill
		\begin{minipage}{0.49\textwidth}
			\centering
			\includegraphics[width=\textwidth]{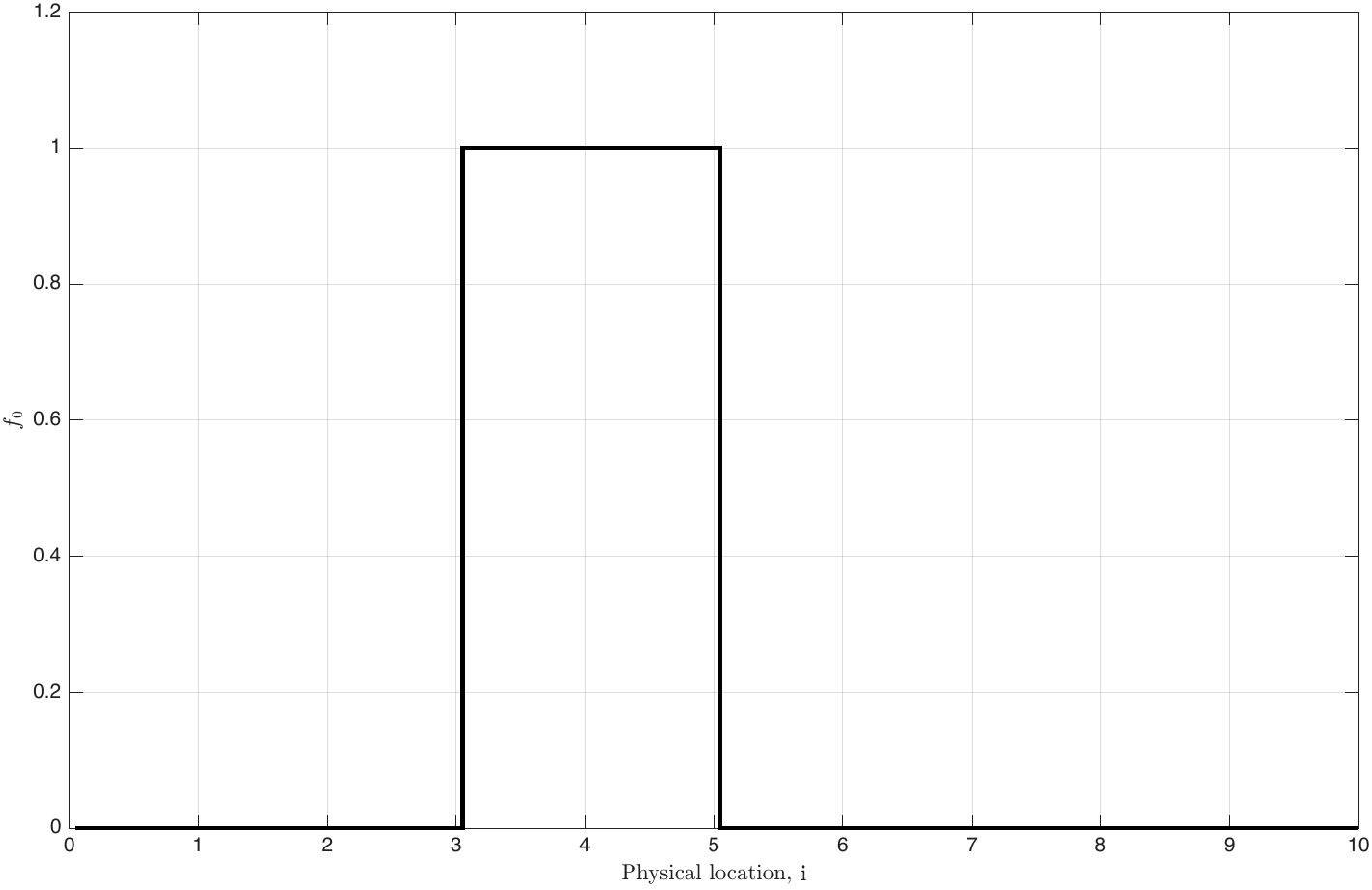}
			\caption{Initial data \(f_0\).}
		\end{minipage}
	\end{figure}
	Figure~\ref{Fisol2} shows the numerical solution at times $T=1$ and $T=10$. The results
	clearly demonstrate the persistence of positivity predicted by
	Theorem~\ref{thm:creation_of_positivity}. In addition, they confirm the
	propagation of polynomial moments established in
	Corollary~\ref{Corr:Poly}, as well as the propagation of Mittag--Leffler moments
	proved in Theorem~\ref{thm:propagation_of_m_l_tails}. 
	To further highlight these theoretical properties, we examine the temporal
	evolution of several moments. Figure~\ref{Fim1m22} displays the first- and second-order
	moments, $m_1\langle f\rangle$ and $m_2\langle f\rangle$, while Figure~\ref{Fim3m42} shows
	the third- and fourth-order moments, $m_3\langle f\rangle$ and
	$m_4\langle f\rangle$, over the time interval $[0,1]$.
	
	Finally, Figure~\ref{Fiml2} presents the evolution of the Mittag--Leffler moments with
	parameters $a=1$ and $\lambda=0.1$ and $\lambda=1$, again plotted over the time
	interval $[0,1]$.

	\begin{figure}[H]		\centering
		\begin{minipage}{0.49\textwidth}
			\centering
			\includegraphics[width=\textwidth]{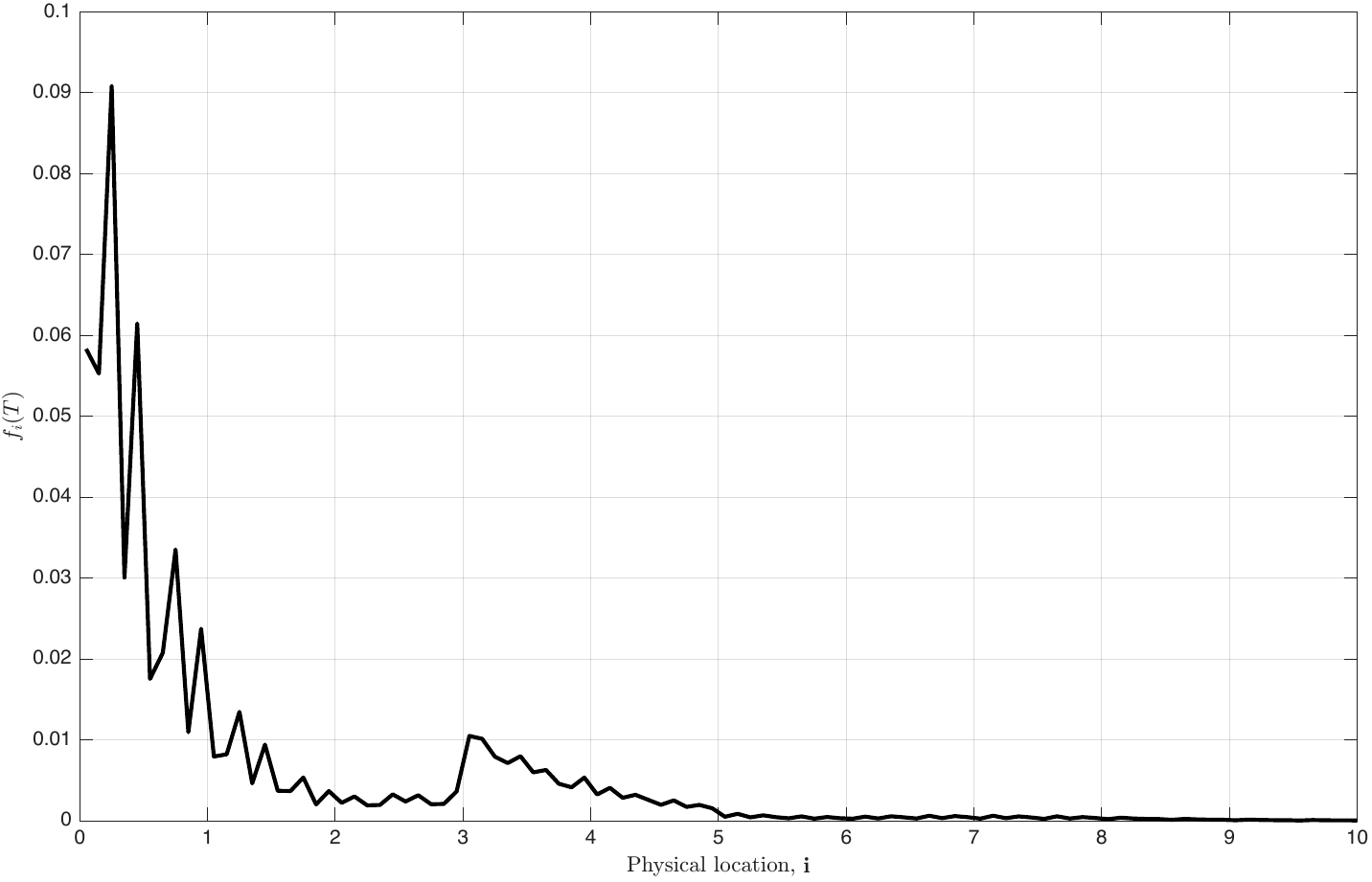}
		\end{minipage}
		\hfill
		\begin{minipage}{0.49\textwidth}
			\centering
			\includegraphics[width=\textwidth]{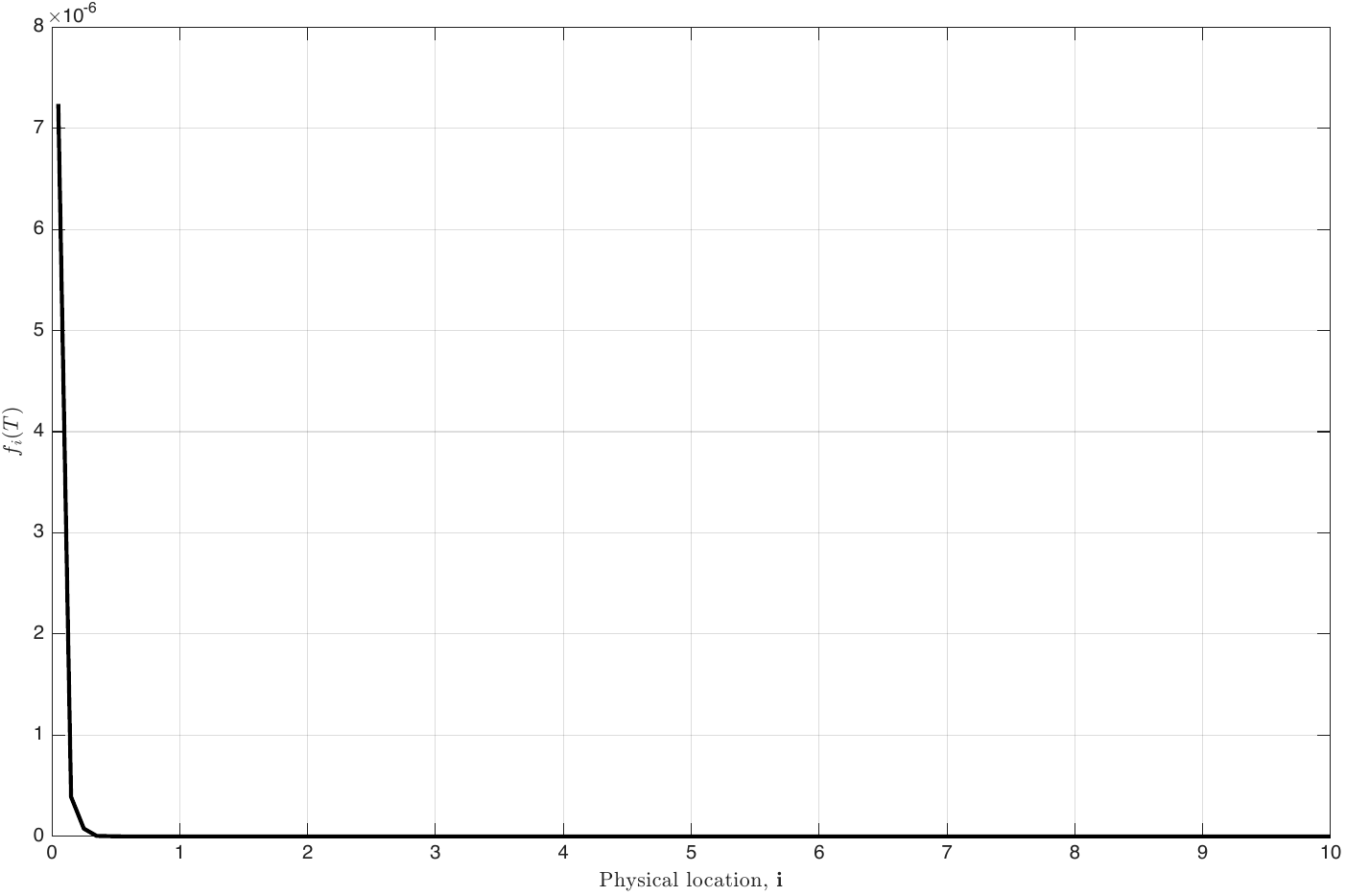}
		\end{minipage}
		\caption{Solution at \(T=1\) (left) and at \(T=10\) (right).}
		\label{Fisol2}
	\end{figure}
	\begin{figure}[H]\label{F2}
		\centering
		\begin{minipage}{0.49\textwidth}
			\centering
			\includegraphics[width=\textwidth]{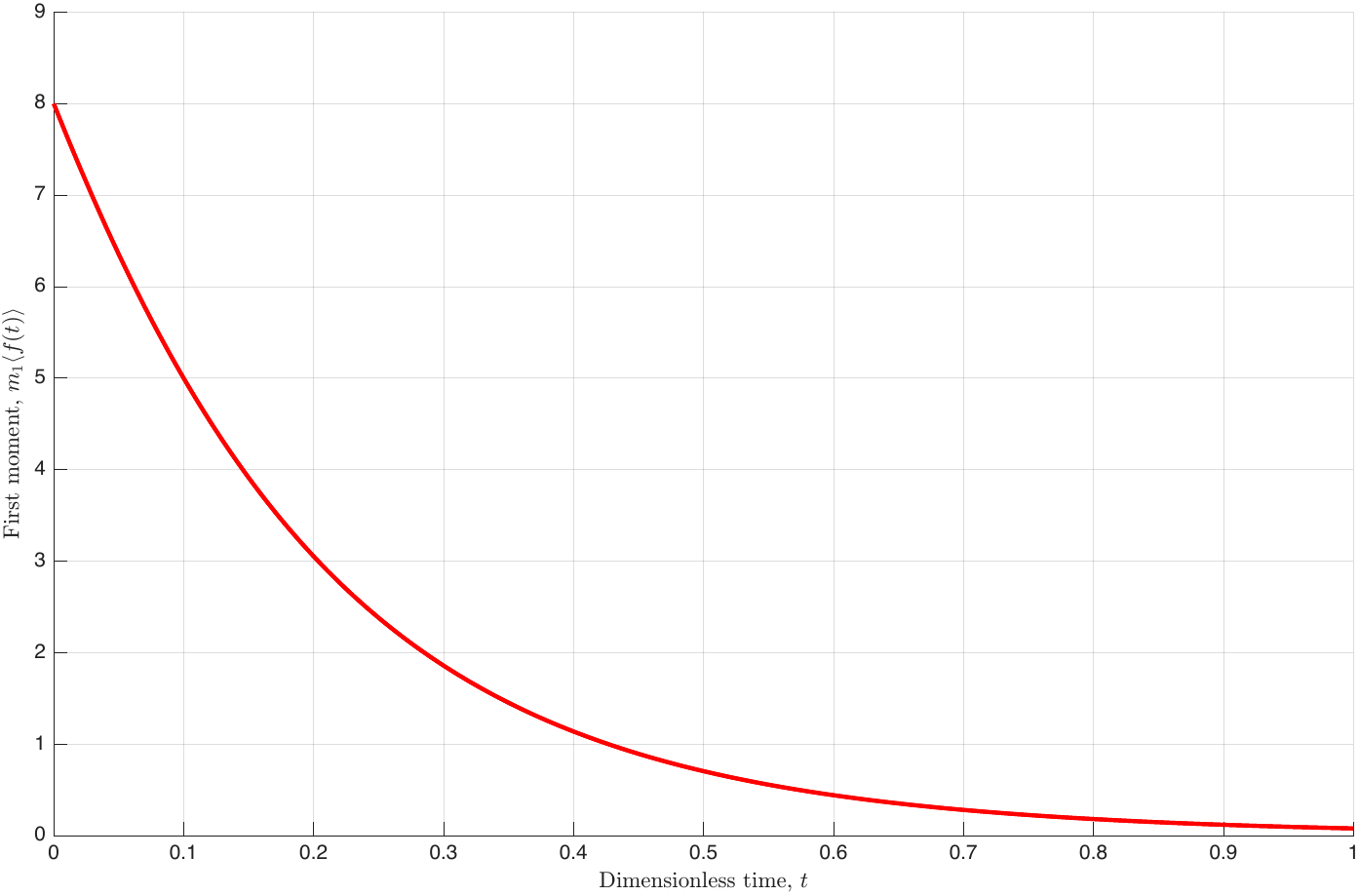}
		\end{minipage}
		\hfill
		\begin{minipage}{0.49\textwidth}
			\centering
			\includegraphics[width=\textwidth]{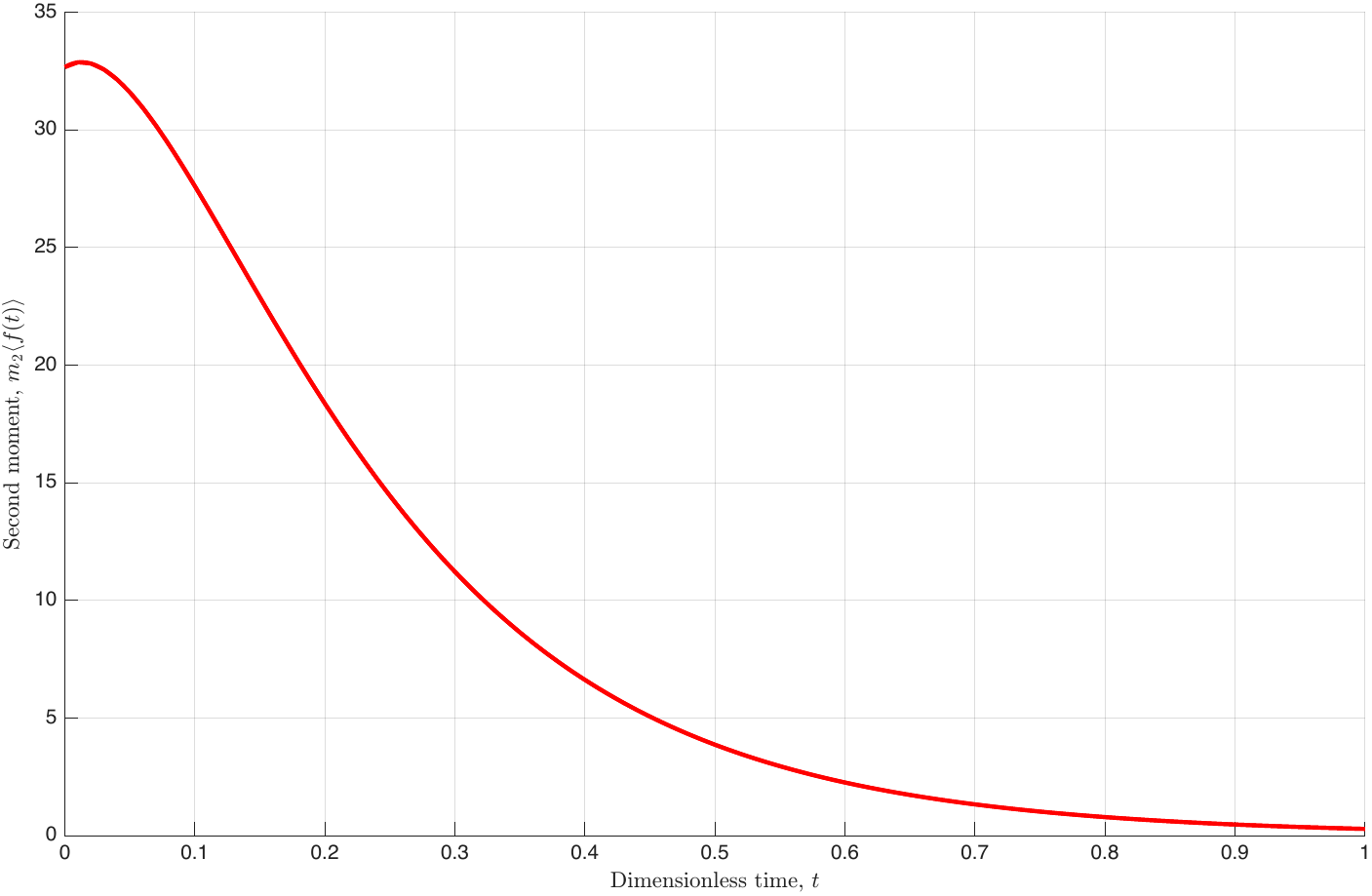}
		\end{minipage}
		\caption{\(m_1\left\langle f(t)\right\rangle\) (left) and \(m_2\left\langle f(t)\right\rangle\) (right) on \([0,1]\).}
		\label{Fim1m22}
	\end{figure}
	\begin{figure}[H]\label{F3}
		\centering
		\begin{minipage}{0.49\textwidth}
			\centering
			\includegraphics[width=\textwidth]{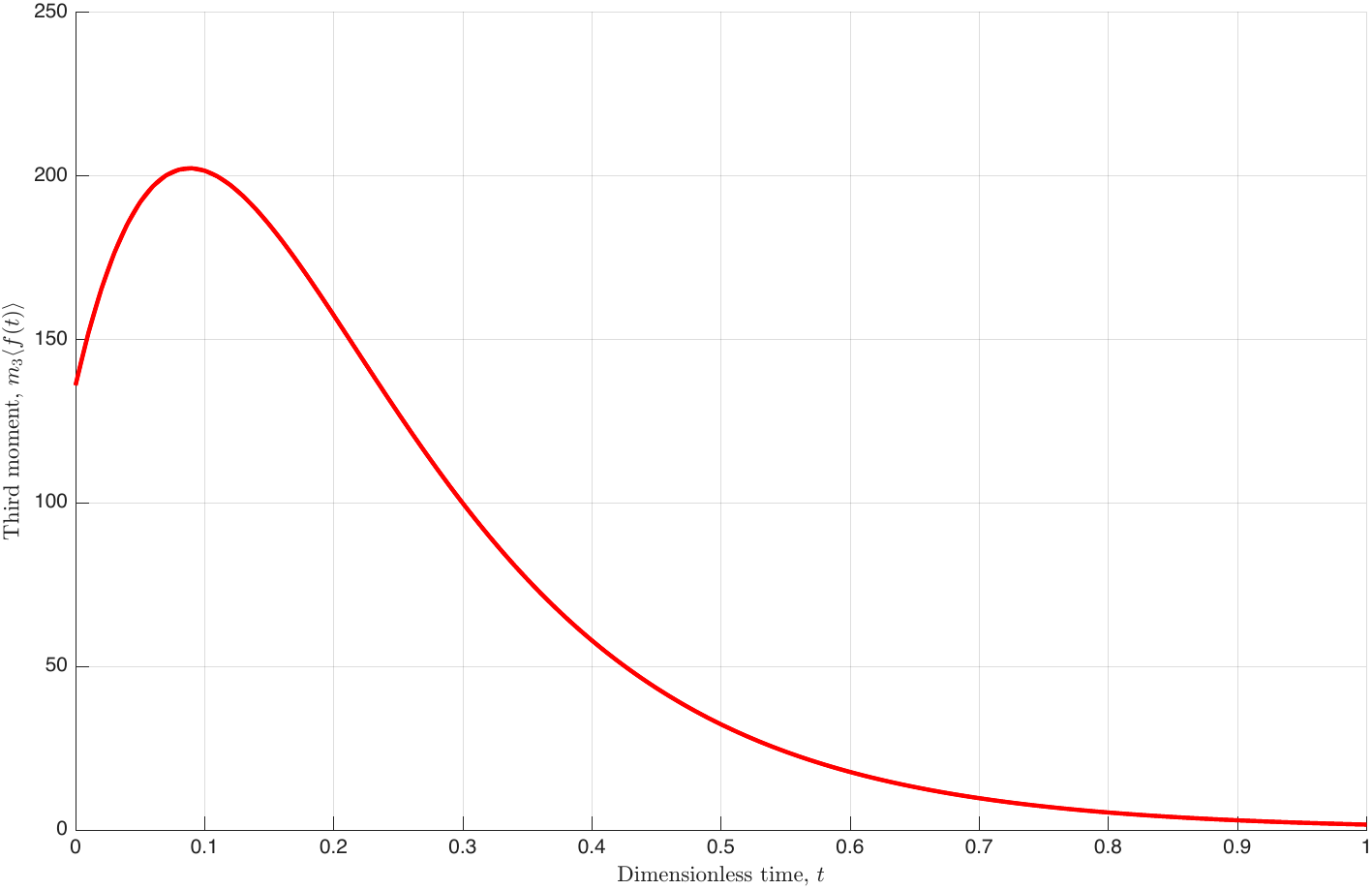}
		\end{minipage}
		\hfill
		\begin{minipage}{0.49\textwidth}
			\centering
			\includegraphics[width=\textwidth]{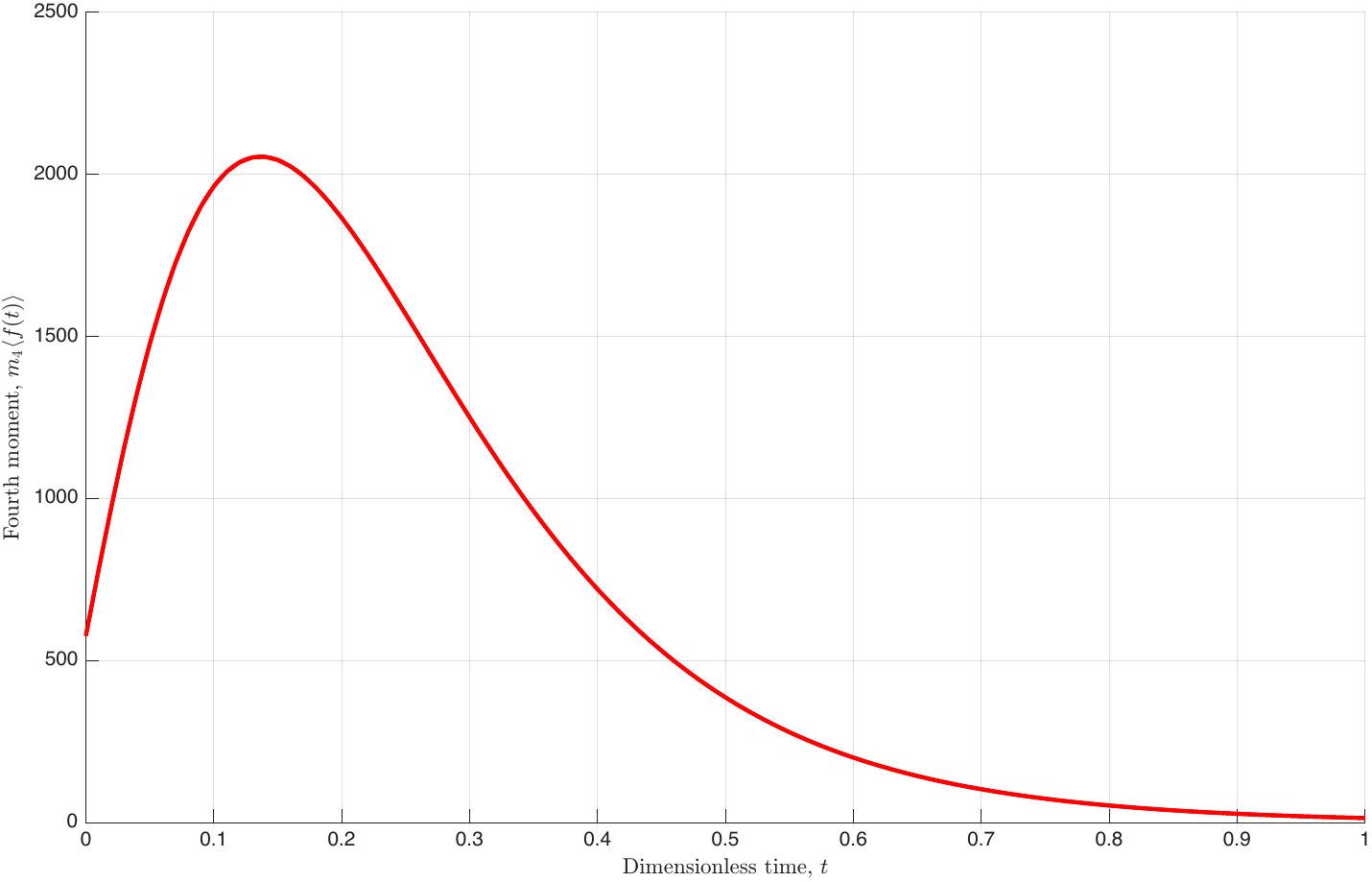}
		\end{minipage}
		\caption{\(m_3\left\langle f(t)\right\rangle\) (left) and \(m_4\left\langle f(t)\right\rangle\) (right) on \([0,1]\).}
		\label{Fim3m42}
	\end{figure}
	\begin{figure}[H]\label{F4}
		\centering
		\begin{minipage}{0.49\textwidth}
			\centering
			\includegraphics[width=\textwidth]{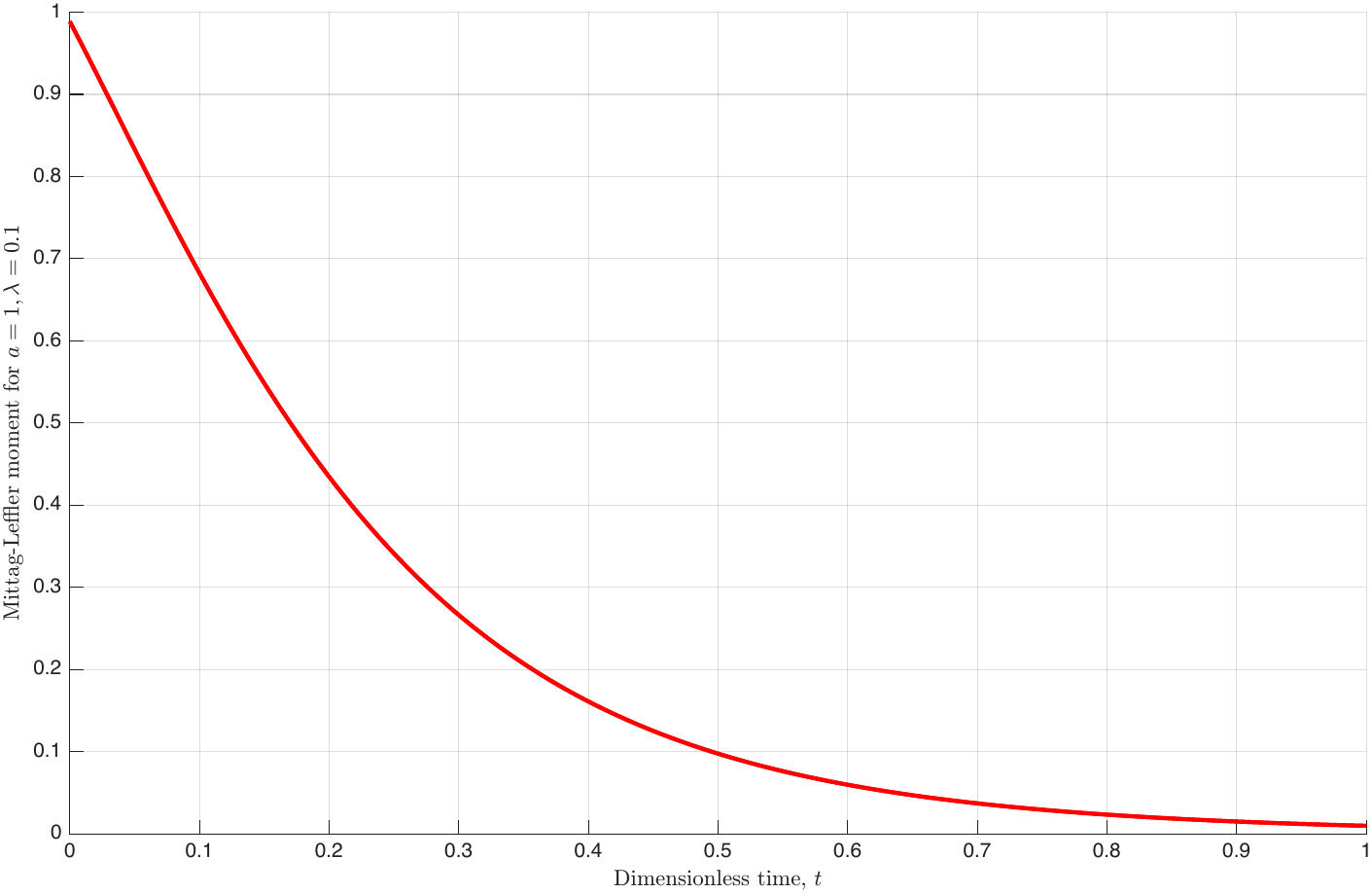}
		\end{minipage}
		\hfill
		\begin{minipage}{0.49\textwidth}
			\centering
			\includegraphics[width=\textwidth]{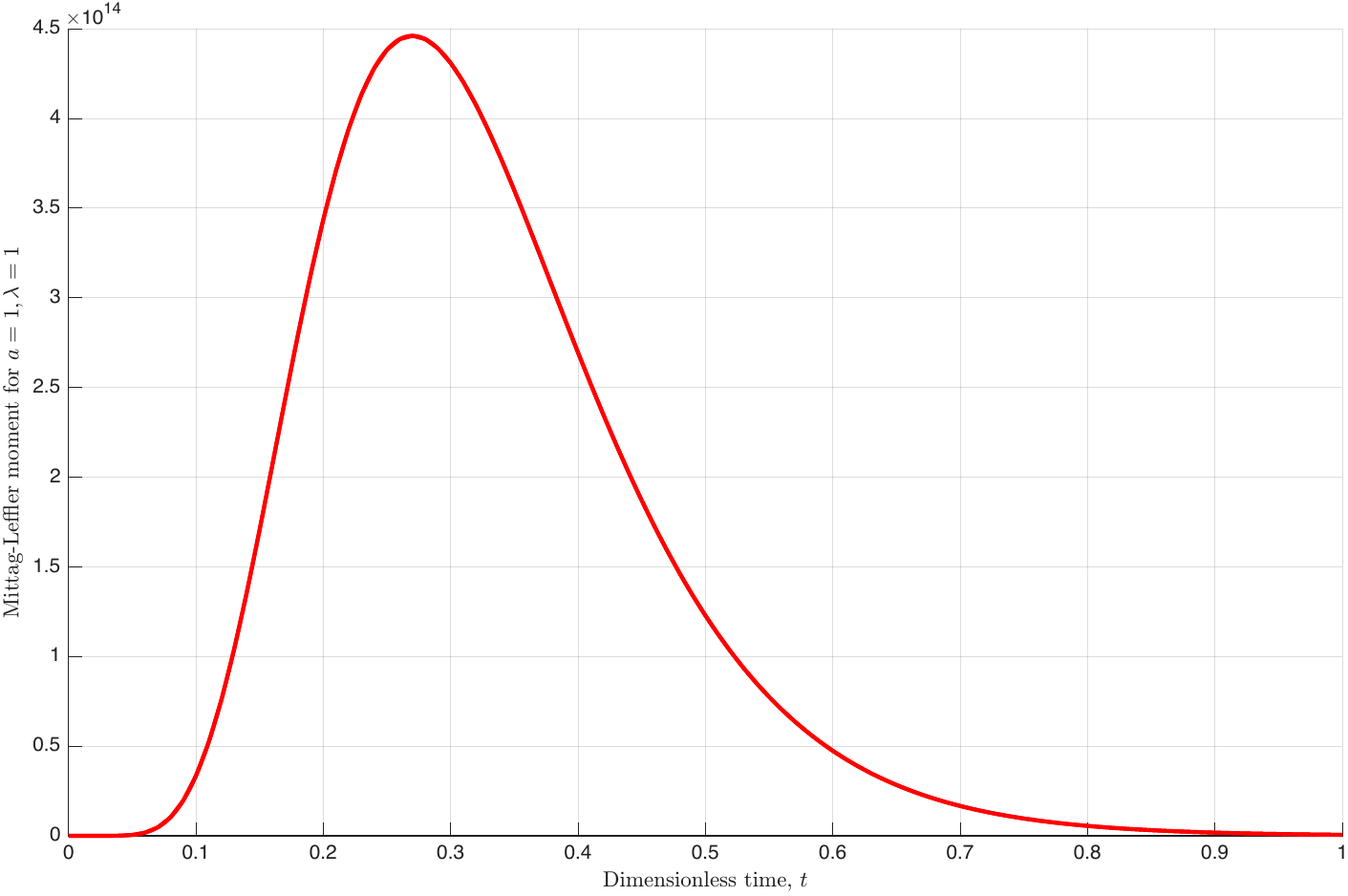}
		\end{minipage}
		\caption{Mittag-Leffler moments \(\mathcal{E}_a^\infty(\lambda)\left\langle f(t)\right\rangle\) for \(a=1,\lambda=0.1\) (left) and for \(a=1,\lambda=1\) (right) on \([0,1]\).}
		\label{Fiml2}
	\end{figure}

	\bibliographystyle{amsplain}
	\bibliography{ref}{}

\def\cprime{$'$} \def\cprime{$'$} \def\cprime{$'$} \def\cprime{$'$}
  \def\cprime{$'$} \def\cprime{$'$} \def\cprime{$'$}
\providecommand{\bysame}{\leavevmode\hbox to3em{\hrulefill}\thinspace}
\providecommand{\MR}{\relax\ifhmode\unskip\space\fi MR }
% \MRhref is called by the amsart/book/proc definition of \MR.
\providecommand{\MRhref}[2]{%
  \href{http://www.ams.org/mathscinet-getitem?mr=#1}{#2}
}
\providecommand{\href}[2]{#2}
\begin{thebibliography}{10}

\bibitem{AlonsoGambaBinh}
R.~Alonso, I.~M. Gamba, and M.-B. Tran, \emph{The {C}auchy problem and {BEC}
  stability for the quantum {B}oltzmann-{G}ross-{P}itaevskii system for bosons
  at very low temperature}, arXiv preprint arXiv:1609.07467 (2016).

\bibitem{banks2025new}
J.~W. Banks and J.~Shatah, \emph{A new approach to direct discretization of
  wave kinetic equations with application to a nonlinear schrodinger system in
  2d}, arXiv preprint arXiv:2509.03432 (2025).

\bibitem{benney1969random}
D.~J. Benney and A.~C. Newell, \emph{Random wave closures}, Studies in Applied
  Mathematics \textbf{48} (1969), no.~1, 29--53.

\bibitem{benney1966nonlinear}
D.~J. Benney and P.~G. Saffman, \emph{Nonlinear interactions of random waves in
  a dispersive medium}, Proc. R. Soc. Lond. A \textbf{289} (1966), no.~1418,
  301--320.

\bibitem{connaughton2010dynamical}
C.~Connaughton and A.~C. Newell, \emph{Dynamical scaling and the
  finite-capacity anomaly in three-wave turbulence}, Physical Review E
  \textbf{81} (2010), no.~3, 036303.

\bibitem{cortes2020system}
E.~Cort{\'e}s and M.~Escobedo, \emph{On a system of equations for the normal
  fluid-condensate interaction in a bose gas}, Journal of Functional Analysis
  \textbf{278} (2020), no.~2, 108315.

\bibitem{CraciunBinh}
G.~Craciun and M.-B. Tran, \emph{A reaction network approach to the convergence
  to equilibrium of quantum {B}oltzmann equations for {B}ose gases}, ESAIM:
  Control, Optimisation and Calculus of Variations (2021).

\bibitem{das2024numerical}
A.~Das and M.-B. Tran, \emph{Numerical schemes for a fully nonlinear
  coagulation--fragmentation model coming from wave kinetic theory},
  Proceedings of the Royal Society A \textbf{481} (2025), no.~2316, 20250197.

\bibitem{deng2019derivation}
Y.~Deng and Z.~Hani, \emph{On the derivation of the wave kinetic equation for
  nls}, arXiv preprint arXiv:1912.09518 (2019).

\bibitem{deng2021propagation}
\bysame, \emph{Derivation of the wave kinetic equation: full range of scaling
  laws}, arXiv preprint arXiv:2110.04565 (2021).

\bibitem{deng2023long}
\bysame, \emph{Long time justification of wave turbulence theory}, arXiv
  preprint arXiv:2311.10082 (2023).

\bibitem{deng2023}
\bysame, \emph{Propagation of chaos and the higher order statistics in the wave
  kinetic theory}, arXiv preprint arXiv:2301.07063 (2023).

\bibitem{deng2021full}
Yu~Deng and Zaher Hani, \emph{Full derivation of the wave kinetic equation},
  arXiv preprint arXiv:2104.11204 (2021).

\bibitem{escobedo2023linearized1}
M.~Escobedo, \emph{On the linearized system of equations for the
  condensate--normal fluid interaction at very low temperature}, Studies in
  Applied Mathematics \textbf{150} (2023), no.~2, 448--456.

\bibitem{escobedo2023linearized}
\bysame, \emph{On the linearized system of equations for the condensate-normal
  fluid interaction near the critical temperature}, Archive for Rational
  Mechanics and Analysis \textbf{247} (2023), no.~5, 92.

\bibitem{EPV}
M.~Escobedo, F.~Pezzotti, and M.~Valle, \emph{Analytical approach to relaxation
  dynamics of condensed {B}ose gases}, Ann. Physics \textbf{326} (2011), no.~4,
  808--827. \MR{2771726 (2012c:82046)}

\bibitem{EscobedoBinh}
M.~Escobedo and M.-B. Tran, \emph{Convergence to equilibrium of a linearized
  quantum {B}oltzmann equation for bosons at very low temperature}, Kinetic and
  Related Models \textbf{8} (2015), no.~3, 493--531.

\bibitem{GambaSmithBinh}
I.~M. Gamba, L.~M. Smith, and M.-B. Tran, \emph{On the wave turbulence theory
  for stratified flows in the ocean}, M3AS: Mathematical Models and Methods in
  Applied Sciences. Vol. 30, No. 1 105-137 (2020).

\bibitem{hasselmann1962non}
K.~Hasselmann, \emph{On the non-linear energy transfer in a gravity-wave
  spectrum part 1. general theory}, Journal of Fluid Mechanics \textbf{12}
  (1962), no.~04, 481--500.

\bibitem{hasselmann1974spectral}
\bysame, \emph{On the spectral dissipation of ocean waves due to white
  capping}, Boundary-Layer Meteorology \textbf{6} (1974), no.~1-2, 107--127.

\bibitem{kim2025wave}
Y.~H. Kim, Y.~V. Lvov, L.~M. Smith, and M.-B. Tran, \emph{On a wave kinetic
  equation with resonance broadening in oceanography and atmospheric sciences},
  arXiv preprint arXiv:2510.25031 (2025).

\bibitem{Nazarenko:2011:WT}
S.~Nazarenko, \emph{Wave turbulence}, Lecture Notes in Physics, vol. 825,
  Springer, Heidelberg, 2011. \MR{3014432}

\bibitem{nguyen2017quantum}
T.~T. Nguyen and M.-B. Tran, \emph{On the {K}inetic {E}quation in {Z}akharov's
  {W}ave {T}urbulence {T}heory for {C}apillary {W}aves}, SIAM J. Math. Anal.
  \textbf{50} (2018), no.~2, 2020--2047. \MR{3784110}

\bibitem{ToanBinh}
\bysame, \emph{Uniform in time lower bound for solutions to a quantum boltzmann
  equation of bosons}, Archive for Rational Mechanics and Analysis \textbf{231}
  (2019), no.~1, 63--89.

\bibitem{Peierls:1993:BRK}
R.~Peierls, \emph{Zur kinetischen theorie der warmeleitung in kristallen},
  Annalen der Physik \textbf{395} (1929), no.~8, 1055--1101.

\bibitem{PomeauBinh}
Y.~Pomeau and M.-B. Tran, \emph{Statistical physics of non equilibrium quantum
  phenomena}, Lecture Notes in Physics, Springer (2019).

\bibitem{rumpf2021wave}
B.~Rumpf, A.~Soffer, and M.-B. Tran, \emph{On the wave turbulence theory:
  ergodicity for the elastic beam wave equation}, Mathematische Zeitschrift
  \textbf{310} (2025), no.~2, 1--41.

\bibitem{soffer2018dynamics}
A.~Soffer and M.-B. Tran, \emph{On the dynamics of finite temperature trapped
  bose gases}, Advances in Mathematics \textbf{325} (2018), 533--607.

\bibitem{staffilani2025evolution}
G.~Staffilani and M.-B. Tran, \emph{Evolution of finite temperature
  bose-einstein condensates: Some rigorous studies on condensate growth}, arXiv
  e-prints (2025), arXiv--2512.

\bibitem{staffilani2025finite}
\bysame, \emph{Finite time energy cascade for mixed $3-$ and $4-$ wave kinetic
  equations}, arXiv preprint arXiv:2512.19531 (2025).

\bibitem{staffilani2025formation}
\bysame, \emph{Formation of condensations for non-radial solutions to 3-wave
  kinetic equations}, arXiv preprint arXiv:2503.17066 (2025).

\bibitem{tran2020reaction}
M.-B. Tran, G.~Craciun, L.~M. Smith, and S.~Boldyrev, \emph{A reaction network
  approach to the theory of acoustic wave turbulence}, Journal of Differential
  Equations \textbf{269} (2020), no.~5, 4332--4352.

\bibitem{walton2023numerical}
S.~Walton and M.-B. Tran, \emph{A numerical scheme for wave turbulence: 3-wave
  kinetic equations}, SIAM Journal on Scientific Computing \textbf{45} (2023),
  no.~4, B467--B492.

\bibitem{walton2024numerical}
\bysame, \emph{Numerical schemes for 3-wave kinetic equations: A complete
  treatment of the collision operator}, Journal of Computational Physics
  (2025), 114147.

\bibitem{walton2022deep}
S.~Walton, M.-B. Tran, and A.~Bensoussan, \emph{A deep learning approximation
  of non-stationary solutions to wave kinetic equations}, Applied Numerical
  Mathematics \textbf{199} (2024), 213--226.

\bibitem{zakharov1967weak}
V.~E. Zakharov and N.~N. Filonenko, \emph{Weak turbulence of capillary waves},
  Journal of applied mechanics and technical physics \textbf{8} (1967), no.~5,
  37--40.

\bibitem{zakharov2012kolmogorov}
V.~E. Zakharov, V.~S. L'vov, and G.~Falkovich, \emph{Kolmogorov spectra of
  turbulence i: Wave turbulence}, Springer Science \& Business Media, 2012.

\end{thebibliography}
	
\end{document}